\documentclass[11pt,reqno,twoside]{amsart}
\usepackage{amsfonts}
\usepackage{amsmath,amssymb}
\usepackage[dvips,bottom=1.4in,right=1in,top=1in, left=1in]{geometry}
\usepackage{epic}
\usepackage{pstricks}
\usepackage{graphicx}

\newtheorem{theorem}{Theorem}[section]
\newtheorem*{theorem*}{Theorem}

\newtheorem{corollary}[theorem]{Corollary}
\newtheorem{definition}[theorem]{Definition}
\newtheorem{example}[theorem]{Example}
\newtheorem{lemma}[theorem]{Lemma}

\newtheorem{proposition}[theorem]{Proposition}
\newtheorem{remark}[theorem]{Remark}
\numberwithin{equation}{section}
\def\RR{{\mathbb{R}}}

\keywords{The Laplace operator, nonlocal Robin boundary conditions
on non-smooth domains, global attractor, exponential attractor,
fractal-like domains, domains with Holder cusps, semilinear
reaction-diffusion equation}
\input{tcilatex}

\begin{document}
\title[Reaction-diffusion equations]{Long-term behavior of
reaction-diffusion equations with nonlocal boundary conditions on rough
domains}
\author{Ciprian G. Gal}
\address{C. G. Gal, Department of Mathematics, Florida International
University, Miami, 33199 (USA)}
\email{cgal@fiu.edu}
\author{Mahamadi Warma}
\address{M.~Warma, University of Puerto Rico, Faculty of Natural Sciences,
Department of Mathematics (Rio Piedras Campus), PO Box 70377 San Juan PR
00936-8377 (USA)}
\email{mjwarma@gmail.com, mahamadi.warma1@upr.edu}

\begin{abstract}
We investigate the long term behavior in terms of finite dimensional global
and exponential attractors, as time goes to infinity, of solutions to a
semilinear reaction-diffusion equation on non-smooth domains subject to
nonlocal Robin boundary conditions, characterized by the presence of
fractional diffusion on the boundary. Our results are of general character
and apply to a large class of irregular domains, including domains whose
boundary is H\"{o}lder continuous and domains which have fractal-like
geometry. In addition to recovering most of the existing results on
existence, regularity, uniqueness, stability, attractor existence, and
dimension, for the well-known reaction-diffusion equation in smooth domains,
the framework we develop also makes possible a number of new results for all
diffusion models in other non-smooth settings.
\end{abstract}

\maketitle
\tableofcontents

\section{Introduction}

\label{intro}

The mathematical theory for global existence and regularity of solutions to
the (scalar) reaction-diffusion equation is considered a central problem in
understanding models of (non-)degenerate reaction-diffusion systems for a
variety of applied problems, especially in chemistry and biology. It is also
essential, for practical applications, to be able to understand, and even
predict, the long time behavior of the solutions of such systems. It is
well-known that the asymptotic behavior of solutions to (scalar)
reaction-diffusion equations can be well described by invariant attracting
sets, and, in particular, by a finite-dimensional global attractor, such
that, the dynamics of these equations, when restricted to these sets, is
effectively described by a finite number of parameters (see, e.g., the
monographs \cite{BV,CV,R,T}).

Analytical results for most PDEs in the literature nowadays revolve around
the most commonly found assumption that the underlying physical space $%
\Omega \subset {\mathbb{R}}^{N}$ ($N\geq 2$) is smooth enough, and that at
best, the boundary of $\Omega $, $\partial \Omega $ is of \emph{Lipschitz}
class. But this is barely non-smooth, since a Lipschitz boundary has a
tangent plane almost everywhere. On the other hand, not much seems to be
known about partial differential equations (except for some scarce results
which we will describe below) and their long-time behavior in general, when
the physical domain $\Omega $ is actually "rough". This is the case, for
instance, of domains whose boundary has either a fractal-like geometry or
domains with cusps which are also frequently used in the applications.
Indeed, it cannot be expected that objects in the real-world, be they are
clouds, trees, snowflakes, blood vessels, etc., will possess the structure
of smooth manifolds \cite{Ma}. One of the main technical difficulties
nowadays of dealing with "bad" domains is the scarcity of Sobolev embedding
theorems and interpolation results in this general context. In fact, for a
general non-smooth domain the usual Sobolev embedding and density theorems
do not hold \cite{MP} (cf. also Section 2).

Our main goal in this paper is to develop well-posedness and long-time
dynamics results for reaction-diffusion equations on domains $\Omega $ with
"rough" boundaries, and then subsequently recover the existing results of
this type for the same models that have been previously obtained in the case
of domains with smooth boundary $\partial \Omega $. Along these lines, we
first establish a number of results for scalar reaction-diffusion equations,
including results on existence, regularity, uniqueness of weak and strong
solutions, existence and finite dimensionality of global and exponential
attractors, and existence of Lyapunov functions. To be more precise, we
shall be concerned with diffusion processes in "rough" domains $\Omega $,
described by the equation%
\begin{equation}
\partial _{t}u-\Delta u+f\left( u\right) =0\text{ in }\Omega \times \left(
0,\infty \right) ,  \label{1.1}
\end{equation}%
subject to the following nonlocal Robin boundary condition%
\begin{equation}
\partial _{\nu }ud\sigma +(u+\Theta _{\mu }\left( u\right) )d\mu =0\text{ on
}\partial \Omega \times \left( 0,\infty \right) ,  \label{1.2}
\end{equation}%
and the initial condition
\begin{equation}
u\left( 0\right) =u_{0}\text{ in }\Omega .  \label{1.3}
\end{equation}%
In Eqn. (\ref{1.1}), $f=f\left( u\right) $ plays the role of nonlinear
source, not necessarily monotone, and $\Theta _{\mu }\left( u\right) $ is a
certain nonlocal operator characterizing the presence of "fractional"
diffusion along $\partial \Omega $ (see Eq. \eqref{nonlocal-op} below). The
normal derivative $\partial _{\nu }u$ is understood in the sense of (\ref{NE}%
) specified below, $\sigma $ denotes the restriction to $\partial \Omega $
of the $(N-1)$-dimensional Hausdorff measure $\mathcal{H}^{N-1}$, $\mu $ is
an appropriate positive regular Borel measure on $\partial \Omega $. In
fact, \emph{the regularity assumptions we will impose on} $\partial \Omega $
\emph{enter through the measure} $\mu $ in (\ref{1.2}). We will make this
more precise in Section \ref{prelim}. Since for a "rough" domain $\Omega $
the boundary $\partial \Omega $ may be so irregular that no normal vector
can be defined, we will use the following generalized version of a normal
derivative in the weak sense introduced in \cite{BW2}. Let $\mu $ be again a
Borel measure on $\partial \Omega $ and let $F:\Omega \rightarrow {\mathbb{R}%
}^{N}$ be a measurable function. If there exists a function $f\in
L_{loc}^{1}({\mathbb{R}}^{N})$ such that
\begin{equation}
\int_{\Omega }F\cdot \nabla \varphi dx=\int_{\Omega }f\varphi
dx+\int_{\partial \Omega }\varphi d\mu  \label{NE}
\end{equation}%
for all $\varphi \in C^{1}(\overline{\Omega })$, then we say that $\mu $ is
the \emph{normal measure} of $F$ which we denoted by $N^{\star }(F):=\mu $.
If $N^{\star }(F)$ exists, then it is unique and $dN^{\star }(\psi F)=\psi
dN^{\star }(F)$ for all $\psi \in C^{1}(\overline{\Omega })$. If $u\in
W_{loc}^{1,1}(\Omega )$ and $N^{\star }(\nabla u)$ exists, then we will
denote by $N(u):=N^{\star }(\nabla u)$ the generalized normal measure of $%
\nabla u$. The derivative $dN(u)/d\sigma $ that we denote by $\partial _{\nu
}u$ will be called the generalized normal derivative of $u$. To justify this
definition, consider the special case of a bounded domain $\Omega \subset {%
\mathbb{R}}^{N}$ whose boundary is Lipschitz continuous, $\nu $ the outer
normal to $\partial \Omega $ and let $\sigma $ be the (natural) surface
measure on $\partial \Omega $ (in this case, $\sigma $ also coincides with $%
\mathcal{H}_{\mid \partial \Omega }^{N-1}$). If $u\in C^{1}(\overline{\Omega
})$ is such that there are $f\in L_{loc}^{1}({\mathbb{R}}^{N})$ and $g\in
L^{1}(\partial \Omega ,\sigma )$ with
\begin{equation*}
\int_{\Omega }\nabla u\cdot \nabla \varphi dx=\int_{\Omega }f\varphi
dx+\int_{\partial \Omega }g\varphi d\sigma
\end{equation*}%
for all $\varphi \in C^{1}(\overline{\Omega })$, then $dN(u)=gd\sigma $ with
$g=\partial _{\nu }u=\partial u/\partial \nu $. Throughout the following,
without any mention, for a bounded arbitrary domain $\Omega ,$ we will
always mean the identity (\ref{NE}) for the generalized outer normal
derivative of $u$.

The interest in analysis and modelling of diffusion processes in bounded
domains whose (part of the)\ boundary possess a fractal geometry arises from
mathematical physics, and dates back to the early 1980's. The first
analytical results aimed at understanding transmission problems, which, in
electrostatics and magnetostatics, describe heat transfer through a
fractal-like interface (such as, the snowflake), can be found in \cite{La1,
La2, La3, La4}. The type of parabolic problems we consider also occur in the
field of the so-called \textquotedblleft hydraulic
fracturing\textquotedblright , a frequently used engineering method to
increase the flow of oil from a reservoir into a producing oil well (see
\cite{Ca}; cf. also \cite{HP} for a related application). Further examples
are also provided in the book of Dautray and Lions \cite{DL}. In all these
applications, the mathematical model is usually a linear parabolic boundary
value problem involving a transmission condition on a fractal-like interface
(layer) which is often a Robin boundary condition. The reaction-diffusion
equation (\ref{1.1}) on unbounded fractal domains has also been considered
in \cite{Fa}, and in \cite{Hu}, for bounded fractal domains for which the
usual Sobolev-type inequalities hold and for which (\ref{1.1}) is equipped
with homogeneous Dirichlet boundary conditions on $\partial \Omega $. The
latter contributions devote their attention mainly to some existence results
for some special cases of nonlinearities. The motivation to consider (\ref%
{1.1})-(\ref{1.3}) is also inspired by a wider and challenging problems
aimed at simulating the diffusion of e.g. medical sprays in the bronchial
tree \cite{MS1, MS2}. In this case, the geometry of the underlying physical
domain can be simulated by some classes of self-similar ramified domains
with a fractal boundary. Oxygen diffusion between the lungs and the
circulatory system takes place only in the last generations of the lung
tree, so that a reasonable diffusion model may need to involve nonlocal
Robin boundary conditions (\ref{1.2}) on the top boundary (the smallest
structures), see Section \ref{prelim} (cf. also \cite{AT, ATbis, ATtris}).

It would be extremely useful if one could give a unified analysis of the
reaction-diffusion problem (\ref{1.1})-(\ref{1.3}) for a large class of
rough domains, including the specified families of "fractal" domains and/or
domains with cusps, using only a minimal number of regularity properties for
$\Omega $, and then use these assumptions about the specific form of the
domain, leading to specific models, to derive sharp results about existence,
regularity and stability of solutions. Then, it is also essential, for
further practical applications, to show the existence of the global
attractor for our general model, and then to determine whether the dynamics
restricted to this global attractor is finite-dimensional or not. Among the
first important contributions made to understand the linear problem
associated with Eqns. (\ref{1.1})-(\ref{1.2}), in general bounded open sets $%
\Omega $ with no \emph{essential} regularity assumptions on $\partial \Omega
,$ can be found in \cite{W3} (cf. also \cite{VW} for related results). In
particular, in \cite{W3} it is shown that the unique solution of the linear
problem is given in terms of a strongly continuous (linear) semigroup of
contraction operators on $L^{2}\left( \Omega \right) ,$ that is order
preserving, nonexpansive on $L^{\infty }\left( \Omega \right) $, and
ultracontractive (see Section 2; cf. also \cite[Sections 3-5]{W3}). The
first tool used to derive this result is the validity of the following
inequality, for arbitrary open sets $\Omega $ with finite measure,%
\begin{equation}
\Vert u\Vert _{L^{\frac{2N}{N-1}}\left( \Omega \right) }\leq C_{\Omega
}\left( \Vert \nabla u\Vert _{L^{2}\left( \Omega \right) }+\Vert u\Vert
_{L^{2}(\partial \Omega ,\sigma )}\right) ,  \label{MaI}
\end{equation}%
which holds for any $u\in W^{1,2}(\Omega )\cap C_{c}(\overline{\Omega })$,
for some $C_{\Omega }=C(N,\Omega )>0$. This crucial inequality is due to
Maz'ya \cite{MP}. We recall that for Lipschitz domains, the optimal exponent
on the left-hand side of (\ref{MaI}) is $2N/\left( N-2\right) ,$ while for
\emph{arbitrary} open domains $\Omega ,$ the best optimal exponent is $%
2N/\left( N-1\right) $, see Section 2.2. The second tool is the notion of
relative capacity\ with respect to $\Omega $ \cite{AW1, AW2}, which is a
fundamental tool both in classical analysis and potential theory. Its most
common property is that it measures small sets more precisely than the usual
Lebesgue measure. With both these tools at disposal, the local Robin
boundary condition (that is, $\Theta _{\mu }\equiv 0$ in Eqn. (\ref{1.2}))
for the linear heat equation has been investigated in \cite{BW1, BW2} (and
the references therein) under the \emph{stronger} restriction that $\Omega $
possesses the extension property of Sobolev functions. The (linear) elliptic
system associated with equations (\ref{1.1})-(\ref{1.2}) has also been
considered in \cite{AW1, AW2, D1, D2, DD}, also without any essential
regularity assumptions on $\Omega $. These latter references are mainly
concerned with the existence of weak solutions for these elliptic systems
and several apriori estimates.

In addition to deriving well-posedness and regularity results for our
nonlinear model, perhaps the study of the asymptotic behavior is equally as
important as it is essential to be able to understand, and even predict, the
long time behavior of the solutions of our system. One object well-suited to
describe the long-time behavior is the global attractor. Assuming $f\in C_{%
\text{loc}}^{1}\left( \mathbb{R}\right) $ is a function which satisfies, for
appropriate positive constants,%
\begin{equation}
C_{f}\left\vert s\right\vert ^{p}-c_{f}\leq f\left( s\right) s\leq
\widetilde{C}_{f}\left\vert s\right\vert ^{p}+\widetilde{c}_{f},\text{ }p>1
\label{star1}
\end{equation}%
and%
\begin{equation}
f^{^{\prime }}\left( s\right) \geq -C_{f},\text{ for all }s\in \mathbb{R}%
\text{,}  \label{star2}
\end{equation}%
resembling the same classical assumptions (see, e.g., \cite{CV, R}), we
prove in Section 3.2 that (\ref{1.1})-(\ref{1.3}) generates a nonlinear
continuous semigroup $\left\{ \mathcal{S}\left( t\right) \right\} _{t\geq 0}$
acting on the phase space $L^{2}\left( \Omega \right) $ (see Corollary \ref%
{dyn_system}). This semigroup possesses the global attractor $\mathcal{G}%
_{\Theta ,\mu }$, which is compact in $L^{2}\left( \Omega \right) $, bounded
in $L^{\infty }\left( \Omega \right) $ and has finite fractal dimension (see
Corollary \ref{gl_fin}). In addition, we show that every unique orbit $%
u\left( t\right) =\mathcal{S}\left( t\right) u_{0},$ $u_{0}\in L^{2}\left(
\Omega \right) $, of the parabolic problem (\ref{1.1})-(\ref{1.3})
"instantaneously" exhibits an improved regularity both in space and time.
More precisely, every such weak solution becomes a strong solution (see
Section 3.3); hence, $\mathcal{G}_{\Theta ,\mu }$ contains only strong
solutions. In fact, the aim of the whole Section 3.1 is to verify the
existence of (unique) strong solutions which are smooth enough (even with
more general assumptions than (\ref{star1})-(\ref{star2})). One of the
advantages of this result is that every weak solution can be approximated by
the strong ones and the rigorous justification of the regularity results for
such solutions becomes immediate. In particular, exploiting Maz'ya
inequality (\ref{MaI}) we extend the scheme developed initially by Alikakos
\cite{Ali} for reaction-diffusion equations, to derive that a particular
smoothing property holds for all weak solutions of (\ref{1.1})-(\ref{1.3})
with a function $f$ satisfying (\ref{star1})-(\ref{star2}). Here we
emphasize again that the original proofs due to Alikakos \cite{Ali} rely
heavily on the application of the Gagliardo-Nirenberg-Sobolev inequality
which unfortunately is no longer valid in our context. Besides, in our
iteration scheme we need to get good control of certain constants in such a
way that they do not depend on (the bad behavior of) $\partial \Omega $.
Furthermore, a key point of this analysis is to employ an appropriate fixed
point argument, coupled together with a hidden regularity theorem (Appendix,
Theorem \ref{ap_reg_thm}) for a non-autonomous equation governed by an
accretive operator, to deduce smooth strong solutions which are
differentiable almost everywhere on $\left( 0,\infty \right) $. This
regularity is crucial to show, for instance, that problem (\ref{1.1})-(\ref%
{1.3}) has a Lyapunov function on the global attractor $\mathcal{G}_{\Theta
,\mu }$ (see Lemma \ref{L2} and Theorem \ref{Lf}). Clearly, $\mathcal{G}%
_{\Theta ,\mu }$ depends on the choice of boundary conditions (\ref{1.2})
and the measure $\mu $ on $\partial \Omega $ (see, for instance, Theorems
(A) and (B) below). At this point one could argue that the long-time
behavior of system (\ref{1.1})-(\ref{1.3}) is properly described by the
global attractor. However, it is well-known that the global attractor can
present several drawbacks, among which we can mention that it may only
attract the trajectories at a slow rate, and that the rate of attraction is
very difficult, if not impossible, to express in terms of the physical
observable quantities (see, e.g., \cite{MZrev}). Furthermore, in many
situations, the global attractor may not be even observable in experiments
or in numerical simulations. This can be seen, for instance, for the
one-dimensional Chaffee-Infante equation%
\begin{equation*}
\partial _{t}u-d\partial _{x}^{2}u+f\left( u\right) =0
\end{equation*}%
on the interval $(0,1)$, with cubic nonlinearity $f\left( s\right) =s^{3}-s,$
and non-homogeneous Dirichlet boundary conditions (i.e., $u\left( 0,t\right)
=u\left( 1,t\right) =-1$, $t>0$), in which case every trajectory is
exponentially attracted to the "single point" attractor $\left\{ -1\right\} $%
. On the other hand, this problem possesses many interesting metastable
\textquotedblleft almost stationary\textquotedblright\ equilibria which live
up to a time $t_{\ast }\sim e^{d^{-1/2}}$ and, thus, for $d>0$ small, one
will not see the global attractor in numerical simulations. This is known to
happen, for instance, for some models of one-dimensional Burgers equations
and models of pattern formation in chemotaxis (see \cite{MZrev}, for further
references). Henceforth, in some situations, the global attractor may fail
to capture important transient behaviors. Besides, in the general setting of
arbitrary open domains, this feature can be further amplified for our system
due to the boundary condition (\ref{1.2}) and the "rough" nature of the
domain $\Omega $ near $\partial \Omega $. It is thus also important to
construct larger objects which contain the global attractor, attract the
trajectories at a fast (typically, exponential) rate which can be computed
explicitly in terms of the physical parameters, and are still finite
dimensional. A natural object is the exponential attractor (see, e.g., \cite[%
Section 3]{MZrev}; cf. also below). In Section \ref{gl}, we prove the
existence of such an exponential attractor not only for the dynamical
systems generated by the weak solutions (see Theorem \ref{expo}) but also by
the strong solutions, which require essentially no growth assumptions on the
nonlinearity as $\left\vert s\right\vert \rightarrow \infty $ (see Theorem %
\ref{expo2}). Roughly speaking, the assumption on $f$ reads:%
\begin{equation}
\liminf_{\left\vert s\right\vert \rightarrow +\infty }\frac{f\left( s\right)
}{s}>-\lambda _{\ast },  \label{star3}
\end{equation}%
for some $\lambda _{\ast }\in \lbrack 0,C_{\Omega })$, where $C_{\Omega
}=C\left( N,\Omega \right) >0$ is the best Sobolev/Poincar\'{e} constant
into (\ref{MaI}). The latter result seems to be new for (\ref{1.1})-(\ref%
{1.3}) even when $\Omega $ is a smooth domain. We refer the reader to
Section \ref{wel} for the precise assumptions, related results and
generalizations.

We emphasize that the measure $\mu $\ and/or boundary regularity assumptions
we employ are of general character, and as a result do not require any
specific form of the domain $\Omega $; this abstraction allows (\ref{1.1})-(%
\ref{1.3}) to recover all of the existing diffusion models that have been
previously studied in smooth bounded domains (including Lipschitz domains),
as well as to represent a much larger family of models for (\ref{1.1})-(\ref%
{1.3}) that have not been explicitly studied in detail. As a result, the
system in (\ref{1.1})-(\ref{1.3}) includes reaction-diffusion models, on
domains $\Omega $ which possess the $W^{1,2}$-extension property of Sobolev
functions, and non-Lipschitz domains whose boundary is only H\"{o}lder
continuous, as special cases, and on many arbitrary open bounded domains $%
\Omega $ (satisfying, for instance, (\ref{MaI})) not previously identified.
In Section \ref{sum}, we discuss how the unified analysis presented here can
be used to establish the same results for other important classes of partial
differential equations, such as, reaction-diffusion systems for a
vector-valued function $\overrightarrow{u}=\left( u_{1},...,u_{k}\right) .$

The following theorems can be treated as \emph{special} cases of our
results. They apply for instance to domains $\Omega $\ with a "fractal"
boundary $\partial \Omega .$

\begin{theorem*}[A]
\label{ThA}Assume that $\Omega $ has the $W^{1,2}$-extension property of
Sobolev functions and $f$ satisfies (\ref{star1})-(\ref{star2}). Let $\mu $
be the restriction to $\partial \Omega $ of the $d$-dimensional Hausdorff
measure $\mathcal{H}^{d},$ for any $d\in \left( N-2,N\right) $. Then, the
dynamical system $\mathcal{S}\left( t\right) :L^{2}\left( \Omega \right)
\rightarrow L^{2}\left( \Omega \right) $, $u_{0}\mapsto u\left( t\right) =%
\mathcal{S}\left( t\right) u_{0},$ associated with weak solutions for the
parabolic problem (\ref{1.1})-(\ref{1.3}) is gradient-like and possesses the
global attractor $\mathcal{G}_{\Theta ,\mu }$ of finite fractal dimension.
Moreover, the semigroup $\mathcal{S}\left( t\right) $ also has an
exponential attractor $\mathcal{E}_{\Theta ,\mu }.$
\end{theorem*}

\begin{theorem*}[B]
\label{ThB}Assume that $\Omega $ has the $W^{1,2}$-extension property of
Sobolev functions and $f$ obeys (\ref{star3}). Let $\mu $ be the restriction
to $\partial \Omega $ of the $d$-dimensional Hausdorff measure $\mathcal{H}%
^{d},$ for any $d\in \left( N-2,N\right) $. Then, the dynamical system $%
\mathcal{T}\left( t\right) :L^{\infty }\left( \Omega \right) \rightarrow
L^{\infty }\left( \Omega \right) $, $u_{0}\mapsto u\left( t\right) =\mathcal{%
T}\left( t\right) u_{0},$ associated with "strong" solutions of the
parabolic problem (\ref{1.1})-(\ref{1.3}) has an exponential attractor $%
\mathcal{Y}_{\Theta ,\mu }$ (hence, also a global attractor of finite
fractal dimension).
\end{theorem*}

The assumption on the measure $\mu $ deserves some additional comments.
First, we note that when $d=N-1$ but $\mu =\mathcal{H}_{\mid \partial \Omega
}^{N-1}\left( =\sigma \right) $ is locally infinite, that is, for all $x\in
\partial \Omega $ and $\delta >0$ we have that
\begin{equation}
\sigma \left( B_{\delta }\left( x\right) \cap \partial \Omega \right)
=\infty ,  \label{loc_inf}
\end{equation}%
in this case the boundary value problem (\ref{1.1})-(\ref{1.2}) coincides
with the (homogeneous) Dirichlet problem for (\ref{1.1}). If (\ref{loc_inf})
only holds on part of the boundary $\Gamma _{\infty }\subset \partial \Omega
$, Dirichlet boundary conditions are satisfied on that part $\Gamma _{\infty
}$ and the usual (nonlocal) Robin boundary condition (\ref{1.2}) is
satisfied on the remaining part $\partial \Omega \backslash \Gamma _{\infty
} $. On the other hand, for domains with fractal boundaries $\partial \Omega
$, one often finds some $d_{0}\in \left( N-2,N\right) $ such that%
\begin{equation*}
\mu _{d_{0}}=\mathcal{H}_{\mid \partial \Omega }^{d_{0}}\text{ and }\mu
_{d_{0}}\left( \partial \Omega \right) <\infty .
\end{equation*}%
In this case, one may think of (\ref{1.2}) as a generalized (nonlocal) Robin
boundary condition ($\mu =\mu _{d_{0}}$) for the boundary value problem (\ref%
{1.1})-(\ref{1.3}). We once again refer the reader to Section \ref{gl} for
more precise statements of the results on the existence of finite
dimensional global and exponential attractors.

The remainder of the paper is structured as follows. In Section \ref{prelim}%
, we establish our notations and give some basic preliminary results for the
operators and spaces appearing in the model (\ref{1.1})-(\ref{1.3}). In
Section \ref{wel}, we prove some well-posedness results for this model; in
particular, we establish existence results for strong solutions (Section
3.1), weak solutions (Section 3.2) and then derive regularity (Section 3.3)
and stability results (Section 3.2). In Section \ref{gl}, we prove results
which establish the existence of global and exponential attractors for (\ref%
{1.1})-(\ref{1.3}), and the existence of Lyapunov functions. Section \ref%
{sum} contains some concluding remarks. To make the paper reasonably
self-contained, in Appendix, we develop some supporting material on
regularity results for abstract non-homogeneous evolution equations which
are necessary to derive the results in Section 3.1.

\section{Preliminaries}

\label{prelim}

\subsection{Some facts from measure theory}

As mentioned in the introduction, we aim to analyze the long-time behavior
of solutions to the reaction-diffusion equation (\ref{1.1}) (supplemented
with the boundary and initial conditions (\ref{1.2})-(\ref{1.3})) in terms
of global and exponential attractors \cite{MZrev}. In order to make the
paper as self-contained as possible, in this section we recall the main
definitions and results in measure and Sobolev function theory which will be
extensively used in what follows. \newline

Let $\Omega \subset {\mathbb{R}}^{N}$ be an open set with boundary $\partial
\Omega $. We say that a Borel measure $\mu $ on $\partial \Omega $ is
regular if for every $A\subset \partial \Omega $ there exists a Borel set $B$
such that $A\subset B$ and $\mu (A)=\mu (B)$. A measure $\mu $ on $\partial
\Omega $ is a Radon measure if $\mu $ is Borel regular and $\mu (K)<\infty $
for each compact set $K\subset \partial \Omega $. Let $\sigma $ be the
restriction to $\partial \Omega $ of the $(N-1)$-dimensional Hausdorff
measure $\mathcal{H}^{N-1}$. Then it is well-known that $\sigma $ is a
regular Borel measure on $\partial \Omega $ but it is not always a Radon
measure, that is, if $\Omega $ has a \textquotedblleft
bad\textquotedblright\ boundary $\partial \Omega $, it may happen that
compact subsets of $\partial \Omega $ have infinite $\sigma $-measure (see
\cite{EG,Fal}).

\begin{definition}
\label{twosets}For a regular Borel measure $\mu $ on $\partial \Omega $, we
denote by
\begin{equation}
\Gamma _{\mu }^{\infty }:=\{z\in \partial \Omega :\;\mu (B(z,r)\cap \partial
\Omega )=\infty \;\forall \;r>0\}  \label{GI}
\end{equation}%
the relatively closed subset of $\partial \Omega $ on which the measure $\mu
$ is \emph{locally infinite}. The complement of $\Gamma _{\mu }^{\infty }$
denoted by
\begin{equation}
\Gamma _{\mu }:=\partial \Omega \setminus \Gamma _{\mu }^{\infty }=\{z\in
\partial \Omega :\;\exists \;r>0:\;\mu (B(z,r)\cap \partial \Omega )<\infty
\}  \label{LF}
\end{equation}%
is the relatively open subset of $\partial \Omega $ on which the measure $%
\mu $ is \emph{locally finite}. We will also call the part $\Gamma _{\mu }$
the domain of the measure $\mu $ and denote $D(\mu )=\Gamma _{\mu }$.
\end{definition}

We notice that the restriction of $\mu $ to $\Gamma _{\mu }$ is a Radon
measure and $\mu (\Gamma _{\mu })<\infty $ if $\Omega $ is such that its
boundary is a compact set. For more properties of Borel measures and the
Hausdorff measure, we refer the reader to the monographs \cite{EG, Fal}.

\begin{example}
\emph{It is clear that by definition, if }$\mu (\partial \Omega )<\infty $%
\emph{, then }$D(\mu )=\partial \Omega $\emph{\ so that }$\Gamma _{\mu
}^{\infty }=\emptyset $\emph{. }

\emph{Next, let $0<a_{n+1}<b_{n+1}<a_{n}<1$ ($n\in {\mathbb{N}}$) be such
that $\lim_{n\rightarrow \infty }a_{n}=0$. Let
\begin{equation*}
\Omega :=\left\{ (x,y)\in (0,1)\times (0,1)\setminus \bigcup_{n\in {\mathbb{N%
}}}[a_{n},b_{n}]\times \left[ \frac{1}{2},1\right] \right\} \subset \mathbb{R%
}^{2}.
\end{equation*}%
Let $\sigma :=\mathcal{H}^{1}|_{\partial \Omega }$. It is easy to see that }$%
\sigma (\partial \Omega )$\emph{$=\infty $ and that }$\Gamma _{\sigma
}^{\infty }=\{0\}\times \lbrack 1/2,1]$\emph{. }
\end{example}

\begin{figure}[tbp]
\centering\includegraphics[scale=0.44]{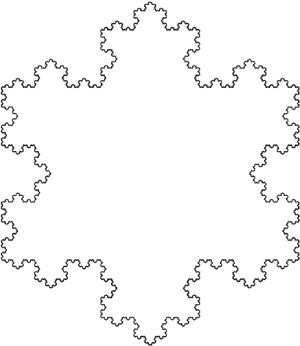}
\caption{The Koch snowflake}
\end{figure}

\begin{example}
\label{snow}\emph{Let $\Omega \subset {\mathbb{R}}^{2}$ be the Koch
snowflake (see Figure 1) and let }$\sigma :=\mathcal{H}^{1}|_{\partial
\Omega }$\emph{. Then }$\Gamma _{\sigma }^{\infty }=\partial \Omega $\emph{\
so that }$\Gamma _{\sigma }=D(\sigma )=\emptyset $\emph{. It is well-known
that }$\partial \Omega $ \emph{has} \emph{Hausdorff dimension equal to $%
d:=\ln (4)/\ln (3),$ see \cite{Wal}. In this case, the }$d$\emph{-dimensional%
} \emph{Hausdorff measure of} $\partial \Omega $ \emph{is finite}, $\mathcal{%
H}^{d}(\partial \Omega )<\infty ,$\emph{\ so that} $\Gamma _{\mathcal{H}%
^{d}}^{\infty }=\emptyset $\emph{\ and }$D(\mathcal{H}^{d})=\Gamma _{%
\mathcal{H}^{d}}=\partial \Omega $\emph{.}
\end{example}

Next, we give another example which is also meaningful for the applications.

\begin{example}
\label{ram}\emph{Let now }$\Omega \subset \mathbb{R}^{2}$\emph{\ be the
class of tree-shaped domains with self-similar fractal boundary introduced
in \cite{MF}. Such a geometry can be seen as a two-dimensional idealization
of the bronchial tree (see \cite{ADT}; cf. also \cite{AT, ATbis, ATtris, AT4}%
). In order to describe this set, we follow \cite{ADT}. Consider four real
numbers }$a,\alpha ,\beta ,\theta $\emph{\ such that }$0<a<1/\sqrt{2}\emph{,}
$ $\alpha >0$, $\beta >0$ \emph{and }$0<\theta <\pi /2$. \emph{Let }$%
f_{i},i=1,2,$\emph{\ be two similitudes in }$\mathbb{R}^{2}$\emph{\ given by}%
\begin{align*}
f_{1}\left( x_{1},x_{2}\right) & =\left( -\alpha ,\beta \right) +a\left(
x_{1}\cos \theta -x_{2}\sin \theta ,x_{1}\sin \theta +x_{2}\cos \theta
\right) , \\
f_{2}\left( x_{1},x_{2}\right) & =\left( \alpha ,\beta \right) +a\left(
x_{1}\cos \theta +x_{2}\sin \theta ,-x_{1}\sin \theta +x_{2}\cos \theta
\right) .
\end{align*}%
\emph{One can define by }$\Gamma _{f}$\emph{\ the self-similar set
associated to the similitudes }$f_{1},f_{2}$\emph{, i.e., the unique compact
subset of }$\mathbb{R}^{2}$\emph{\ such that }$\Gamma _{f}=f_{1}\left(
\Gamma _{f}\right) \cup f_{2}\left( \Gamma _{f}\right) $\emph{. To construct
the ramified domain whose boundary is }$\Gamma _{f}$\emph{, we further let }$%
\mathbb{A}_{n}$\emph{\ be the set containing all the }$2^{n}$\emph{\
mappings from }$\{1,...,n\}$\emph{\ to }$\{1,2\},$\emph{\ called strings of
length }$n$\emph{\ for }$n>1$\emph{, and define }$\mathbb{A}%
=\bigcup\limits_{n\geq 1}\mathbb{A}_{n}$\emph{\ as the set containing the
empty string and all the finite strings. Next, consider two points }$%
P_{1}=\left( -1,0\right) $\emph{\ and }$P_{2}=\left( 1,0\right) $\emph{\ and
let }$\Gamma _{b}$\emph{\ be the line segment }$\left[ P_{1},P_{2}\right] $%
\emph{. We impose that }$f_{2}\left( P_{1}\right) $\emph{\ and }$f_{2}\left(
P_{2}\right) $\emph{\ have positive coordinates (i.e., }$a\cos \theta
<\alpha $\emph{\ and }$a\sin \theta <\beta $\emph{). The first cell }$Y_{0}$%
\emph{\ of the tree domain }$\Omega $\emph{\ (i.e., the bottom of the tree,
see Figure 2) is constructed by assuming that }$Y_{0}$\emph{\ is the convex,
hexagonal, open domain inside the closed polygonal line joining the points }$%
P_{1},$\emph{\ }$P_{2}$\emph{, }$f_{2}\left( P_{2}\right) $\emph{, }$%
f_{2}\left( P_{1}\right) $\emph{, }$f_{1}\left( P_{2}\right) $\emph{, }$%
f_{1}\left( P_{1}\right) $\emph{, and }$P_{1}$\emph{\ in this order. With
the above assumptions on }$\theta $\emph{,}$\alpha ,\beta ,a,$\emph{\ this
is true if and only if }$\left( \alpha -1\right) \sin \theta +\beta \cos
\theta >0$\emph{. Under these assumptions, the domain }$Y_{0}$\emph{\ is
contained in the half-plane }$x_{2}>0$\emph{\ and is symmetric with respect
to the vertical axis }$x_{1}=0$\emph{. Next, we introduce }$K_{0}=\overline{%
Y_{0}}$\emph{. It is possible to glue together }$K_{0}$\emph{, }$f_{1}\left(
K_{0}\right) $\emph{\ and }$f_{2}\left( K_{0}\right) $\emph{\ and obtain a
new polygonal domain, also symmetric with respect to the axis }$x_{1}=0$%
\emph{. We define the ramified open domain }$\Omega $\emph{,}%
\begin{equation*}
\Omega =Interior\left( \bigcup\limits_{\sigma \in \mathbb{A}}f_{\sigma
}\left( K_{0}\right) \right) ,
\end{equation*}%
\emph{see Figure 2. Note that }$\Omega $\emph{\ is symmetric with respect to
the axis }$x_{1}=0$\emph{. In \cite{MF} it was proved that, for any }$%
0<\theta <\pi /2,$\emph{\ there exists a unique positive number }$a_{\ast
}<1/\sqrt{2},$\emph{\ which does not depend on }$\left( \alpha ,\beta
\right) ,$\emph{\ such that for }$0<a<a_{\ast }$\emph{, }$\Gamma _{f}$\emph{%
\ has "no self-contact". In this case the Hausdorff dimension of }$\Gamma
_{f}$\emph{\ is }$d=-\ln \left( 2\right) /\ln \left( a\right) >1$\emph{\ and
by} \cite{AT08}, $\Omega $\emph{\ possesses the }$W^{1,2}$\emph{-extension
property of Sobolev functions. Moreover, since }$\mathcal{H}^{d}(\Gamma
_{f})<\infty $\emph{\ we have }$\Gamma _{\mathcal{H}^{d}}^{\infty
}=\emptyset $\emph{\ and }$D(\mathcal{H}^{d})=\Gamma _{\mathcal{H}%
^{d}}=\partial \Omega \supset \Gamma _{f}$ \emph{(cf. Definition \ref%
{twosets}). On the other hand, note that for }$\sigma =\mathcal{H}%
^{1}|_{\partial \Omega }$, \emph{we have }$\Gamma _{\sigma }^{\infty
}=\Gamma _{f}$\emph{\ so that }$\Gamma _{\sigma }=D(\sigma )=\partial \Omega
\setminus \Gamma _{f}.$
\end{example}

\begin{figure}[tbp]
\centering\includegraphics[scale=0.66]{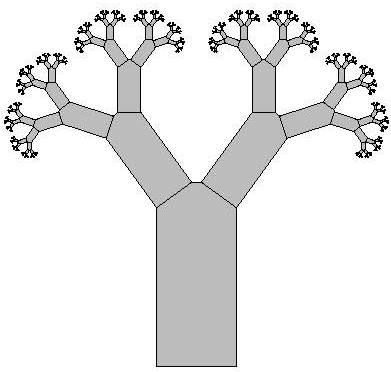}
\caption{Tree-domain ($0<a<a_{\ast }$)}
\end{figure}

\begin{definition}
Let $d\in \lbrack 0,N]$ and let $\mu $ be a regular Borel measure on $%
\partial \Omega $.

\begin{enumerate}
\item We say that $\mu $ is an \emph{upper $d$-Ahlfors measure} if there
exists a constant $C_{2}>0$ such that for every $x\in \partial \Omega $ and
every $r\in (0,1]$, one has
\begin{equation*}
\mu (B(x,r)\cap \partial \Omega )\leq C_{2}r^{d}.
\end{equation*}

\item We say that $\mu $ is a \emph{lower $d$-Ahlfors measure} if there
exists a constant $C_{1}>0$ such that for every $x\in \partial \Omega $ and
every $r\in (0,1]$, one has
\begin{equation*}
\mu (B(x,r)\cap \partial \Omega )\geq C_{1}r^{d}.
\end{equation*}
\end{enumerate}
\end{definition}

By \cite[Remark 2.20]{BW2}, every upper $d$-Ahlfors measure $\mu $ is
absolutely continuous with respect to the $d$-dimensional Hausdorff measure $%
\mathcal{H}^{d}$, that is, there exists a function $f\in L^{\infty
}(\partial \Omega ,\mathcal{H}^{d})$ such that $\mu (A)=\int_{A}f\left(
x\right) d\mathcal{H}^{d}$ for every Borel set $A\subset\partial \Omega$.
Moreover, by \cite[Lemma 1.17]{MZ}, the $d$-dimensional Hausdorff measure $%
\mathcal{H}^{d}$ is absolutely continuous with respect to every lower $d $%
-Ahlfors measure $\mu $.

\begin{definition}
\label{def_dset}Let $F$ be a closed set in ${\mathbb{R}}^{N}$. We say that $%
F $ is a $d$-set for some $d\in (0,N]$ in the sense that there exist a Borel
measure $\mu $ on $F$ and some constants $C_{2}>C_{1}>0$ such that
\begin{equation}
C_{1}r^{d}\leq \mu (F\cap B(x,r))\leq C_{2}r^{d}\;\text{ for all }x\in F
\text{ and all }r\in (0,1],  \label{hausdor}
\end{equation}%
where $B(x,r)$ is the Euclidean metric ball. Notice that a measure
satisfying \eqref{hausdor} is called a $d$-Ahlfors measure.
\end{definition}

Let now $\Omega \subset \mathbb{R}^{N}$ be an arbitrary bounded domain with
boundary $\partial \Omega .$ By \cite{JW} (see also \cite{CK}), $\partial
\Omega $ is a $d$-set if and only if \eqref{hausdor} holds with $\mu $ being
the restriction to $\partial \Omega $ of the $d$-dimensional Hausdorff
measure $\mathcal{H}^{d}$. Hence, it follows directly from \eqref{hausdor}
that $\mathcal{H}^{d}$ is an upper $d$-Ahlfors measure. In particular, $D(%
\mathcal{H}^{d})=\Gamma _{\mathcal{H}^{d}}=\partial \Omega $ so that $\Gamma
_{\mathcal{H}^{d}}^{\infty }=\emptyset $. If $d>N-1$, then $D(\sigma
)=\Gamma _{\sigma }=\emptyset $ so that $\Gamma _{\sigma }^{\infty
}=\partial \Omega$, where we recall that $\sigma=\mathcal{H}%
^{N-1}|_{\partial\Omega}$.

\begin{example}
\label{ex-d-set} \emph{Let }$\Omega \subset \mathbb{R}^{2}$\emph{\ be the
Koch snowflake and recall that $d=\ln (4)/\ln (3)$ is the Hausdorff
dimension of $\partial \Omega $ (cf. Example \ref{snow}). By \cite{Wal}, the
restriction of $\mathcal{H}^{d}$ to $\partial \Omega $ is an upper $d$%
-Ahlfors measure$.$ Let now }$\Omega \subset \mathbb{R}^{2}$\emph{\ be the
tree-domain in Example \ref{ram}. By \cite{H81}, there exists a unique
regular Borel invariant measure }$\mu $\emph{\ on }$\Gamma _{f}$\emph{\ in
the sense that for every Borel set }$B,$\emph{\ }$\mu \left( B\right) =\mu
\left( f_{1}^{-1}\left( B\right) \right) +\mu \left( f_{2}^{-1}\left(
B\right) \right) .$\emph{\ The set }$\Gamma _{f},$\emph{\ endowed with the
self-similar measure }$\mu ,$\emph{\ is a d-set where }$d=-\ln \left(
2\right) /\ln \left( a\right) $ \emph{(see \cite{ADT}), hence }$\mu _{\mid
\Gamma _{f}}=\mathcal{H}^{d}|_{\Gamma _{f}}$\emph{.}
\end{example}

For more properties of Ahlfors measures and $d$-sets we refer the reader to
\cite{CK,DGN,Fal,JW} and the references therein.

\subsection{The Maz'ya space}

In this subsection, we introduce the function spaces needed to study our
problem. We recall that $\Omega \subset {\mathbb{R}}^{N}$ ($N\ge 2$) is an
arbitrary bounded domain with boundary $\partial \Omega $ and the measure $%
\sigma =\mathcal{H}^{N-1}|_{\partial \Omega }$. \newline

For $p\in \lbrack 1,\infty )$, we denote by $W^{1,p}(\Omega )$ the first
order Sobolev space endowed with the norm
\begin{equation*}
\Vert u\Vert _{W^{1,p}(\Omega )}=\left( \Vert \nabla u\Vert _{p,\Omega
}^{p}+\Vert u\Vert _{p,\Omega }^{p}\right) ^{1/p}.
\end{equation*}%
Then $W^{1,p}(\Omega )$ is a Banach space and $W^{1,2}(\Omega )$ is a
Hilbert space. We let
\begin{equation*}
\mathcal{W}^{1,2}(\Omega ):=\overline{W^{1,2}(\Omega )\cap C(\overline{%
\Omega })}^{W^{1,2}(\Omega )}.
\end{equation*}%
Then $\mathcal{W}^{1,2}(\Omega )$ is a proper closed subspace of $%
W^{1,2}(\Omega )$ such that they coincide if, for instance, $\Omega $ has
the segment property (see \cite[Theorem 1, p.28]{MP}). We recall that the
latter property is equivalent to $\Omega $ being of class $C$, which is
slightly weaker than the Lipschitz property of $\Omega .$

The following important inequality is due to Maz'ya \cite[Section 3.6, p.189]%
{Maz85}.

\begin{theorem}[Maz'ya]
There is a constant $C=C(N)>0$ such that for every function $u\in
W^{1,1}(\Omega )\cap C(\overline{\Omega })$,
\begin{equation}
\Vert u\Vert _{\frac{N}{N-1},\Omega }\leq C(N)\left( \Vert \nabla u\Vert
_{1,\Omega }+\Vert u\Vert _{L^{1}(\partial \Omega ,\sigma )}\right) .
\label{M1}
\end{equation}%
In the inequality \eqref{M1}, the constant $C(N)$ is exactly the so called
isoperimetric constant.
\end{theorem}

The following result is a direct consequence of the Maz'ya inequality %
\eqref{M1}.

\begin{corollary}
(\cite[Corollary 2.11.2]{MP})\label{CM2} Let $1\leq p<\infty $ and $p^{\star
}:=pN/(N-1)$. There is a constant $C=C(N,p,\Omega )>0$ such that for every $%
u\in W^{1,p}(\Omega )\cap C(\overline{\Omega })$,
\begin{equation}
\Vert u\Vert _{p^{\star },\Omega }^{p}\leq C\left( \Vert \nabla u\Vert
_{p,\Omega }^{p}+\Vert u\Vert _{L^{p}(\partial \Omega ,\sigma )}^{p}\right) .
\label{M2-N-1}
\end{equation}
\end{corollary}

\begin{remark}
\label{oppy}\emph{The exponent $p^{\star }$ in (\ref{M2-N-1}) is optimal in
the sense that it cannot be improved without further regularity assumptions
on $\Omega ,$ see \cite[Example 2.11, p.123]{MP}}.\newline
\end{remark}

Now, we let the classical Maz'ya space $W_{p,p}^{1}(\Omega ,\partial \Omega
) $ to be the abstract completion of
\begin{equation*}
W_{\sigma,p }:=\{u\in W^{1,p}(\Omega )\cap C(\overline{\Omega }%
):\;\int_{\partial \Omega }|u|^{p}d\sigma <\infty \}
\end{equation*}%
with respect to the norm
\begin{equation*}
\Vert u\Vert _{W_{p,p}^{1}(\Omega ,\partial \Omega )}=\left( \Vert \nabla
u\Vert _{p,\Omega }^{p}+\Vert u\Vert _{L^{p}(\partial \Omega ,\sigma
)}^{p}\right) ^{1/p}.
\end{equation*}%
By \eqref{M2-N-1}, it follows that%
\begin{equation}
W_{p,p}^{1}(\Omega ,\partial \Omega )\hookrightarrow L^{p^{\star }}(\Omega
)\;\mbox{(only continuous embedding)}.  \label{M2-N}
\end{equation}%
In this paper, we focus on the Hilbert space case $p=2$, that is, the space $%
W_{2,2}^{1}(\Omega ,\partial \Omega )$.

Next, let $\mu $ be an arbitrary regular Borel measure on the boundary $%
\partial \Omega $. We define the Maz'ya type space $W_{2,2}^{1}(\Omega
,\partial \Omega ,\mu )$ to be the abstract completion of
\begin{equation*}
W_{\mu,2 }:=\{u\in W^{1,2}(\Omega )\cap C(\overline{\Omega }%
):\;\int_{\partial \Omega }|u|^{2}d\mu <\infty \}
\end{equation*}%
with respect to the norm
\begin{equation*}
\Vert u\Vert _{W_{2,2}^{1}(\Omega ,\partial \Omega ,\mu )}=\left( \Vert
u\Vert _{W^{1,2}(\Omega )}^{2}+\int_{\partial \Omega }|u|^{2}\,d\mu \right)
^{1/2}.
\end{equation*}%
It follows from \eqref{M2-N-1} that the spaces $W_{2,2}^{1}(\Omega ,\partial
\Omega ,\sigma )$ and $W_{2,2}^{1}(\Omega ,\partial \Omega )$ coincide with
equivalent norms.

Now, for a measurable function $u$ on $\partial \Omega $ and $s\in (0,1)$,
we let
\begin{equation*}
\mathcal{N}_{s,\mu }(u):=\int \int_{\partial \Omega \times \partial \Omega
}K_{s}(x,y)|u(x)-u(y)|^{2}d\mu _{x}d\mu _{y},\text{ }K_{s}(x,y):=\frac{1}{%
|x-y|^{N-1+2s}}.
\end{equation*}%
We let the fractional order Sobolev-Besov type space
\begin{equation*}
{H}^{s}(\partial \Omega ,\mu ):=\{u\in L^{2}(\partial \Omega ,\mu ):\;\;%
\mathcal{N}_{s,\mu }(u)<\infty \}
\end{equation*}%
be equipped with the norm
\begin{equation*}
\Vert u\Vert _{{H}^{s}(\partial \Omega ,\mu )}:=\left( \Vert u\Vert
_{L^{2}(\partial \Omega ,\mu )}^{2}+\mathcal{N}_{s,\mu }(u)\right) ^{1/2}.
\end{equation*}%
For further properties of the space ${H}^{s}(\partial \Omega ,\mu )$ we
refer the reader to \cite{DGN} and the references therein. Let $({H}%
^{s}(\partial \Omega ,\mu ))^{\ast }$ denote the dual of the Hilbert space ${%
H}^{s}(\partial \Omega ,\mu )$. We define on the boundary a (nonlocal)
operator $\Theta _{\mu }:\;{H}^{s}(\partial \Omega ,\mu )\rightarrow ({H}%
^{s}(\partial \Omega ,\mu ))^{\ast }$ as follows: for $u,v\in {H}%
^{s}(\partial \Omega ,\mu )$ we set
\begin{equation}
\langle \Theta _{\mu }(u),v\rangle :=\int \int_{\partial \Omega \times
\partial \Omega }K_{s}(x,y)(u(x)-u(y))(v(x)-v(y))d\mu _{x}d\mu _{y},
\label{nonlocal-op}
\end{equation}%
where $\langle \cdot ,\cdot \rangle $ denotes the duality between ${H}%
^{s}(\partial \Omega ,\mu )$ and $({H}^{s}(\partial \Omega ,\mu ))^{\ast }$.
With some further abuse of notation, from now on $\langle \cdot ,\cdot
\rangle $ denotes the duality between any Banach space $X$ and its dual $%
X^{\ast }.$

Now, we let $\mathcal{V}_{\Theta }^{1,2}(\Omega ,\partial \Omega ,\mu )$ be
the abstract completion of
\begin{equation*}
W_{\mu ,\Theta }:=\{u\in W_{\mu ,2}:\;\int \int_{\partial \Omega \times
\partial \Omega }K_{s}(x,y)|u(x)-u(y)|^{2}d\mu _{x}d\mu _{y}<\infty \}
\end{equation*}%
with respect to the norm
\begin{equation*}
\Vert u\Vert _{\mathcal{V}_{\Theta }^{1,2}(\Omega ,\partial \Omega ,\mu
)}:=\left( \Vert u\Vert _{W^{1,2}(\Omega )}^{2}+\int_{\partial \Omega
}|u|^{2}d\mu +\int \int_{\partial \Omega \times \partial \Omega
}K_{s}(x,y)|u(x)-u(y)|^{2}d\mu _{x}d\mu _{y}\right) ^{1/2}.
\end{equation*}%
It is clear that $\mathcal{V}_{\Theta }^{1,2}(\Omega ,\partial \Omega ,\mu )$
is continuously embedded into $W_{2,2}^{1}(\Omega ,\partial \Omega ,\mu )$
and if $\mu =\sigma $, then
\begin{equation*}
\Vert |u\Vert |:=\left( \Vert \nabla u\Vert _{2,\Omega }^{2}+\int_{\partial
\Omega }|u|^{2}d\sigma +\int \int_{\partial \Omega \times \partial \Omega
}K_{s}(x,y)|u(x)-u(y)|^{2}d\sigma _{x}d\sigma _{y}\right) ^{1/2}.
\end{equation*}%
defines an equivalent norm on the space $\mathcal{V}_{\Theta }^{1,2}(\Omega
,\partial \Omega ,\sigma ),$ and this space is continuously embedded into $%
W_{2,2}^{1}(\Omega ,\partial \Omega )$. We notice that if $\Omega $ has
Lipschitz boundary, then the spaces $W_{2,2}^{1}(\Omega ,\partial \Omega
)=W_{2,2}^{1}(\Omega ,\partial \Omega ,\sigma )=W^{1,2}(\Omega )$ with
equivalent norms. Moreover, if the measure $\mu \equiv 0$, then we always
have $W_{2,2}^{1}(\Omega ,\partial \Omega ,0)=\mathcal{V}_{\Theta
}^{1,2}(\Omega ,\partial \Omega ,0)=\mathcal{W}^{1,2}(\Omega )$ with
equivalent norms. On the other hand, for arbitrary domains or/and for
arbitrary regular Borel measures, in order to have an explicit description
of the Maz'ya type space $W_{2,2}^{1}(\Omega ,\partial \Omega ,\mu )$ and
its subspace $\mathcal{V}_{\Theta }^{1,2}(\Omega ,\partial \Omega ,\mu )$,
we need the following notion of capacity.

\begin{definition}
The \emph{relative capacity} $Cap_{2,\Omega }$ with respect to $\Omega $ is
defined for sets $A\subset \overline{\Omega }$ by
\begin{equation*}
Cap_{2,\Omega }(A):=\inf \left\{ \Vert u\Vert _{W^{1,2}(\Omega )}^{2}\;:\;%
\begin{array}{l}
u\in {\mathcal{W}}^{1,2}(\Omega ),\;\exists \;O\subset {\mathbb{R}}^{N}\text{
open,} \\
A\subset O\text{ and }u\geq 1\text{ a.e. on }\Omega \cap O%
\end{array}%
\right\} .
\end{equation*}

\begin{itemize}
\item A set $P\subset \overline{\Omega }$ is called $Cap_{2,\Omega }$\emph{%
-polar} if $Cap_{2,\Omega }(P)=0$.

\item We say that a property holds $Cap_{2,\Omega }$-quasi everywhere
(briefly q.e.) on a set $A\subset \overline{\Omega }$, if there exists a $%
Cap_{2,\Omega }$-polar set $P$ such that the property holds for all $x\in
A\setminus P$.

\item A function $u$ is called $Cap_{2,\Omega }$-quasi continuous on a set $%
A\subset \overline{\Omega }$ if for all $\varepsilon >0$ there exists an
open set $O$ in the metric space $\overline{\Omega }$ such that $%
Cap_{2,\Omega }(O)\leq \varepsilon $ and $u$ restricted to $A\setminus O$ is
continuous.
\end{itemize}
\end{definition}

The relative capacity $Cap_{2,\Omega }$ has been introduced in \cite{AW1}
(see also \cite{AW2,War,W3} for a more general version) to study the Laplace
operator with linear (local) Robin boundary conditions on arbitrary open
subsets in ${\mathbb{R}}^{N}$. Note that if $\Omega ={\mathbb{R}}^{N}$, then
$Cap_{2,\mathbb{R}^{N}}=Cap_{2}$ is the classical Wiener capacity \cite{B,
EG}. By \cite[Section 2]{AW1}, for every $u\in {\mathcal{W}}^{1,2}(\Omega )$
there exists a unique (up to a $Cap_{2,\Omega }$-polar set) $Cap_{2,\Omega }$%
-quasi continuous function $\tilde{u}:\overline{\Omega }\rightarrow {\mathbb{%
R}}$ such that $\tilde{u}=u$ a.e. on $\Omega $. Moreover, if $u\in {\mathcal{%
W}}^{1,2}(\Omega )$ and $u_{n}\in {\mathcal{W}}^{1,2}(\Omega )$ is a
sequence which converges to $u$ in ${\mathcal{W}}^{1,2}(\Omega )$, then
there is a subsequence of $\tilde{u}_{n}$ which converges to $\tilde{u}$
q.e. on $\overline{\Omega }$.

We have the following situation regarding the Maz'ya type space $%
W_{2,2}^{1}(\Omega ,\partial \Omega ,\mu )$. Let $\Gamma _{\mu }^{\infty }$
be the relatively closed, and let $D(\mu )=\Gamma _{\mu }$ be the relatively
open subsets of $\partial \Omega $ defined in \eqref{GI} and \eqref{LF},
respectively. Then every function $u\in W_{\mu,2 }$ is such that $u|_{\Gamma
_{\mu }^{\infty }}=0$, where we recall that
\begin{equation*}
W_{\mu,2 }:=\{u\in W^{1,2}(\Omega )\cap C(\overline{\Omega }%
):\;\int_{\partial \Omega }|u|^{2}d\mu <\infty \}.
\end{equation*}%
Let $\Gamma \subset \partial \Omega $ be a relatively closed set. Since the
closure of the set $\{u\in W^{1,2}(\Omega )\cap C(\overline{\Omega }%
):\;u|_{\Gamma }=0\}$ in $W^{1,2}(\Omega )$ is the space $\{u\in \mathcal{W}%
^{1,2}(\Omega ):\;\tilde{u}=0\;\mbox{ q.e. on }\;\Gamma \}$, it follows that
every function $u$ in $W_{2,2}^{1}(\Omega ,\partial \Omega ,\mu )$ is zero
q.e. on $\Gamma _{\mu }^{\infty }$ (see \cite{AW1, AW2}). With these
observations, one has the following descriptions:
\begin{equation}
W_{2,2}^{1}(\Omega ,\partial \Omega ,\mu )=\{u\in {\mathcal{W}}^{1,2}(\Omega
),\;\tilde{u}=0\;\mbox{ q.e. on }\;\Gamma _{\mu }^{\infty }\;\mbox{ and }%
\int_{\Gamma _{\mu }}|\tilde{u}|^{2}d\mu <\infty \}  \label{sp10}
\end{equation}%
and
\begin{equation}
\mathcal{V}_{\Theta }^{1,2}(\Omega ,\partial \Omega ,\mu )=\{u\in
W_{2,2}^{1}(\Omega ,\partial \Omega ,\mu ),\;\;\int \int_{\Gamma _{\mu
}\times \Gamma _{\mu }}K_{s}(x,y)|\tilde{u}(x)-\tilde{u}(y)|^{2}d\mu
_{x}d\mu _{y}<\infty \}.  \label{sp11}
\end{equation}%
Note that if $\Gamma _{\mu }^{\infty }=\partial \Omega $ or $\func{Cap}%
_{2,\Omega}(\Gamma_\mu)=0$, then $W_{2,2}^{1}(\Omega ,\partial \Omega ,\mu )=%
\mathcal{V}_{\Theta }^{1,2}(\Omega ,\partial \Omega ,\mu
)=W_{0}^{1,2}(\Omega ):=\overline{\mathcal{D}(\Omega )}^{W^{1,2}(\Omega )}$.
This is the case for the well-known $2$-dimensional open set bounded by the
Koch snowflake if one takes the measure $\mu =\sigma $ and is also the case
for $d$-sets if $\mu =\sigma $ and $d>N-1$. \newline

Throughout the remainder, we identify each function $u\in {\mathcal{W}}%
^{1,2}(\Omega )$ with its quasi-continuous version $\tilde{u}$. With this
identification, we define the bilinear symmetric form $\mathcal{A}_{\Theta
,\mu }$ with domain $\mathcal{V}_{\Theta }^{1,2}(\Omega ,\partial \Omega
,\mu )$ by
\begin{align}
\mathcal{A}_{\Theta ,\mu }(u,v):=& \int_{\Omega }\nabla u\cdot \nabla
vdx+\int_{\partial \Omega }uvd\mu  \notag \\
& +\int \int_{\partial \Omega \times \partial \Omega
}K_{s}(x,y)(u(x)-u(y))(v(x)-v(y))d\mu _{x}d\mu _{y}.  \label{form0}
\end{align}%
Since every function $u\in \mathcal{V}_{\Theta }^{1,2}(\Omega ,\partial
\Omega ,\mu )$ is such that $u=0$ q.e. on $\Gamma _{\mu }^{\infty }$, one
has that the form $\mathcal{A}_{\Theta ,\mu }$ defined in \eqref{form0} is
given by
\begin{align}
\mathcal{A}_{\Theta ,\mu }(u,v):=& \int_{\Omega }\nabla u\cdot \nabla
vdx+\int_{\Gamma _{\mu }}uvd\mu  \notag \\
& +\int \int_{\Gamma _{\mu }\times \Gamma _{\mu
}}K_{s}(x,y)(u(x)-u(y))(v(x)-v(y))d\mu _{x}d\mu _{y}.  \label{form}
\end{align}%
The form $\mathcal{A}_{\Theta ,\mu }$ is closed on $L^{2}(\Omega )$ if and
only if the functional $\varphi _{\Theta ,\mu }:\;L^{2}(\Omega )\rightarrow
\lbrack 0,+\infty ]$ defined by
\begin{equation*}
\varphi _{\Theta ,\mu }(u):=%
\begin{cases}
\frac{1}{2}\mathcal{A}_{\Theta ,\mu }(u,u)\;\;\; & \mbox{ if }\;u\in
\mathcal{V}_{\Theta }^{1,2}(\Omega ,\partial \Omega ,\mu ), \\
+\infty & \mbox{ if }\;u\in L^{2}(\Omega )\setminus \mathcal{V}_{\Theta
}^{1,2}(\Omega ,\partial \Omega ,\mu )%
\end{cases}%
\end{equation*}%
is lower semi-continuous on $L^{2}(\Omega )$. To characterize the lower
semi-continuity of $\varphi _{\Theta ,\mu }$, and hence, the closedness of $%
\mathcal{A}_{\Theta ,\mu }$, we need to following notion of boundary
regularity and/or measure regularity.

\begin{definition}
\label{adm} We say that a subset $\Gamma $ of $\partial \Omega $ is \emph{$%
\func{Cap}_{2,\Omega }$}-admissible with respect to $\mu $, if for every
Borel set $A\subset \Gamma $, one has $\func{Cap}_{2,\Omega }(A)=0$ implies $%
\mu (A)=0$.
\end{definition}

The following result is taken from \cite[Theorems 4.4 and 5.2]{W3} (see also
\cite[Theorem 4.2.1]{War} and \cite[Theorem 2.11]{BW2}).

\begin{theorem}
\label{lower-semi} The following assertions are equivalent.

\begin{enumerate}
\item[(i)] The operator $R:\;W_{2,2}^{1}(\Omega ,\partial \Omega ,\mu
)\rightarrow L^{2}(\Omega )$, $u\mapsto u|_{\Omega}$ is injective.

\item[(ii)] The set $\Gamma _{\mu }$ is $\func{Cap}_{2,\Omega }$-admissible
with respect to $\mu $.

\item[(iii)] The bilinear form $\mathcal{A}_{\Theta ,\mu }$ is closed on $%
L^{2}(\Omega )$.
\end{enumerate}
\end{theorem}

To see that Theorem \ref{lower-semi} applies to a large class of rough
domains, let us recall that $\Omega $ is said to have the $W^{1,2}$%
-extension property if for every $u\in W^{1,2}(\Omega )$ there exists $U\in
W^{1,2}({\mathbb{R}}^{N})$ such that $U|_{\Omega }=u$ a.e. In that case, by
\cite[Theorem 5]{Kos}, there exists a bounded linear extension operator $%
\mathcal{E}$ from $W^{1,2}(\Omega )$ into $W^{1,2}({\mathbb{R}}^{N})$. It is
worth emphasizing that for extension domains the spaces $\mathcal{W}%
^{1,2}(\Omega )$ and $W^{1,2}(\Omega )$ coincide. The following important
result is proven in \cite[Lemma 4.7]{W3}.

\begin{theorem}
If $\Omega \subset {\mathbb{R}}^{N}$ ($N\geq 2$) has the $W^{1,2}$-extension
property, then $\partial \Omega $ is $\func{Cap}_{2,\Omega }$-admissible
with respect to $\mu =\mathcal{H}_{\mid \partial \Omega }^{d},$ the $d$%
-dimensional Hausdorff measure of $\partial \Omega ,$ for any $d\in (N-2,N).$
\end{theorem}

Furthermore, if $\mu $ is an upper $d$ -Ahlfors measure on $\partial \Omega $
for some $d\in (N-2,N)\cap (0,N)$, then by \cite[Remark 2.20]{BW2}, $%
\partial \Omega $ is $\func{Cap}_{2,\Omega }$-admissible with respect to $%
\mu $. In particular, since bounded Lipschitz domains and, for example, the
Koch snowflake and the ramified tree domains (see Examples \ref{snow} and %
\ref{ram}, respectively), have the $W^{1,2}$-extension property then their
boundaries are $\func{Cap}_{2,\Omega }$-admissible with respect to $\sigma $%
. For the Koch snowflake and the tree domain, $\partial \Omega $ is also $%
\func{Cap}_{2,\Omega }$-admissible with respect to $\mathcal{H}_{\mid
\partial \Omega }^{d}$ where $d=\ln (4)/\ln (3)$ for the snowflake and $%
d=-\ln (2)/\ln (a)$ for the tree domain, respectively. Some examples of open
sets, whose boundaries are not $\func{Cap}_{2,\Omega }$-admissible with
respect to $\sigma $ are contained in \cite[Examples 1.5, 1.6, 4.2, 4.3]{AW1}%
.

\subsection{The Laplacian with nonlocal Robin boundary conditions}

\label{sec-Maz}

In this subsection we define a realization of the Laplace operator with
nonlocal Robin boundary conditions on $L^{2}(\Omega )$. We also deduce some
properties of this operator that are needed to prove our main results.%
\newline

Throughout the remainder of this article, we assume that $\Omega \subset {%
\mathbb{R}}^{N}$ is a bounded domain with boundary $\partial \Omega $, $\mu $
is a regular Borel measure on $\partial \Omega $ such that its domain $%
\Gamma _{\mu }$ defined in \eqref{LF} is $\func{Cap}_{2,\Omega }$-admissible
with respect to $\mu $. Then, by Theorem \ref{lower-semi}, the bilinear form
$\mathcal{A}_{\Theta ,\mu }$ defined in \eqref{form} is closed on $%
L^{2}(\Omega )$. Let $A_{\Theta ,\mu }$ be the closed, linear, self-adjoint
operator associated with $\mathcal{A}_{\Theta ,\mu }$ in the sense that
\begin{equation}
\begin{cases}
D(A_{\Theta ,\mu }) & :=\{u\in \mathcal{V}_{\Theta }^{1,2}(\Omega ,\partial
\Omega ,\mu ),\;\exists \;v\in L^{2}(\Omega ),\;\mathcal{A}_{\Theta ,\mu
}(u,\varphi )=(v,\varphi )_{L^{2}(\Omega )},\;\forall \;\varphi \in \mathcal{%
V}_{\Theta }^{1,2}(\Omega ,\partial \Omega ,\mu )\} \\
A_{\Theta ,\mu }u & :=v.%
\end{cases}
\label{sa_op}
\end{equation}

The operator $A_{\Theta ,\mu }$ can be described explicitly as follows:%
\begin{equation*}
D(A_{\Theta ,\mu })=\{u\in \mathcal{V}_{\Theta }^{1,2}(\Omega ,\partial
\Omega ,\mu ):\Delta u\in L^{2}(\Omega ),\;\partial _{\nu }ud\sigma
+(u+\Theta _{\mu }(u))d\mu =0\;\mbox{ on }\;\partial \Omega \},
\end{equation*}%
and
\begin{equation*}
A_{\Theta ,\mu }u=-\Delta u.
\end{equation*}%
We note that if $\mu =\sigma $ or, more generally, if $\mu $ is absolutely
continuous with respect to $\sigma $ (in particular, $d\mu _{x}=b\left(
x\right) d\sigma _{x}$, for some nonnegative bounded measurable function $%
b\left( x\right) \in L^{\infty }\left( \partial \Omega \right) $), then
\begin{equation*}
D(A_{\Theta ,\sigma })=\{u\in \mathcal{V}_{\Theta }^{1,2}(\Omega ,\partial
\Omega ,\sigma ):\Delta u\in L^{2}(\Omega ),\;\partial _{\nu
}u+b(x)u+a(x,y)\Theta _{\sigma }(u)=0\;\mbox{ on }\;\partial \Omega \},
\end{equation*}%
which corresponds to the classical (or usual) nonlocal Robin boundary
conditions. Here $a\left( x,y\right) =b\left( x\right) b\left( y\right) $ is
a bounded measurable function on $\partial \Omega \times \partial \Omega $
such that $d\mu _{x}d\mu _{y}=a\left( x,y\right) d\sigma _{x}d\sigma _{y}$.
We also notice that the description \eqref{sp11} for $\mathcal{V}_{\Theta
}^{1,2}(\Omega ,\partial \Omega ,\mu )$ shows that one always has a
Dirichlet boundary conditions on the part $\Gamma _{\mu }^{\infty }$.
Clearly the operator $A_{\Theta ,\mu },$ defined above, satisfies
\begin{equation}
A_{\Theta ,\mu }:\;\mathcal{V}_{\Theta }^{1,2}(\Omega ,\partial \Omega ,\mu
)\rightarrow (\mathcal{V}_{\Theta }^{1,2}(\Omega ,\partial \Omega ,\mu
))^{\ast },  \label{op-star}
\end{equation}%
that is, $A_{\Theta ,\mu }$ maps $\mathcal{V}_{\Theta }^{1,2}(\Omega
,\partial \Omega ,\mu )$ into $(\mathcal{V}_{\Theta }^{1,2}(\Omega ,\partial
\Omega ,\mu ))^{\ast }$.\newline

Throughout the remainder of this article, if $\mu $ is absolutely continuous
with respect to $\sigma $, we will simply say that $\mu =\sigma $. We have
the following result.

\begin{theorem}
\label{3prop} Let $\Gamma _{\mu }$ be $\func{Cap}_{2,\Omega }$-admissible
with respect to $\mu $. Assume that either $\mu =\sigma $ or that
\begin{equation}
\mathcal{V}_{\Theta }^{1,2}(\Omega ,\partial \Omega ,\mu )\hookrightarrow
L^{2q}(\Omega )\mbox{ for some }q=q\left( N,\Omega \right) >1\;\mbox{ if }%
\;\mu \neq \sigma .  \label{CI-mu}
\end{equation}%
Then the following assertions hold.

\begin{enumerate}
\item The operator $-A_{\Theta ,\mu }$ generates a strongly continuous
semigroup $(e^{-tA_{\Theta ,\mu }})_{t\geq 0}$ of contractions on $%
L^{2}(\Omega ) $ which can be extended to contraction strongly continuous
semigroups on $L^{p}(\Omega )$ for every $p\in \lbrack 1,\infty )$.

\item The operator $A_{\Theta ,\mu }$ has a compact resolvent, and hence has
a discrete spectrum. The spectrum of $A_{\Theta ,\mu }$ is an increasing
sequence of real numbers $0\leq \lambda _{1}<\lambda _{2}<\cdots <\lambda
_{n}<\dots ,$ that converges to $+\infty $. Moreover, if $u_{n}$ is an
eigenfunction associated with $\lambda _{n}$, then $u_{n}\in D\left(
A_{\Theta ,\mu }\right) \cap L^{\infty }(\Omega )$.

\item Assume also that either $\func{Cap}_{2,\Omega }(\Gamma _{\mu }^{\infty
})>0$ or $\mu (\Gamma _{\mu })>0$. Then $0\in \rho (A_{\Theta ,\mu }),$ the
resolvent set of $A_{\Theta ,\mu }$, and there exists a constant $C>0$ such
that, for every $u\in \mathcal{V}_{\Theta }^{1,2}(\Omega ,\partial \Omega
,\mu )$,
\begin{equation}
\Vert u\Vert _{2,\Omega }^{2}\leq C\mathcal{A}_{\Theta ,\mu }(u,u).
\label{eq-coer}
\end{equation}

\item Denoting the generator of the semigroup on $L^{p}(\Omega )$ by $%
A_{p,\Theta ,\mu }$, so that $A_{\Theta ,\mu }=A_{2,\Theta ,\mu }$, then the
spectrum of $A_{p,\Theta ,\mu }$ is independent of $p$. The spectral bound
\begin{equation*}
s\left( A_{p,\Theta ,\mu }\right) =\inf \left\{ \func{Re}\lambda :\lambda
\in \sigma \left( A_{p,\Theta ,\mu }\right) \right\}
\end{equation*}%
of $A_{p,\Theta ,\mu }$ is an algebraically simple eigenvalue. It is the
only eigenvalue having a positive eigenfunction.

\item For each $\theta \in (0,1]$, there holds $D(A_{\Theta ,\mu }^{\theta
})\subset L^{\infty }\left( \Omega \right) $ provided that $\theta >\gamma .$
Here, $\gamma =N/2$ if $\mu =\sigma $ and $\gamma =q/2(q-1)$ if $\mu \neq
\sigma $.
\end{enumerate}
\end{theorem}

\begin{proof}
(a) This part follows from \cite[Corollary 5.3]{W3} (where the assumption %
\eqref{CI-mu} is not needed).

(b) First, we notice that since we have assumed that $\Gamma _{\mu }$ is $%
\func{Cap}_{2,\Omega }$-admissible with respect to $\mu $, we have from
Theorem \ref{lower-semi} that the continuous embedding $\mathcal{V}_{\Theta
}^{1,2}(\Omega ,\partial \Omega ,\mu )\hookrightarrow L^{2}(\Omega )$ is
also an injection. Moreover it follows from \eqref{M2-N} if $\mu =\sigma $
and the assumption \eqref{CI-mu} if $\mu \neq \sigma $, that the embedding $%
\mathcal{V}_{\Theta }^{1,2}(\Omega ,\partial \Omega ,\mu )\hookrightarrow
L^{2}(\Omega )$ is compact since $\Omega $ is bounded. Since $A_{\Theta ,\mu
}$ is a nonnegative self-adjoint operator and has a compact resolvent, then
it has a discrete spectrum which is an increasing sequence of real numbers $%
0\leq \lambda _{1}<\lambda _{2}<\cdots <\lambda _{n}\dots ,$ that converges
to $+\infty $.


Now, let $u_{n}\in \mathcal{V}_{\Theta }^{1,2}(\Omega ,\partial \Omega ,\mu
) $ be an eigenfunction associated with $\lambda _{n}$. Then, for every $%
v\in \mathcal{V}_{\Theta }^{1,2}(\Omega ,\partial \Omega ,\mu )$,
\begin{align*}
& \int_{\Omega }\nabla u_{n}\cdot \nabla vdx+\int_{\Gamma _{\mu }}u_{n}vd\mu
\\
& +\int \int_{\Gamma _{\mu }\times \Gamma _{\mu }}K_{s}\left( x,y\right)
(u_{n}(x)-u_{n}(y))(v(x)-v(y))d\mu _{x}d\mu _{y}=\lambda _{n}\int_{\Omega
}u_{n}vdx.
\end{align*}%
This equality means that $A_{\Theta ,\mu }u_{n}=\lambda _{n}u_{n}$. Let $%
\alpha >0$ be a real number. Since $\alpha \in \rho (-A_{\Theta ,\mu })$, we
have that $\alpha I+A_{\Theta ,\mu }$ is invertible. From $A_{\Theta ,\mu
}u_{n}=\lambda _{n}u_{n}$ we have that
\begin{equation*}
u_{n}=(\alpha I+A_{\Theta ,\mu })^{-1}(\lambda _{n}+\alpha )u_{n}=(\lambda
_{n}+\alpha )(\alpha I+A_{\Theta ,\mu })^{-1}(u_{n}).
\end{equation*}%
By \cite[Theorem 5.4]{W3}, the semigroup $(e^{-tA_{\Theta ,\mu }})_{t\geq 0}$
is ultracontrative in the sense that it maps $L^{2}(\Omega )$ into $%
L^{\infty }(\Omega )$ (this follows from \eqref{M2-N} if $\mu =\sigma ,$ and
from the assumption \eqref{CI-mu} if $\mu \neq \sigma $). More precisely,
there is a constant $C>0$ and $\gamma\in (0,1)$ such that for every $f\in
L^{2}(\Omega )$ and $t\in (0,1)$,
\begin{equation}
\Vert e^{-tA_{\Theta ,\mu}}f\Vert _{\infty ,\Omega }\leq
Ct^{-\gamma}\left\Vert f\right\Vert _{2,\Omega }  \label{ultra}
\end{equation}%
where $\gamma=N/2$ if $\mu=\sigma$ and $\gamma=q/2(q-1)$ if $\mu\ne\sigma$.
Since for every $f\in L^{2}(\Omega )$ and $\alpha >0$,
\begin{equation*}
(\alpha I+A_{\Theta ,\mu })^{-1}f=\int_{0}^{\infty }e^{-\alpha
t}e^{-tA_{\Theta ,\mu }}fdt,
\end{equation*}%
it follows from \eqref{ultra} that there exists a constant $M>0$ such that
\begin{equation*}
\Vert u_{n}\Vert _{\infty ,\Omega }\leq M(\lambda _{n}+\alpha )\Vert
u_{n}\Vert _{2,\Omega }.
\end{equation*}%
This completes the proof of part (b).

(c) Assume to the contrary that $0\notin \rho (A_{\Theta ,\mu })$ (i.e., $0$
is an eigenvalue for $A_{\Theta ,\mu }$). Then there exists a nonzero
function $u\in D(A_{\Theta ,\mu })$ such that $A_{\Theta ,\mu }u=0$, that
is, $\mathcal{A}_{\Theta ,\mu }(u,v)=0,$ for every $v\in \mathcal{V}_{\Theta
}^{1,2}(\Omega ,\partial \Omega ,\mu )$. In particular, we have%
\begin{equation*}
\mathcal{A}_{\Theta ,\mu }(u,u):=\int_{\Omega }|\nabla u|^{2}dx+\int_{\Gamma
_{\mu }}|u|^{2}d\mu +\int \int_{\Gamma _{\mu }\times \Gamma _{\mu
}}K_{s}(x,y)|u(x)-u(y)|^{2}d\mu _{x}d\mu _{y}=0.
\end{equation*}%
This shows that $|\nabla u|=0$. Since $\Omega $ is bounded and connected, we
get that $u=C_{1}$ for some constant $C_{1}$. Since $u=0$ q.e. on $\Gamma
_{\mu }^{\infty }$ (as $u$ belongs to $\mathcal{V}_{\Theta }^{1,2}(\Omega
,\partial \Omega ,\mu )$), there are two alternatives. First, $\func{Cap}%
_{2,\Omega }(\Gamma _{\mu }^{\infty })>0$ in which case $u=C_{1}\not\in
D(A_{\Theta ,\mu })$, whence a contradiction. On the other hand, from the
facts that $u=C_{1}$, we get from the previous equation that $C_{1}\mu
(\Gamma _{\mu })=0$. The second alternative $\mu (\Gamma _{\mu })>0$ then
implies that $u=C_{1}=0$, again a contradiction. We have shown that $0\in
\rho (A_{\Theta ,\mu })$. Now the estimate \eqref{eq-coer} follows from
standard properties of semigroups (see, e.g., \cite{Dav, FW}), since we have
just shown that there exists a constant $w>0$ such that $\Vert
e^{-tA_{\Theta ,\mu }}\Vert _{\mathcal{L}(L^{2}(\Omega ),L^{2}(\Omega
))}\leq e^{-wt}$ for every $t>0$ . This completes the proof of part (c).

(d) Let $A_{p,\Theta ,\mu }$ be the generator of the semigroup on $%
L^{p}(\Omega )$. Since $A_{\Theta ,\mu }=A_{2,\Theta ,\mu }$ has a compact
resolvent and $\Omega $ is bounded, it follows from the ultracontractivity
that each semigroup has a compact resolvent on $L^{p}(\Omega )$ for $p\in
\lbrack 1,\infty ]$. Now it follows from \cite[Corollary 1.6.2]{Dav} that
the spectrum of $A_{p,\Theta ,\mu }$ is independent of $p$. The last
assertion in (d) follows from \cite[Proposition 5.9 and Corollary 5.10]{D1}.

(e) Since $I+A_{\Theta ,\mu }$ is invertible we have that the $L^{2}\left(
\Omega \right) $-norm of $\left( I+A_{\Theta ,\mu }\right) ^{\theta }$
defines an equivalent norm on $D(A_{\Theta ,\mu }^{\theta })$ and%
\begin{equation*}
\left( I+A_{\Theta ,\mu }\right) ^{-\theta }=\frac{1}{\Gamma \left( \theta
\right) }\int_{0}^{\infty }t^{\theta -1}e^{-t}e^{-tA_{\Theta ,\mu }}fdt.
\end{equation*}%
Using (\ref{ultra}) for $t\in \left( 0,1\right) $ and the contractivity of $%
e^{-tA_{\Theta ,\mu }}$ for $t>1$, for $u\in D(A_{\Theta ,\mu }^{\theta })$,
we deduce%
\begin{equation*}
\left\Vert u\right\Vert _{L^{\infty }\left( \Omega \right) }\leq C\left\Vert
u\right\Vert _{D(A_{\Theta ,\mu }^{\theta })}\int_{0}^{1}t^{\theta -1-\gamma
}dt+C\left\Vert u\right\Vert _{D(A_{\Theta ,\mu }^{\theta
})}\int_{1}^{\infty }e^{-t}dt.
\end{equation*}%
The first integral is finite if and only if $\gamma <\theta $. The proof of
theorem is finished.
\end{proof}

\begin{remark}
\label{3rem}
\end{remark}

\begin{enumerate}
\item \emph{We remark that either }$\func{Cap}_{2,\Omega }(\Gamma _{\mu
}^{\infty })>0$\emph{\ or }$\mu (\Gamma _{\mu })>0,$\emph{\ in part (c) of
Theorem \ref{3prop}, is actually not an assumption about the regularity of
the open set }$\Omega $\emph{, but it is only an assumption about the
measure }$\mu $\emph{. Roughly speaking, it means that the trivial choice }$%
\mu \equiv 0$\emph{\ (the zero measure) is not allowed.}

\item \emph{In the case of domains with Lipschitz continuous boundary }$%
\partial \Omega ,$\emph{\ further interesting spectral properties for the
operator }$A_{\Theta ,\mu }$\emph{\ can be also derived (see \cite{GM})}.

\item \emph{Assuming that either $\func{Cap}_{2,\Omega }(\Gamma _{\mu
}^{\infty })>0$ or $\mu (\Gamma _{\mu })>0$ is satisfied, we deduce from %
\eqref{eq-coer},%
\begin{equation}
\Vert |u\Vert |_{\mathcal{V}_{\Theta }^{1,2}(\Omega ,\partial \Omega ,\mu
)}:=\left( \Vert \nabla u\Vert _{2,\Omega }^{2}+\Vert u\Vert _{L^{2}(\Gamma
_{\mu })}^{2}+\mathcal{N}_{s,\mu }(u)\right) ^{1/2}  \label{eq-norm1}
\end{equation}%
defines an equivalent norm on the space $\mathcal{V}_{\Theta }^{1,2}(\Omega
,\partial \Omega ,\mu )$. In particular, under the same assumption, we also
have that
\begin{equation}
\Vert |u\Vert |_{W_{2,2}^{1}(\Omega ,\partial \Omega ,\mu )}:=\left( \Vert
\nabla u\Vert _{2,\Omega }^{2}+\Vert u\Vert _{L^{2}(\Gamma _{\mu
})}^{2}\right) ^{1/2}  \label{eq-norm2}
\end{equation}%
defines an equivalent norm on the space $W_{2,2}^{1}(\Omega ,\partial \Omega
,\mu )$.}

\item \emph{In the case when }$\theta \in (0,1]$ \emph{and }$\mu =\sigma $%
\emph{, the embedding }$D(A_{\Theta ,\mu }^{\theta })\subset L^{\infty
}\left( \Omega \right) $\emph{\ holds provided that }$N<2\theta $\emph{.
This consideration is optimal for an arbitrary open subset }$\Omega \subset
R^{N}$\emph{\ in view of Remark \ref{oppy}. Finally, when }$\Omega $\emph{\
has the }$W^{1,2}$\emph{-extension property of Sobolev functions, we have }$%
q=N/\left( N-2\right) $\emph{\ for }$N\geq 3$\emph{\ and }$q\in \left(
1,\infty \right) $\emph{\ for }$N=2$\emph{. In that case }$D(A_{\Theta ,\mu
}^{\theta })\subset L^{\infty }\left( \Omega \right) $\emph{\ provided that }%
$N<4\theta $\emph{\ if }$N\geq 3$\emph{. In a smooth situation (say when }$%
\partial \Omega \in \mathcal{C}^{2}$\emph{), one has }$D(A_{\Theta ,\mu
})\subset W^{2,2}\left( \Omega \right) $\emph{\ and thus }$D(A_{\Theta ,\mu
}^{\theta })\subset W^{2\theta ,2}\left( \Omega \right) $\emph{. It is
well-known that the embedding }$W^{2\theta ,2}\left( \Omega \right) \subset
L^{\infty }\left( \Omega \right) $\emph{\ holds in the smooth setting
provided that }$N<4\theta .$\emph{\ This shows that the stated embedding }$%
D(A_{\Theta ,\mu }^{\theta })\subset L^{\infty }\left( \Omega \right) $
\emph{appears to be optimal in the non-smooth setting in this sense.}
\end{enumerate}

It is important to observe that the operator $A_{\Theta ,\mu }$ depends on
the choice of the measure $\mu $ on $\partial \Omega .$ In order to see
this, consider Example \ref{snow} when $\Omega $ is the Koch snowflake in $%
\mathbb{R}^{2}$ and take $\mu =\sigma =\mathcal{H}_{\mid \partial \Omega
}^{1}$. For this choice, $\sigma $ is locally infinite on $\partial \Omega $
so that there are no Robin boundary conditions for $\Omega $, on account of
the description (\ref{sp10})-(\ref{sp11}). Indeed, we have \cite[Theorem 2.3]%
{AW1},%
\begin{equation*}
D\left( \mathcal{A}_{\Theta ,\sigma }\right) =\left\{ u\in \mathcal{W}%
^{1,2}\left( \Omega \right) :u=0\text{ q.e. on }\partial \Omega \right\}
=W_{0}^{1,2}\left( \Omega \right)
\end{equation*}%
and, consequently, $A_{\Theta ,\sigma }$ coincides with the realization of
the Laplace operator $-\Delta $ with homogeneous Dirichlet boundary
conditions on $\partial \Omega $. However, with a different choice $\mu =%
\mathcal{H}_{\mid \partial \Omega }^{d}$ (where $d=\ln \left( 4\right) /\ln
\left( 3\right) $ is the Hausdorff dimension of $\partial \Omega $), we have
the generalized Robin boundary condition (\ref{1.2})\ on $\partial \Omega $
since $\Gamma _{\mu }^{\infty }=\emptyset $ (see Example \ref{snow}). This
is also the case for the tree-domain $\Omega $ of Example \ref{ram}, in
which case $A_{\Theta ,\sigma }$ coincides exactly with the realization of
the Laplace operator with homogeneous Dirichlet boundary conditions on $%
\Gamma _{f}$ (the small structures), and nonlocal Robin boundary conditions
on $\partial \Omega \setminus \Gamma _{f}$. However, recall that the fractal
boundary $\Gamma _{f}$ of\ $\Omega $ is a $d$-set, with $d=-\ln \left(
2\right) /\ln \left( a\right) .$ Thus, the natural measure $\mu $ of $%
\partial \Omega $ should be once again the $d$-dimensional Hausdorff measure
$\mu =\mathcal{H}_{\mid \partial \Omega }^{d}$ so that the generalized
(nonlocal) Robin boundary condition (\ref{1.2}) is still realized on the
portion $\Gamma _{f}$. In fact, a more appropriate choice for $\mu $ is%
\begin{equation*}
\mu =\left\{
\begin{array}{ll}
\mathcal{H}^{d}, & \text{on }\Gamma _{f}, \\
\sigma , & \text{on }\partial \Omega \setminus \Gamma _{f},%
\end{array}%
\right.
\end{equation*}%
so that one has the usual (nonlocal) Robin boundary conditions on the
portion $\partial \Omega \setminus \Gamma _{f},$ and the generalized
(nonlocal) Robin boundary conditions on $\Gamma _{f}$. In the general case
when $\Omega \subset \mathbb{R}^{N}$, such arguments also hold if $\partial
\Omega $ is a $d$-set with $d\in (N-1,N].$%
\begin{figure}[tbp]
\centering\includegraphics[scale=0.77]{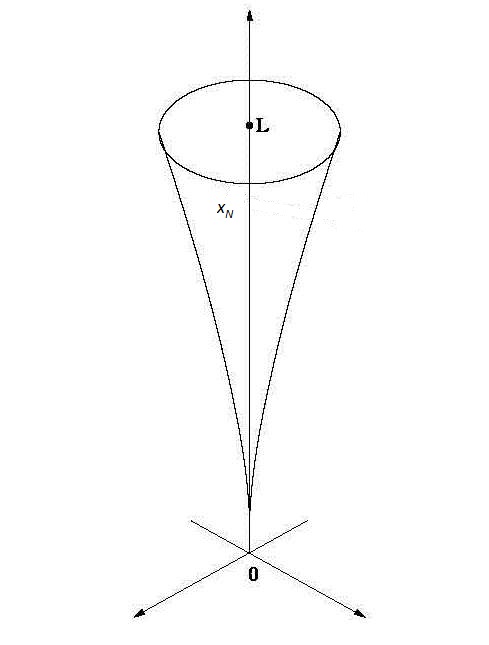}
\caption{Cusp-shaped domain}
\end{figure}

Finally, it is worth emphasizing that the assumption \eqref{CI-mu} in
Theorem \ref{3prop} is satisfied with $q=N/(N-2)$ if, for example, $\Omega $
has the $W^{1,2}$-extension property. On the other hand, if $\Omega \subset {%
\mathbb{R}}^{N}$ is a non-Lipschitz bounded domain whose boundary admits a
finite numbers of outward or inward H\"{o}lder cusps points and exponent at
cusp points $0<\gamma <1$, then $\Omega $ does \emph{not} enjoy the
extension property. However, for such domains $\Omega $ we still have the
following continuous embeddings:%
\begin{equation}
\mathcal{V}_{\Theta }^{1,2}(\Omega ,\partial \Omega ,\mu )\hookrightarrow
W^{1,2}(\Omega )\hookrightarrow L^{2q}(\Omega )\;\mbox{ with }\;q:=\frac{%
(N-1+\gamma )}{N-1-\gamma }>1,  \label{cusp_ineq}
\end{equation}%
see \cite{FT1,FT2}. Many examples of domains with cusps, which are the
simplest non-Lipschitz domains used in the applications, can be found in
\cite[Chapters 7-8]{MP}; see, for instance, Example \ref{cusp} below. For H%
\"{o}lder $C^{\gamma }$-domains and for the measure $\mu =\sigma =\mathcal{H}%
_{\mid \partial \Omega }^{N-1},$ we have from Definition \ref{twosets} that $%
\Gamma _{\sigma }^{\infty }=\emptyset $, whereas $\Gamma _{\sigma }=D\left(
\sigma \right) =\partial \Omega $. Thus, it is clear that in such cases, the
usual (nonlocal) Robin boundary condition (\ref{1.2}) is satisfied on $%
\Gamma _{\sigma }=\partial \Omega .$

\begin{example}
\label{cusp}(see Figure 3) Let $\Omega _{hc}\subset \mathbb{R}^{N}$, $N\geq
2 $ be the following set:%
\begin{equation*}
\Omega _{hc}:=\left\{ x=\left( x_{1},x_{2},...,x_{N}\right) \in \mathbb{R}%
^{N}:x_{N}\in \left( 0,L\right)
,0<\sum\nolimits_{j=1}^{N-1}x_{i}^{2}<lx_{N}^{2/\gamma }\right\} ,
\end{equation*}%
for some $L,l>0$ and $\gamma \in \left( 0,1\right) $. The inequality (\ref%
{cusp_ineq}) is fulfilled for $\Omega =\Omega _{hc}$ and $N\geq 2$.
Generally, the Sobolev inequality (\ref{cusp_ineq}) is satisfied in any H%
\"{o}lder domain $\Omega $ of class $C^{\gamma }$, $\gamma \in \left(
0,1\right) $. By a (bounded) H\"{o}lder $C^{\gamma }$-domain\ we mean the
following: there exist a finite number of balls $B\left( x_{i},r_{i}\right) $%
, $i=1,...,m_{N}$, whose union contains $\Omega $, and for each $%
i=1,...,m_{N}$, there exists a H\"{o}lder-continuous ($C^{\gamma },$ $\gamma
\in \left( 0,1\right) $)\ function $F_{i}:\mathbb{R}^{N-1}\rightarrow
\mathbb{R}$ such that for some coordinate system, $B\left(
x_{i},r_{i}\right) \cap \Omega $ is equal to the intersection of $B\left(
x_{i},r_{i}\right) $ with the region above the graph of $F_{i}$.
\end{example}

\section{Well-posedness and regularity}

\label{wel}

In this section $\Omega \subset {\mathbb{R}}^{N}$ is an arbitrary bounded
domain with boundary $\partial \Omega $, $\mu $ is a regular Borel measure
on $\partial \Omega $ such that $\mu $ is either the $\left( N-1\right) $%
-dimensional Hausdorff meausure $\mathcal{H}_{\mid \partial \Omega }^{N-1}$
or another "arbitrary" regular Borel measure, according to the assumptions
of Section \ref{prelim} (see (H$_{\mu }$) below). We begin this section by
stating all the hypotheses on $f$ we need, even though not all of them must
be satisfied at the same time.

\begin{enumerate}
\item[\textbf{(H1)}] $f\in C_{\text{loc}}^{1}\left( \mathbb{R}\right) $
satisfies%
\begin{equation*}
\liminf_{\left\vert s\right\vert \rightarrow +\infty }\frac{f\left( s\right)
}{s}>-\lambda _{\ast },
\end{equation*}%
for some constant $\lambda _{\ast }\in \lbrack 0,C_{\Omega })$, where $%
C_{\Omega }=C\left( N,\Omega \right) >0$ is the best Sobolev/Poincar\'{e}
constant into (\ref{M2-N-1}) and (\ref{CI-mu}), respectively.

\item[\textbf{(H2)}] $f\in C_{\text{loc}}^{1}\left( \mathbb{R}\right) $
satisfies%
\begin{equation*}
\begin{array}{ll}
C_{f}\left\vert s\right\vert ^{p}-c_{f}\leq f\left( s\right) s\leq
\widetilde{C}_{f}\left\vert s\right\vert ^{p}+\widetilde{c}_{f}, & \text{
for all }s\in {\mathbb{R}}\text{,}%
\end{array}%
\end{equation*}%
for some appropriate positive constants and some $p>1$.

\item[\textbf{(H3)}] $f\in C^{1}\left( \mathbb{R}\right) $ satisfies%
\begin{equation*}
f^{^{\prime }}\left( s\right) \geq -C_{f},\text{ for all }s\in {\mathbb{R}}%
\text{,}
\end{equation*}%
for some positive constant $C_{f}.$
\end{enumerate}

Finally, throughout the remainder of this article our assumption about the
measure $\mu $ will be as follows.

\begin{enumerate}
\item[\textbf{(H$_{\protect\mu }$)}] Assume that $\Gamma _{\mu }$ is $\func{%
Cap}_{2,\Omega }$-admissible with respect to $\mu ,$ and either $\func{Cap}%
_{2,\Omega }(\Gamma _{\mu }^{\infty })>0$ or $\mu (\Gamma _{\mu })>0$.%
\newline
\end{enumerate}

In what follows we shall use classical (linear/nonlinear semigroup)
definitions of strong and generalized (weak) solutions to (\ref{1.1})-(\ref%
{1.3}). \textquotedblleft Strong\textquotedblright\ solutions are defined
via nonlinear semigroup theory for sufficiently smooth initial data and
satisfy the differential equations almost everywhere in $t>0$.
\textquotedblleft Generalized\textquotedblright\ or weak solutions are
defined as (strong) limits of strong solutions. We first introduce the
rigorous notion of (global) weak solutions to the system (\ref{1.1})-(\ref%
{1.3}). Throughout the remainder of this article the solution of the system (%
\ref{1.1})-(\ref{1.3}) is a function that depends on the variables $t\in
[0,\infty)$ and $x\in\Omega$, but in our notation we sometime omit the
dependence in $x$.

\begin{definition}
\label{weak} Let $u_{0}\in L^{2}(\Omega )$ be given and assume (H2) holds
for some $p>1$. The function $u$ is said to be a \emph{weak} solution of (%
\ref{1.1})-(\ref{1.3}) if, for a.e. $t\in \left( 0,T\right) ,$ for any $T>0$%
, the following properties are valid:

\begin{itemize}
\item Regularity:%
\begin{equation}
\left\{
\begin{array}{l}
u\in L^{\infty }\left( 0,T;L^{2}\left( \Omega \right) \right) \cap
L^{p}\left( \left( 0,T\right) \times \Omega \right) \cap L^{2}(0,T;\mathcal{V%
}_{\Theta }^{1,2}(\Omega ,\partial \Omega ,\mu )), \\
\partial _{t}u\in L^{2}(0,T;(\mathcal{V}_{\Theta }^{1,2}(\Omega ,\partial
\Omega ,\mu ))^{\ast })+L^{p^{^{\prime }}}\left( \left( 0,T\right) \times
\Omega \right) ,%
\end{array}%
\right.  \label{reg_weak}
\end{equation}%
where $p^{^{\prime }}=p/\left( p-1\right) .$

\item Variational identity: for the weak solutions the following equality%
\begin{equation}
\left\langle \partial _{t}u\left( t\right) ,\xi \right\rangle +\mathcal{A}%
_{\Theta ,\mu }(u\left( t\right) ,\xi )+\left\langle f\left( u\left(
t\right) \right) ,\xi \right\rangle =0  \label{de_form}
\end{equation}%
holds for all $\xi \in \mathcal{V}_{\Theta }^{1,2}(\Omega ,\partial \Omega
,\mu )\cap L^{p}\left( \Omega \right) ,$ a.e. $t\in \left( 0,T\right) $.
Finally, we have, in the space $L^{2}\left( \Omega \right) ,$ $u\left(
0\right) =u_{0}$ almost everywhere.

\item Energy identity: weak solutions satisfy the following identity%
\begin{equation}
\frac{1}{2}\left\Vert u\left( t\right) \right\Vert _{2,\Omega
}^{2}+\int_{0}^{t}\left[ \mathcal{A}_{\Theta ,\mu }\left( u\left( s\right)
,u\left( s\right) \right) +\left\langle f\left( u\left( s\right) \right)
,u\left( s\right) \right\rangle \right] ds=\frac{1}{2}\left\Vert u\left(
0\right) \right\Vert _{2,\Omega }^{2}.  \label{en_id}
\end{equation}
\end{itemize}
\end{definition}

\begin{remark}
\label{weakrem}Note that by (\ref{reg_weak}), $u \in C_{\text{weak}}\left( %
\left[ 0,T\right] ;L^{2}\left( \Omega \right) \right) $. Therefore the
initial value $u\left( 0\right) =u_{0}$ is meaningful when $u_{0}\in
L^{2}\left( \Omega \right) $.
\end{remark}

Finally, our notion of (global) strong solution is as follows.

\begin{definition}
\label{strong} Let $u_0\in L^\infty(\Omega)$ be given. A weak solution $u$
is "\emph{strong"} if, in addition, it fulfills the regularity properties:%
\begin{equation}
u\in W_{\text{loc}}^{1,\infty }((0,T];L^{2}(\Omega ))\cap C\left( \left[ 0,T%
\right] ;L^{\infty }\left( \Omega \right) \right) ,  \label{time_reg}
\end{equation}%
such that $u\left( t\right) \in D(A_{\Theta ,\mu }),$ a.e. $t\in \left(
0,T\right) ,$ for any $T>0.$
\end{definition}

This section consists of three main parts. At first we will establish the
existence and uniqueness of a strong solution on a finite time interval
using the theory of monotone operators, and exploiting a (modified)
Alikakos-Moser iteration-type argument to show that the strong solution is
actually a global solution. In the second part, we will show the existence
of \ (globally-defined) weak solutions which satisfy the energy identity (%
\ref{en_id}) and the variational form (\ref{de_form}). We shall also discuss
their uniqueness. Finally, by using the energy method combined with a
refined iteration scheme, we show that the weak solution is actually more
regular on intervals of the form $[\delta ,\infty )$, for every $\delta >0.$

\subsection{Strong solutions}

Now we state the main theorem of this (sub)section.

\begin{theorem}
\label{SG} Assume that the nonlinearity $f$\ obeys (H1) and $\mu $ satisfies
(H$_{\mu }$). For every $u_{0}\in L^{\infty }\left( \Omega \right) ,$ there
exists a unique strong solution of (\ref{1.1})-(\ref{1.3}) in the sense of
Definition \ref{strong}.
\end{theorem}

\begin{proof}
\emph{Step 1} (Local existence). Let $u_{0}\in L^{\infty }\left( \Omega
\right) \subset L^{2}\left( \Omega \right) =\overline{D(\mathcal{A}_{\Theta
,\mu })}^{L^{2}(\Omega )}=\overline{D({\varphi }_{\Theta ,\mu })}%
^{L^{2}(\Omega )}$. By Theorem \ref{lower-semi}, ${\varphi }_{\Theta ,\mu }$
is proper, convex and lower semi-continuous on the Hilbert space $%
L^{2}(\Omega )$ (cf. also \cite[Theorem 5.2]{W3}). Moreover, from Theorem %
\ref{3prop} we know that $-A_{\Theta ,\mu }$ ($=-\partial \varphi _{\Theta
,\mu }$) generates a strongly continuous (linear) semigroup $(e^{-tA_{\Theta
,\mu }})_{t\geq 0}$ of contraction operators on $L^{2}(\Omega )$. Finally,
by Theorem \ref{SGbis} (Appendix), $e^{-tA_{\Theta ,\mu }}$ is non-expansive
on $L^{\infty }\left( \Omega \right) ,$
\begin{equation}
\left\Vert e^{-tA_{\Theta ,\mu }}u_{0}\right\Vert _{\infty ,\Omega }\leq
\left\Vert u_{0}\right\Vert _{\infty ,\Omega },\text{ }t\geq 0\text{ and }%
u_{0}\in L^{\infty }\left( \Omega \right) ,  \label{nonexp}
\end{equation}%
and $A_{\Theta ,\mu }=\partial \varphi _{\Theta ,\mu }$ is strongly
accretive by Theorem \ref{3prop}-(c), see (\ref{eq-coer}). Thus, we can give
an operator theoretic version of the original problem (\ref{1.1})-(\ref{1.3}%
):%
\begin{equation}
\partial _{t}u=-A_{\Theta ,\mu }u-f\left( u\right) ,  \label{op_v}
\end{equation}%
with $f$ as a locally Lipschitz perturbation of a monotone operator, on
account of the fact that $f\in C_{\text{loc}}^{1}\left( \mathbb{R}\right) $.
We construct the (locally-defined) strong solution by a fixed point
argument. To this end, fix $0<T^{\ast }\leq T$, consider the space
\begin{equation*}
\mathcal{X}_{T^{\ast },R^{\ast }}\equiv \left\{ u\in C\left( \left[
0,T^{\ast }\right] ;L^{\infty }\left( \Omega \right) \right) :\left\Vert
u\left( t\right) \right\Vert _{\infty ,\Omega }\leq R^{\ast }\right\}
\end{equation*}%
and define the following mapping%
\begin{equation}
\Sigma \left( u\right) \left( t\right) =e^{-tA_{\Theta ,\mu
}}u_{0}-\int_{0}^{t}e^{-\left( t-s\right) A_{\Theta ,\mu }}f\left( u\left(
s\right) \right) ds,\text{ }t\in \left[ 0,T^{\ast }\right] .  \label{mapping}
\end{equation}%
Note that $\mathcal{X}_{T^{\ast },R^{\ast }}$, when endowed with the norm of
$C\left( \left[ 0,T^{\ast }\right] ;L^{\infty }\left( \Omega \right) \right)
,$ is a closed subset of $C\left( \left[ 0,T^{\ast }\right] ;L^{\infty
}\left( \Omega \right) \right) ,$ and since $f$ is continuously
differentiable, $\Sigma \left( u\right) \left( t\right) $ is continuous on $%
\left[ 0,T^{\ast }\right] $. We will show that, by properly choosing $%
T^{\ast },R^{\ast }>0$, $\Sigma :\mathcal{X}_{T^{\ast },R^{\ast
}}\rightarrow \mathcal{X}_{T^{\ast },R^{\ast }}$ is a contraction mapping
with respect to the metric induced by the norm of $C\left( \left[ 0,T^{\ast }%
\right] ;L^{\infty }\left( \Omega \right) \right) .$ The appropriate choices
for $T^{\ast },R^{\ast }>0$ will be specified below. First, we show that $%
u\in \mathcal{X}_{T^{\ast },R^{\ast }}$ implies that $\Sigma \left( u\right)
\in \mathcal{X}_{T^{\ast },R^{\ast }}$, that is, $\Sigma $ maps $\mathcal{X}%
_{T^{\ast },R^{\ast }}$ to itself. From (\ref{nonexp}) and the fact that $%
f\in C_{\text{loc}}^{1}\left( \mathbb{R}\right) $, we observe that the
mapping $\Sigma $\ satisfies the following estimate%
\begin{align*}
\left\Vert \Sigma \left( u\left( t\right) \right) \right\Vert _{\infty
,\Omega }& \leq \left\Vert u_{0}\right\Vert _{\infty ,\Omega
}+\int_{0}^{t}\left\Vert e^{-\left( t-s\right) A_{\Theta ,\mu }}\left(
f\left( 0\right) +\left( f\left( u\left( s\right) \right) -f\left( 0\right)
\right) \right) \right\Vert _{\infty ,\Omega }ds \\
& \leq \left\Vert u_{0}\right\Vert _{\infty ,\Omega }+t\left( \left\vert
f\left( 0\right) \right\vert +Q_{f^{^{\prime }}}\left( R^{\ast }\right)
R^{\ast }\right) ,
\end{align*}%
for some positive continuous function $Q_{f^{^{\prime }}}$ which depends
only on the size of the nonlinearity $f^{^{\prime }}$. Thus, provided that
we set $R^{\ast }\geq 2\left\Vert u_{0}\right\Vert _{\infty ,\Omega }$, we
can find a sufficiently small time $T^{\ast }>0$ such that
\begin{equation}
2T^{\ast }\left( \left\vert f\left( 0\right) \right\vert +Q_{f^{^{\prime
}}}\left( R^{\ast }\right) R^{\ast }\right) \leq R^{\ast },  \label{3.21time}
\end{equation}%
in which case $\Sigma \left( u\left( t\right) \right) \in \mathcal{X}%
_{T^{\ast },R^{\ast }}$, for any $u\left( t\right) \in \mathcal{X}_{T^{\ast
},R^{\ast }}.$ Next, we show that by possibly choosing $T^{\ast }>0$
smaller, $\Sigma :\mathcal{X}_{T^{\ast },R^{\ast }}\rightarrow \mathcal{X}%
_{T^{\ast },R^{\ast }}$ is also a contraction. Indeed, for any $%
u_{1},u_{2}\in \mathcal{X}_{T^{\ast },R^{\ast }}$, exploiting again (\ref%
{nonexp}), we estimate%
\begin{align}
\left\Vert \Sigma \left( u_{1}\left( t\right) \right) -\Sigma \left(
u_{2}\left( t\right) \right) \right\Vert _{\infty ,\Omega }& \leq
Q_{f^{^{\prime }}}\left( R^{\ast }\right) \int_{0}^{t}\left\Vert e^{-\left(
t-s\right) A_{\Theta ,\mu }}(u_{1}\left( s\right) -u_{2})\left( s\right)
\right\Vert _{\infty ,\Omega }ds  \label{3.21} \\
& \leq tQ_{f^{^{\prime }}}\left( R^{\ast }\right) \left\Vert
u_{1}-u_{2}\right\Vert _{C\left( \left[ 0,T^{\ast }\right] ;L^{\infty
}\left( \Omega \right) \right) }.  \notag
\end{align}%
This shows that $\Sigma $ is a contraction on $\mathcal{X}_{T^{\ast
},R^{\ast }}$ (compare with (\ref{3.21time})) provided that
\begin{equation*}
\max \left\{ \frac{2T^{\ast }(|f(0)|+Q_{f^{\prime }}(R^{\ast })R^{\ast })}{%
R^{\ast }},T^{\ast }Q_{f^{\prime }}(R^{\ast })\right\} =\frac{2T^{\ast
}(|f(0)|+Q_{f^{\prime }}(R^{\ast }))R^{\ast }}{R^{\ast }}<1.
\end{equation*}%
Therefore, owing to the contraction mapping principle, we conclude that
problem (\ref{op_v}) has a unique local solution $u\in \mathcal{X}_{T^{\ast
},R^{\ast }}$. This solution can certainly be (uniquely) extended on a right
maximal time interval $[0,T_{\max })$, with $T_{\max }>0$ depending on $%
\left\Vert u_{0}\right\Vert _{\infty ,\Omega },$ such that, either $T_{\max
}=\infty $ or $T_{\max }<\infty $, in which case $\lim_{t\nearrow T_{\max
}}\left\Vert u\left( t\right) \right\Vert _{\infty ,\Omega }=\infty .$
Indeed, if $T_{\max }<\infty $ and the latter condition does not hold, we
can find a sequence $t_{n}\nearrow T_{\max }$ such that $\left\Vert u\left(
t_{n}\right) \right\Vert _{\infty ,\Omega }\leq C.$ This would allow us to
extend $u$ as a solution to Equation (\ref{op_v}) to an interval $%
[0,t_{n}+\delta ),$ for some $\delta >0$ independent of $n$. Hence $u$ can
be extended beyond $T_{\max }$ which contradicts the construction of $%
T_{\max }>0$. To conclude that the solution $u$ belongs to the class in
Definition \ref{strong}, let us further set $\mathcal{G}\left( t\right)
:=-f\left( u\left( t\right) \right) ,$ for $u\in C\left( [0,T_{\max }\right)
;L^{\infty }\left( \Omega \right) )$ and notice that $u$ is the
"generalized" solution of%
\begin{equation}
\partial _{t}u+A_{\Theta ,\mu }u=\mathcal{G}\left( t\right) ,\text{ }t\in
\lbrack 0,T_{\max }),  \label{3.44}
\end{equation}%
such that $u\left( 0\right) =u_{0}\in L^{\infty }\left( \Omega \right)
\subset L^{2}(\Omega )=\overline{D(A_{\Theta ,\mu })}.$ By Theorem \ref{m1}
(Appendix), the "generalized" solution $u$ has the additional regularity $%
\partial _{t}u\in L^{2}\left( \delta ,T_{\max };L^{2}\left( \Omega \right)
\right) ,$ which together with the facts that $u$ is continuous on $%
[0,T_{\max })$ and $f\in C_{\text{loc}}^{1}\left( \mathbb{R}\right) $, yield
\begin{equation}
\mathcal{G}\in W^{1,2}\left( \delta ,T_{\max };L^{2}\left( \Omega \right)
\right) \cap L^{\infty }\left( \delta ,T_{\max };L^{2}\left( \Omega \right)
\right) .  \label{3.45}
\end{equation}%
Thus, we can apply Theorem \ref{ap_reg_thm} and Corollary \ref{ap_reg_co}
(see Appendix), to deduce the following regularity properties for $u:$%
\begin{equation}
u\in L^{\infty }\left( [\delta ,T_{\max });D\left( A_{\Theta ,\mu }\right)
\right) \cap W^{1,\infty }\left( [\delta ,T_{\max });L^{2}\left( \Omega
\right) \right) ,  \label{3.46}
\end{equation}%
such that the solution $u$ is Lipschitz continuous on $[\delta ,T_{\max })$,
for every $\delta >0.$ Thus, we have obtained a locally-defined strong
solution in the sense of Definition \ref{strong}. As to the variational
equality in Definition \ref{weak}, we note that this equality is satisfied
even pointwise (in time $t\in \left( 0,T_{\max }\right) $) by the strong
solutions. Our final point is to show that $T_{\max }=\infty ,$ because of
condition (H1). This ensures that the strong solution constructed above is
also global.

\emph{Step 2} (Energy estimate) Let $m\geq 1$ and consider the function $%
E_{m}:\;(0,\infty )\rightarrow \lbrack 0,\infty )$ defined by $%
E_{m}(t):=\Vert u(t)\Vert _{m+1,\Omega }^{m+1}$. First, notice that $E_{m}$
is well-defined on $\left( 0,T_{\max }\right) $ because $u$ is bounded in $%
\Omega \times (0,T_{\max })$ and because $\Omega $ has finite measure. Since
$u$ is a strong solution on $\left( 0,T_{\max }\right) ,$ see Definition \ref%
{strong} (or (\ref{3.46})), recall that $u$ is continuous from $[0,T_{\max
})\rightarrow L^{\infty }(\Omega )$ and Lipschitz continuous on $[\delta
,T_{\max })$ for every $\delta >0$. Thus, $u$ (as function of $t$) is
differentiable a.e., whence, the function $E_{m}(t)$ is also differentiable
for a.e. $t\in \left( 0,T_{\max }\right) .$

For strong solutions and $t\in \left( 0,T_{\max }\right) $, integration by
parts procedure yields the following standard energy identity:%
\begin{equation}
\frac{1}{2}\frac{d}{dt}E_{1}\left( t\right) +\mathcal{A}_{\Theta ,\mu
}\left( u\left( t\right) ,u\left( t\right) \right) +\int_{\Omega }f\left(
u\left( t\right) \right) u\left( t\right) dx=0.  \label{3.1}
\end{equation}%
Assumption (H1) implies that%
\begin{equation}
f\left( s\right) s\geq -\lambda _{\ast }s^{2}-C_{f},  \label{H1bis}
\end{equation}%
for some $C_{f}>0$ and for all $s\in \mathbb{R}$. This inequality allows us
to estimate the nonlinear term in (\ref{3.1}). We have (by using the
equivalent norm \eqref{eq-norm1})
\begin{equation}
\frac{1}{2}\frac{d}{dt}E_{1}\left( t\right) +C_{\Omega }\left\Vert u\left(
t\right) \right\Vert _{\mathcal{V}_{\Theta }^{1,2}(\Omega ,\partial \Omega
,\mu )}^{2}\leq C_{f}\left\vert \Omega \right\vert +\lambda _{\ast
}E_{1}\left( t\right) ,  \label{3.2}
\end{equation}%
where $\left\vert \Omega \right\vert $ denotes the $N$-dimensional Lebesgue
measure of $\Omega .$ In view of \eqref{eq-norm1} and Gronwall's inequality,
(\ref{3.2}) gives the following estimate for $t\in \left( 0,T_{\max }\right)
,$%
\begin{equation}
\left\Vert u\left( t\right) \right\Vert _{2,\Omega }^{2}+2\left( C_{\Omega
}-\lambda _{\ast }\right) \int_{0}^{t}\left\Vert u\left( s\right)
\right\Vert _{\mathcal{V}_{\Theta }^{1,2}(\Omega ,\partial \Omega ,\mu
)}^{2}ds\leq \left\Vert u_{0}\right\Vert _{2,\Omega }^{2}e^{-\rho t}+C\left(
f,\left\vert \Omega \right\vert \right) ,  \label{3.3}
\end{equation}%
for some constants $\rho =\rho \left( N,\Omega \right) >0$, $C\left(
f,\left\vert \Omega \right\vert \right) >0.$

\emph{Step 3} (The iteration argument). In this step, $c>0$ will denote a
constant that is independent of $t,$ $T_{\max }$, $m,$ $k$ and initial data,
which only depends on the other structural parameters of the problem. Such a
constant may vary even from line to line. Moreover, we shall denote by $%
Q_{\tau }\left( m\right) $ a monotone nondecreasing function in $m$ of order
$\tau ,$ for some nonnegative constant $\tau ,$ independent of $m.$ More
precisely, $Q_{\tau }\left( m\right) \sim cm^{\tau }$ as $m\rightarrow
+\infty .$ We begin by showing that $E_{m}\left( t\right) $ satisfies a
local recursive relation which can be used to perform an iterative argument.
Again standard procedure for the strong solutions on $\left( 0,T_{\max
}\right) $ (i.e., testing the variational equation (\ref{de_form}) with $%
\left\vert u\right\vert ^{m-1}u,$ $m\geq 1;$ see, e.g., \cite[Lemma 9.3.1]%
{CD}, \cite[Theorem 3.2]{G} or \cite[Theorem 3.1]{Ali}) gives on account of (%
\ref{H1bis}) 
the following inequality:%
\begin{eqnarray}
&&\frac{d}{dt}E_{m}\left( t\right) +\gamma \int_{\Omega }\left\vert \nabla
\left\vert u\right\vert ^{\frac{m+1}{2}}\right\vert ^{2}dx+\int_{\Gamma
_{\mu }}\left\vert u\right\vert ^{m+1}d\mu  \label{eqn1} \\
&\leq &Q_{1}\left( m+1\right) \left( E_{m}\left( t\right) +\left(
E_{m}\left( t\right) \right) ^{\frac{m}{m+1}}\right) ,  \notag
\end{eqnarray}%
for some constant $\gamma >0$ independent of $m,t$ and $T_{\max }$. Next,
set $m_{k}+1=2^{k},$ $k\in \mathbb{N}$, and define%
\begin{equation}
M_{k}:=\sup_{t\in \left( 0,T_{\max }\right) }\int_{\Omega }\left\vert
u\left( t,x\right) \right\vert ^{2^{k}}dx=\sup_{t\in \left( 0,T_{\max
}\right) }E_{m_{k}}\left( t\right) .  \label{def}
\end{equation}%
Our goal is to derive a recursive inequality for $M_{k}$ using (\ref{eqn1}).
In order to do so, we define%
\begin{equation*}
\overline{p}_{k}:=\frac{m_{k}-m_{k-1}}{q\left( 1+m_{k}\right) -\left(
1+m_{k-1}\right) }=\frac{1}{2q-1}<1,\text{ }\overline{q}_{k}:=1-\overline{p}%
_{k}=2\frac{q-1}{2q-1}
\end{equation*}%
since $q>1.$ We aim to estimate the terms on the right-hand side of (\ref%
{eqn1}) in terms of the $L^{1+m_{k-1}}\left( \Omega \right) $-norm of $u.$
First, H\"{o}lder inequality and Maz'ya inequality (i.e., $%
W_{2,2}^{1}(\Omega ,\partial \Omega ,\mu )\subset L^{2q}\left( \Omega
\right) $, for some $q=q\left( N,\Omega \right) >1$, see (\ref{M2-N}) with $%
q=N/\left( N-1\right) $ and $\mu =\sigma ,$ and (\ref{CI-mu}) when $\mu \neq
\sigma $) yield (by using the equivalent norm \eqref{eq-norm2})%
\begin{align}
\int_{\Omega }\left\vert u\right\vert ^{1+m_{k}}dx& \leq \left( \int_{\Omega
}\left\vert u\right\vert ^{\left( 1+m_{k}\right) q}dx\right) ^{\overline{p}%
_{k}}\left( \int_{\Omega }\left\vert u\right\vert ^{1+m_{k-1}}dx\right) ^{%
\overline{q}_{k}}  \label{ee5} \\
& \leq c\left( \int_{\Omega }\left\vert \nabla \left\vert u\right\vert ^{%
\frac{\left( 1+m_{k}\right) }{2}}\right\vert ^{2}dx+\int_{\Gamma _{\mu
}}\left\vert u\right\vert ^{1+m_{k}}d\mu \right) ^{\overline{s}_{k}}\left(
\int_{\Omega }\left\vert u\right\vert ^{1+m_{k-1}}dx\right) ^{\overline{q}%
_{k}},  \notag
\end{align}%
with $\overline{s}_{k}=\overline{p}_{k}q\equiv q/\left( 2q-1\right) \in
\left( 0,1\right) $. Applying Young's inequality on the right-hand side of (%
\ref{ee5}), we get for every $\varepsilon >0,$%
\begin{eqnarray}
Q_{1}\left( m_{k}+1\right) \int_{\Omega }\left\vert u\right\vert
^{1+m_{k}}dx &\leq &\varepsilon \int_{\Omega }\left\vert \nabla \left\vert
u\right\vert ^{\frac{1+m_{k}}{2}}\right\vert ^{2}dx+\varepsilon \int_{\Gamma
_{\mu }}\left\vert u\right\vert ^{1+m_{k}}d\mu  \label{ee5bis} \\
&&+Q_{\alpha }\left( m_{k}+1\right) \left( \int_{\Omega }\left\vert
u\right\vert ^{1+m_{k-1}}dx\right) ^{\max \left\{ 2,z_{k}\right\} },  \notag
\end{eqnarray}%
for some $\alpha >0$ independent of $k$, where in fact $z_{k}:=\overline{q}%
_{k}/\left( 1-\overline{s}_{k}\right) \equiv 2$, for each $k$, since $%
m_{k}+1\equiv 2^{k}$. In order to estimate the last term on the right-hand
side of (\ref{eqn1}), we define two decreasing and increasing sequences $%
\left( l_{k}\right) _{k\in \mathbb{N}}$ and $\left( w_{k}\right) _{k\in
\mathbb{N}},$ respectively, such that%
\begin{equation*}
l_{k}=\frac{m_{k}+1}{\overline{s}_{k}m_{k}}\text{ and }w_{k}:=\frac{%
\overline{q}_{k}m_{k}}{m_{k}\left( 1-\overline{s}_{k}\right) +1},
\end{equation*}%
and observe that they satisfy%
\begin{equation*}
1<l_{k}\leq 2\left( 2-\frac{1}{q}\right) ,\frac{2\left( q-1\right) }{3q-2}%
\leq w_{k}\leq 2
\end{equation*}%
for all $k\in \mathbb{N}$ (in particular, $w_{k}\rightarrow 2$ as $%
k\rightarrow \infty $). The application of Young's inequality in (\ref%
{ee5bis}) yields again%
\begin{eqnarray}
Q_{1}\left( m_{k}+1\right) \left( \int_{\Omega }\left\vert u\right\vert
^{1+m_{k}}dx\right) ^{\frac{m_{k}}{m_{k}+1}} &\leq &\varepsilon \left(
\int_{\Omega }\left\vert \nabla \left\vert u\right\vert ^{\frac{1+m_{k}}{2}%
}\right\vert ^{2}dx+\int_{\Gamma _{\mu }}\left\vert u\right\vert
^{1+m_{k}}d\mu \right)  \label{ee5tris} \\
&&+Q_{\beta _{k}}\left( m_{k}+1\right) \left( \int_{\Omega }\left\vert
u\right\vert ^{1+m_{k-1}}dx\right) ^{w_{k}},  \notag
\end{eqnarray}%
for every $\varepsilon >0$, where now%
\begin{equation*}
Q_{\beta _{k}}\left( m_{k}+1\right) \sim \frac{c}{\varepsilon ^{1/\left(
l_{k}-1\right) }}\left( 1+m_{k}\right) ^{\beta _{k}}
\end{equation*}%
with $\beta _{k}:=\left( m_{k}+1\right) /\left( m_{k}\left( 1-\overline{s}%
_{k}\right) +1\right) \rightarrow \left( 2q-1\right) /\left( q-1\right) $,
as $k\rightarrow \infty $. Hence, inserting (\ref{ee5bis}), (\ref{ee5tris})
into equation (\ref{eqn1}), choosing a sufficiently small $0<\varepsilon
\leq \varepsilon _{0}:=\frac{1}{4}\min \left( \gamma ,1\right) $, and
simplifying, we obtain for $t\in \left( 0,T_{\max }\right) ,$%
\begin{align}
& \frac{d}{dt}\int_{\Omega }\left\vert u\left( t,x\right) \right\vert
^{1+m_{k}}dx+\frac{\varepsilon _{0}}{2}\left( \int_{\Omega }\left\vert
\nabla \left\vert u\right\vert ^{\frac{1+m_{k}}{2}}\right\vert
^{2}dx+\int_{\Gamma _{\mu }}\left\vert u\right\vert ^{m_{k}+1}d\mu \right)
\label{ee6} \\
& \leq Q_{\delta }\left( m_{k}+1\right) \left( \int_{\Omega }\left\vert
u\right\vert ^{1+m_{k-1}}dx\right) ^{2},  \notag
\end{align}%
for some positive constant $\delta >0$ independent of $k.$

Next, we have $\left\vert u\right\vert ^{\frac{1+m_{k}}{2}}\in
W_{2,2}^{1}\left( \Omega ,\partial \Omega ,\mu \right) $ (recall that $%
u\left( t\right) \in W_{2,2}^{1}\left( \Omega ,\partial \Omega ,\mu \right)
\cap L^{\infty }\left( \Omega \right) ,$ for a.e. $t\in \left( 0,T_{\max
}\right) $), so that by using the equivalent norm \eqref{eq-norm2}, we can
find another positive constant $c=c\left( \Omega ,N\right) $, depending on $%
\Omega $ and $N,$ such that%
\begin{equation}
\int_{\Omega }\left\vert \nabla \left\vert u\right\vert ^{\frac{1+m_{k}}{2}%
}\right\vert ^{2}dx+\int_{\Gamma _{\mu }}\left\vert u\right\vert
^{m_{k}+1}d\mu \geq c\int_{\Omega }\left\vert u\right\vert ^{m_{k}+1}dx.
\label{ee6bis}
\end{equation}%
We can now combine (\ref{ee6bis}) with (\ref{ee6}) to deduce%
\begin{equation}
\frac{d}{dt}\int_{\Omega }\left\vert u\left( t,x\right) \right\vert
^{2^{k}}dx+\frac{c\varepsilon _{0}}{2}\int_{\Omega }\left\vert u\left(
t,x\right) \right\vert ^{2^{k}}dx\leq Q_{\delta }\left( 2^{k}\right)
M_{k-1}^{2},  \label{ee7bis}
\end{equation}%
for $t\in \left( 0,T_{\max }\right) .$ Integrating (\ref{ee7bis}) over $%
\left( 0,t\right) $, we infer from Gronwall-Bernoulli's inequality \cite[%
Lemma 1.2.4]{CD} that there exists yet another constant $c>0,$ independent
of $k$, such that%
\begin{equation}
M_{k}\leq \max \left\{ \int_{\Omega }\left\vert u_{0}\right\vert
^{2^{k}}dx,c2^{k\delta }M_{k-1}^{2}\right\} ,\text{ for all }k\geq 2.
\label{claim2}
\end{equation}%
On the other hand, let us observe that there exists a positive constant $%
C_{\infty }=C_{\infty }(\left\Vert u_{0}\right\Vert _{\infty ,\Omega })\geq
1,$ independent of $k$, such that $\left\Vert u_{0}\right\Vert
_{2^{k},\Omega }\leq C_{\infty }$. Taking the $2^{k}$-th root on both sides
of (\ref{claim2}), and defining $X_{k}:=\sup_{t\in \left( 0,T_{\max }\right)
}\left\Vert u\left( t\right) \right\Vert _{2^{k},\Omega },$ we easily arrive
at%
\begin{equation}
X_{k}\leq \max \left\{ C_{\infty },\left( c2^{\delta k}\right) ^{\frac{1}{%
2^{k}}}X_{k-1}\right\} ,\text{ for all }k\geq 2.  \label{ee7}
\end{equation}%
By straightforward induction in (\ref{ee7}) (see \cite[Lemma 3.2]{Ali}; cf.
also \cite[Lemma 9.3.1]{CD}), we finally obtain the estimate%
\begin{equation}
\sup_{t\in \left( 0,T_{\max }\right) }\left\Vert u\left( t\right)
\right\Vert _{\infty ,\Omega }\leq \lim_{k\rightarrow +\infty }X_{k}\leq
c\max \left\{ C_{\infty },\sup_{t\in \left( 0,T_{\max }\right) }\left\Vert
u\left( t\right) \right\Vert _{2,\Omega }\right\} .  \label{ee8}
\end{equation}%
It remains to notice that (\ref{ee8}) together with the bound (\ref{3.3})
shows that $\left\Vert u\left( t\right) \right\Vert _{\infty ,\Omega }$ is
uniformly bounded for all times $t>0$ with a bound, independent of $T_{\max
},$ depending only on $\left\Vert u_{0}\right\Vert _{\infty ,\Omega }$, the
"size" of the domain and the non-linear function $f.$ This gives $T_{\max
}=+\infty $ so that strong solutions are in fact global. This completes the
proof of the theorem.
\end{proof}

\begin{remark}
\label{improved_reg}\emph{Strong solutions to Eqs. (\ref{1.1})--(\ref{1.3})
exhibit an improved regularity in time, we have}%
\begin{equation}
u\in C\left( \left[ 0,T\right] ;L^{\infty }\left( \Omega \right) \right)
\cap C((\delta ,T];\mathcal{V}_{\Theta }^{1,2}(\Omega ,\partial \Omega ,\mu
)),  \label{imp_reg}
\end{equation}%
\emph{for any }$T>\delta >0.$ \emph{This follows from the fact that the
nonlinear function }$f$\emph{\ is continuously differentiable. Note that the
second regularity in (\ref{imp_reg}) is a consequence of the first one, the
time regularity in (\ref{time_reg}) (see Definition \ref{strong}) and the
variational identity (}$\emph{\ref{de_form}}$\emph{).}
\end{remark}

The following result is immediate.

\begin{corollary}
\label{dyn_sys_str}Let the assumptions of Theorem \ref{SG} be satisfied. The
reaction-diffusion system (\ref{1.1})-(\ref{1.3}) defines a (nonlinear)
continuous semigroup $\mathcal{T}\left( t\right) :L^{\infty }\left( \Omega
\right) \rightarrow L^{\infty }\left( \Omega \right) $, given by
\begin{equation}
\mathcal{T}\left( t\right) u_{0}=u\left( t\right) ,  \label{3.7}
\end{equation}%
where $u$ is the (unique)\ strong solution in the sense of Definition \ref%
{strong}.
\end{corollary}

\subsection{Weak Solutions}

We aim to prove the existence of weak solutions in the sense of Definition %
\ref{weak}. For this, Theorem \ref{SG} proves essential in the sense that we
can proceed by an approximation argument, which we briefly describe below.
We consider, for each $\epsilon >0$, the following system%
\begin{equation}
\left\{
\begin{array}{ll}
\partial _{t}u_{\epsilon }-\Delta u_{\epsilon }+f\left( u_{\epsilon }\right)
=0, & \text{in }\Omega \times \left( 0,\infty \right) , \\
\partial _{\nu }u_{\epsilon }d\sigma +(u_{\epsilon }+\Theta _{\mu }\left(
u_{\epsilon }\right) )d\mu =0, & \text{on }\partial \Omega \times \left(
0,\infty \right) ,%
\end{array}%
\right.  \label{ap1}
\end{equation}%
subject to the initial condition%
\begin{equation}
u_{\epsilon }\left( 0\right) =u_{0\epsilon },  \label{ap2}
\end{equation}%
where $u_{0\epsilon }\in L^{\infty }\left( \Omega \right) \cap \mathcal{V}%
_{\Theta }^{1,2}(\Omega ,\partial \Omega ,\mu )$ such that%
\begin{equation}
u_{0\epsilon }\rightarrow u_{0}=u\left( 0\right) \text{ in }L^{2}\left(
\Omega \right) .  \label{ap3}
\end{equation}%
Then, by Theorem \ref{SG}, under the assumption that $f$ satisfies (H1), the
approximate problem (\ref{ap1})-(\ref{ap2}) admits a unique strong solution
with%
\begin{equation}
u_{\epsilon }\in W_{\text{loc}}^{1,\infty }((0,T_{\ast }];L^{2}(\Omega
))\cap C\left( \left[ 0,T_{\ast }\right] ;L^{\infty }\left( \Omega \right)
\right) \cap L_{\text{loc}}^{\infty }\left( (0,T_{\ast }];D\left( A_{\Theta
,\mu }\right) \right) ,  \label{smooth}
\end{equation}%
for some $T_{\ast }>0$ and each $\epsilon >0$, such that both (\ref{de_form}%
) and (\ref{en_id}) are satisfied even pointwise in time $t\in \left(
0,T_{\ast }\right) $. The advantage of this construction is that now every
weak solution can be approximated by the strong ones and the rigorous
justification of our estimates for such solutions is immediate.

We shall now deduce the first result concerning the solvability of problem (%
\ref{1.1})-(\ref{1.3}) with the new assumption (H2). Note that a function
that satisfies (H2) also satisfies the assumption (H1).

\begin{theorem}
\label{T1}Assume that the nonlinearity $f$\ obeys (H2) and $\mu $ satisfies
(H$_{\mu }$). Then, for any initial data $u_{0}\in L^{2}\left( \Omega
\right) ,$ there exists at least one (globally-defined) weak solution
\begin{equation*}
u\in C\left( \left[ 0,T\right] ;L^{2}\left( \Omega \right) \right)
\end{equation*}%
in the sense of Definition \ref{weak}.
\end{theorem}

\begin{proof}
We divide the proof into two main steps. For practical purposes, $C$ will
denote a positive constant that is independent of time, $T_{\ast }$, $%
\epsilon >0$ and initial data, but which only depends on the other
structural parameters. Such a constant may vary even from line to line.

\emph{Step 1.} (The main energy estimate). Since $f$ obeys (H2), then (H1)
is satisfied and by Theorem \ref{SG}, the approximate problem (\ref{ap1})-(%
\ref{ap2}) has a strong solution that also satisfies the weak formulation (%
\ref{de_form}). Thus, in view of (\ref{smooth}) the key choice $\xi
=u_{\epsilon }\left( t\right) $ in (\ref{de_form}) is justified. We have the
following energy identity%
\begin{equation}
\frac{1}{2}\frac{d}{dt}\left\Vert u_{\epsilon }\left( t\right) \right\Vert
_{2,\Omega }^{2}+\mathcal{A}_{\Theta ,\mu }\left( u_{\epsilon }\left(
t\right) ,u_{\epsilon }\left( t\right) \right) +\int_{\Omega }f\left(
u_{\epsilon }\left( t\right) \right) u_{\epsilon }\left( t\right) dx=0,
\notag
\end{equation}%
for all $t\in (0,T_{\ast }).$ Invoking assumption (H2), we infer%
\begin{equation}
\frac{1}{2}\frac{d}{dt}\left\Vert u_{\epsilon }\left( t\right) \right\Vert
_{2,\Omega }^{2}+C_{\Omega }\left\Vert u_{\epsilon }\left( t\right)
\right\Vert _{\mathcal{V}_{\Theta }^{1,2}(\Omega ,\partial \Omega ,\mu
)}^{2}+\left\Vert u_{\epsilon }\left( t\right) \right\Vert _{p,\Omega
}^{p}\leq C\left\vert \Omega \right\vert  \label{3.5}
\end{equation}%
where we have used the equivalent norm \eqref{eq-norm1}. We can now
integrate this inequality over $\left( 0,T_{\ast }\right) $ to deduce%
\begin{equation}
\left\Vert u_{\epsilon }\left( t\right) \right\Vert _{2,\Omega
}^{2}+\int_{0}^{t}\left( C_{\Omega }\left\Vert u_{\epsilon }\left( s\right)
\right\Vert _{\mathcal{V}_{\Theta }^{1,2}(\Omega ,\partial \Omega ,\mu
)}^{2}+C_{f}\left\Vert u_{\epsilon }\left( s\right) \right\Vert _{p,\Omega
}^{p}\right) ds\leq \left\Vert u_{\epsilon }\left( 0\right) \right\Vert
_{2,\Omega }^{2}e^{-\rho t}+C  \label{3.6}
\end{equation}%
for all $t\in \left( 0,T_{\ast }\right) ,$ for some $\rho >0$ independent of
$\epsilon >0$. Incidently, by (\ref{ap3}) this uniform estimate also shows
that we can take $T_{\ast }=\infty ,$ i.e., the weak solution that we
construct is in fact globally-defined. On account of (\ref{3.6}), we deduce
the following uniform (in $\epsilon >0$) bounds%
\begin{align}
u_{\epsilon }& \in L^{\infty }\left( 0,T;L^{2}\left( \Omega \right) \right)
\cap L^{p}\left( 0,T;L^{p}\left( \Omega \right) \right) ,  \label{est1} \\
u_{\epsilon }& \in L^{2}(0,T;\mathcal{V}_{\Theta }^{1,2}(\Omega ,\partial
\Omega ,\mu )),  \notag
\end{align}%
for any $T>0.$ Hence, by virtue of \eqref{op-star} and (\ref{est1}), we also
get%
\begin{equation}
A_{\Theta ,\mu }u_{\epsilon }\in L^{2}(0,T;(\mathcal{V}_{\Theta
}^{1,2}(\Omega ,\partial \Omega ,\mu ))^{\ast }),\text{ }f\left( u_{\epsilon
}\right) \in L^{p^{^{\prime }}}(0,T;L^{p^{^{\prime }}}\left( \Omega \right)
),  \label{est2}
\end{equation}%
uniformly in $\epsilon >0$. Here, recall that $A_{\Theta ,\mu }$ is the
positive self-adjoint operator associated with the bilinear form $\mathcal{A}%
_{\Theta ,\mu }$ (see Section 2.2). By \eqref{op-star}, we can interpret (%
\ref{de_form}) as the equality%
\begin{equation}
-\partial _{t}u_{\epsilon }=A_{\Theta ,\mu }u_{\epsilon }+f\left(
u_{\epsilon }\right)  \label{fform}
\end{equation}%
in $L^{2}(0,T;(\mathcal{V}_{\Theta }^{1,2}(\Omega ,\partial \Omega ,\mu
))^{\ast })+L^{p^{^{\prime }}}(0,T;L^{p^{^{\prime }}}\left( \Omega \right)
), $ where $p^{^{\prime }}=p/\left( p-1\right) $ (see also (\ref{op_v})).

\emph{Step 2.} (Passage to limit). From the above properties (\ref{est1})-(%
\ref{est2}), we see that there exists a subsequence $\left\{ u_{\epsilon
}\right\} _{\epsilon >0}$ (still denoted by $\left\{ u_{\epsilon }\right\} $%
), such that as $\epsilon \rightarrow 0^{+}$,%
\begin{equation}
\begin{array}{l}
u_{\epsilon }\rightarrow u\text{ weakly star in }L^{\infty }\left(
0,T;L^{2}\left( \Omega \right) \right) , \\
u_{\epsilon }\rightarrow u\text{ weakly in }L^{p}\left( 0,T;L^{p}\left(
\Omega \right) \right) , \\
\partial _{t}u_{\epsilon }\rightarrow \partial _{t}u\text{ weakly in }%
L^{2}(0,T;(\mathcal{V}_{\Theta }^{1,2}(\Omega ,\partial \Omega ,\mu ))^{\ast
})+L^{p^{^{\prime }}}(0,T;L^{p^{^{\prime }}}\left( \Omega \right) ).%
\end{array}
\label{2.27}
\end{equation}%
Since the continuous embedding $\mathcal{V}_{\Theta }^{1,2}(\Omega ,\partial
\Omega ,\mu )\hookrightarrow L^{2}\left( \Omega \right) $ is compact, then
we can exploit standard embedding results for vector valued functions (see,
e.g., \cite{CV}), to deduce%
\begin{equation}
u_{\epsilon }\rightarrow u\text{ strongly in }L^{2}\left( 0,T;L^{2}\left(
\Omega \right) \right) .  \label{2.30}
\end{equation}%
By refining in \eqref{2.30}, $u_{\epsilon }$ converges to $u$ a.e. in $%
\Omega \times \left( 0,T\right) $. Then, by means of known results in
measure theory \cite{CV}, the continuity of $f$ and the convergence of %
\eqref{2.30} imply that $f\left( u_{\epsilon }\right) $ converges weakly to $%
f\left( u\right) $ in $L^{p^{^{\prime }}}(\left( 0,T\right) \times \Omega )$%
, while from (\ref{est1})-(\ref{est2}) and the linearity of $A_{\Theta ,\mu
} $, we further see that%
\begin{equation}
A_{\Theta ,\mu }u_{\epsilon }\rightarrow A_{\Theta ,\mu }u\text{ weakly in }%
L^{2}(0,T;(\mathcal{V}_{\Theta }^{1,2}(\Omega ,\partial \Omega ,\mu ))^{\ast
}).  \label{2.32}
\end{equation}%
We can now pass to the limit as $\epsilon \rightarrow 0$ in equation (\ref%
{fform}) to deduce the desired weak solution $u$, satisfying the variational
identity (\ref{de_form}) and the regularity properties (\ref{reg_weak}). In
order to show the energy identity (\ref{en_id}), we can proceed in a
standard way as in \cite[Theorem II.1.8]{CV}, by observing that any
distributional derivative $\partial _{t}u\left( t\right) $ from $\mathcal{D}%
^{^{\prime }}(\left[ 0,T\right] ;(\mathcal{V}_{\Theta }^{1,2}(\Omega
,\partial \Omega ,\mu ))^{\ast }+ L^{p^{^{\prime }}}\left( \Omega \right) )$
can be represented as $\partial _{t}u\left( t\right) =\mathcal{Z}_{1}\left(
t\right) +\mathcal{Z}_{2}\left( t\right) ,$ where%
\begin{equation*}
\mathcal{Z}_{1}\left( t\right) :=-A_{\Theta ,\mu }u\left( t\right) \in
L^{2}(0,T;(\mathcal{V}_{\Theta }^{1,2}(\Omega ,\partial \Omega ,\mu ))^{\ast
}),\text{ }\mathcal{Z}_{2}\left( t\right) :=-f\left( u\left( t\right)
\right) \in L^{p^{^{\prime }}}(0,T;L^{p^{^{\prime }}}\left( \Omega \right) ),
\end{equation*}%
which are precisely the dual of the spaces $L^{2}(0,T;(\mathcal{V}_{\Theta
}^{1,2}(\Omega ,\partial \Omega ,\mu )))$, and $L^{p}(0,T;L^{p}\left( \Omega
\right) )$, respectively. In particular, we obtain that every weak solution $%
u\in C\left( \left[ 0,T\right] ;L^{2}\left( \Omega \right) \right) ,$ and
that the map $t\mapsto \left\Vert u\left( t\right) \right\Vert _{2,\Omega
}^{2}$ is absolutely continuous on $\left[ 0,T\right] $, such that $u$
satisfies the energy identity%
\begin{equation}
\frac{1}{2}\frac{d}{dt}\left\Vert u\left( t\right) \right\Vert _{2,\Omega
}^{2}+\left\langle A_{\Theta ,\mu }u\left( t\right) ,u\left( t\right)
\right\rangle +\left\langle f\left( u\left( t\right) \right) ,u\left(
t\right) \right\rangle =0,  \label{en_id2}
\end{equation}%
a.e. $t\in \left[ 0,T\right] $, whence (\ref{en_id}) follows. The proof of
Theorem \ref{T1} is finished.
\end{proof}

As in the classical case \cite{CV, R, T}, we can prove the following
stability result for the class of weak solutions constructed by means of
Definition \ref{weak}.

\begin{proposition}
\label{uniq}Let the assumptions of Theorem \ref{T1} be satisfied, and in
addition, assume that (H3) holds. Then, there exists a unique weak solution
to problem (\ref{1.1})-(\ref{1.3}), which depends continuously on the
initial data in a Lipschitz way.
\end{proposition}

\begin{proof}
As usual, consider any two weak solutions $u_{1},u_{2},$ and set $u\left(
t\right) =u_{1}\left( t\right) -u_{2}\left( t\right) $. Then, according to (%
\ref{en_id}) (cf. also (\ref{en_id2})), $u\left( t\right) $ satisfies the
identity%
\begin{equation*}
\frac{1}{2}\frac{d}{dt}\left\Vert u\left( t\right) \right\Vert _{2,\Omega
}^{2}+\mathcal{A}_{\Theta ,\mu }\left( u\left( t\right) ,u\left( t\right)
\right) =-\left\langle f\left( u_{1}\left( t\right) \right) -f\left(
u_{2}\left( t\right) \right) ,u\left( t\right) \right\rangle ,
\end{equation*}%
a.e. $t\in \left[ 0,T\right] .$ Assumption (H3) implies (by using the
equivalent norm \eqref{eq-norm1})
\begin{equation}
\frac{d}{dt}\left\Vert u\left( t\right) \right\Vert _{2,\Omega
}^{2}+C\left\Vert u\left( t\right) \right\Vert _{\mathcal{V}%
_{\Theta}^{1,2}(\Omega ,\partial \Omega ,\mu )}^{2}\leq 2C_{f}\left\Vert
u\left( t\right) \right\Vert _{2,\Omega }^{2}.  \label{3.8}
\end{equation}%
This yields the desired result by application of Gronwall's inequality.
\end{proof}

As an immediate consequence of this stability result, problem (\ref{1.1})-(%
\ref{1.3}) defines a dynamical system in the classical sense \cite{BV, CV,
R, T}.

\begin{corollary}
\label{dyn_system}Let the assumptions of Proposition \ref{uniq} be
satisfied. The reaction-diffusion system (\ref{1.1})-(\ref{1.3}) defines a
(nonlinear) continuous semigroup $\mathcal{S}\left( t\right) :L^{2}\left(
\Omega \right) \rightarrow L^{2}\left( \Omega \right) $, given by
\begin{equation}
\mathcal{S}\left( t\right) u_{0}=u\left( t\right) ,  \label{2.15}
\end{equation}%
where $u$ is the (unique)\ weak solution in the sense of Definition \ref%
{weak}.
\end{corollary}

\subsection{Regularity of solutions}

The next proposition is a direct consequence of estimate (\ref{ee7bis}) of
Theorem \ref{SG}. Strong solutions guaranteed by Theorem \ref{SG} provide
sufficient regularity to justify all the calculations performed in the proof
of Proposition \ref{bounded} below. In this case, at the very end we pass to
the limit and obtain the estimate even for the generalized solutions.

\begin{proposition}
\label{bounded}Let the assumptions of Theorem \ref{SG} or Theorem \ref{T1}
be satisfied. Let $\tau ^{^{\prime }}>\tau >0$ and fix $\mu :=\tau
^{^{\prime }}-\tau $. There exists a positive constant $C=C\left( \mu
\right) \sim \mu ^{-\eta }$ (for some $\eta >0$), independent of $t$ and the
initial data, such that%
\begin{equation}
\sup_{t\geq \tau ^{^{\prime }}}\left\Vert u\left( t\right) \right\Vert
_{\infty ,\Omega }\leq C\sup_{s\geq \tau }\left\Vert u\left( s\right)
\right\Vert _{2,\Omega }.  \label{sup-b}
\end{equation}
\end{proposition}

\begin{proof}
The argument leading to (\ref{sup-b}) is analogous to \cite[Theorem 2.3]{G}
(cf. also \cite[Theorem 2.3]{GalNS}). It is based on the following recursive
inequality for $E_{m_{k}}\left( t\right) $, which is a consequence of (\ref%
{ee7bis}) and (\ref{ee5})-(\ref{ee6}):%
\begin{equation}
\sup_{t\geq t_{k-1}}E_{m_{k}}\left( t\right) \leq C\left( 2^{k}\right)
^{\delta }\left( \sup_{s\geq t_{k}}E_{m_{k-1}}\left( s\right) \right) ^{2},%
\text{ \ for all }k\geq 1,  \label{rec_rel}
\end{equation}%
where the sequence $\left\{ t_{k}\right\} _{k\in \mathbb{N}}$ is defined
recursively $t_{k}=t_{k-1}-\mu /2^{k},$ $k\geq 1$, $t_{0}=\tau ^{^{\prime
}}. $ Here we recall that $C=C\left( \mu \right) >0,$ $\delta >0$ are
independent of $k.$ For the sake of completeness, we report a sketch of the
argument for (\ref{rec_rel}). To this end, let $\zeta \left( s\right) $ be a
positive function $\zeta :\mathbb{R}_{+}\rightarrow \left[ 0,1\right] $ such
that $\zeta \left( s\right) =0$ for $s\in \left[ 0,t-\mu /2^{k}\right] ,$ $%
\zeta \left( s\right) =1$ if $s\in \left[ t,+\infty \right) $ and $%
\left\vert d\zeta /ds\right\vert \leq 2^{k}/\mu $, if $s\in \left( t-\mu
/2^{k},t\right) $. We define $Z_{k}\left( s\right) =\zeta \left( s\right)
E_{m_{k}}\left( s\right) $ and notice that%
\begin{align}
\frac{d}{ds}Z_{k}\left( s\right) & \leq \zeta \left( s\right) \frac{d}{ds}%
E_{m_{k}}\left( s\right) +\frac{2^{k}}{\mu }E_{m_{k}}\left( s\right)
\label{e11bis} \\
& =\zeta \left( s\right) \frac{d}{ds}E_{m_{k}}\left( s\right) +Q_{1}\left(
2^{k}\right) \int_{\Omega }\left\vert u\right\vert ^{1+m_{k}}dx.  \notag
\end{align}%
The last integral in (\ref{e11bis}) can be estimated as in (\ref{ee5bis})
and (\ref{ee6bis}). Combining the above estimates and the fact that $%
Z_{k}\leq E_{m_{k}}$, we deduce the following inequality:%
\begin{equation}
\frac{d}{ds}Z_{k}\left( s\right) +C2^{k}Z_{k}\left( s\right) \leq C\left(
2^{k}\right) ^{\sigma }(\sup_{s\geq t-\mu /2^{k}}E_{m_{k-1}}\left( s\right)
)^{2},\text{ for all }s\in \left[ t-\mu /2^{k},+\infty \right) .  \label{e12}
\end{equation}%
Note that $C=C\left( \mu \right) \sim \mu ^{-1}$ as $\mu \rightarrow 0$, and
$C\left( \mu \right) $ is bounded if $\mu $ is bounded away from zero.
Integrating (\ref{e12}) with respect to $s$ from $t-\mu /2^{k}$ to $t,$ and
taking into account the fact that $Z_{k}\left( t-\mu /2^{k}\right) =0,$ we
obtain that $E_{m_{k}}\left( t\right) =Z_{k}\left( t\right) \leq C\left(
2^{k}\right) ^{\sigma }(\sup_{s\geq t-\mu /2^{k}}E_{m_{k-1}}\left( s\right)
)^{2}\left( 1-e^{-C\mu }\right) $, which proves the claim (\ref{rec_rel}).
Thus, we can iterate in (\ref{rec_rel}) with respect to $k\geq 1$ and obtain
that%
\begin{align}
\sup_{t\geq t_{k-1}}E_{m_{k}}\left( t\right) & \leq \left( C\left(
2^{k}\right) ^{\delta }\right) \left( C\left( 2^{k-1}\right) ^{\delta
}\right) ^{2}\cdot ...\cdot \left( C\left( 2\right) ^{\delta }\right)
^{2^{k}}(\sup_{s\geq \tau }\left\Vert u\left( s\right) \right\Vert
_{2,\Omega })^{2^{k}}  \label{3.9} \\
& \leq C^{\left( 2^{k}\sum_{i=1}^{k}\frac{1}{2^{i}}\right) }2^{\left( \delta
2^{k}\sum_{i=1}^{k}\frac{i}{2^{i}}\right) }(\sup_{s\geq \tau }\left\Vert
u\left( s\right) \right\Vert _{2,\Omega })^{2^{k}},  \notag
\end{align}%
Therefore, we can take the $1+m_{k}=2^{k}$-th root on both sides of (\ref%
{3.9}) and let $k\rightarrow +\infty $. Using the facts that $\zeta
:=\sum_{i=1}^{\infty }\frac{1}{2^{i}}<\infty ,$ $\overline{\zeta }%
:=\sum_{i=1}^{\infty }\frac{i}{2^{i}}<\infty $, we deduce%
\begin{equation}
\sup_{t\geq t_{0}=\tau ^{^{\prime }}}\left\Vert u\left( t\right) \right\Vert
_{\infty ,\Omega }\leq \lim_{k\rightarrow +\infty }\sup_{t\geq t_{0}}\left(
E_{m_{k}}\left( t\right) \right) ^{1/\left( 1+m_{k}\right) }\leq C^{\zeta
}2^{\delta \overline{\zeta }}\sup_{s\geq \tau }\left\Vert u\left( s\right)
\right\Vert _{2,\Omega }.  \notag
\end{equation}%
This clearly proves the proposition.
\end{proof}

Combining estimate (\ref{3.6}) (which is also satisfied by the "generalized"
solution) with Proposition \ref{bounded}, and arguing as in the proof of
Theorem \ref{SG}, we obtain the following.

\begin{corollary}
\label{weak-strong}Let the assumptions of Proposition \ref{uniq} be
satisfied. For every $u_{0}\in L^{2}\left( \Omega \right) ,$ the (unique)
orbit $u\left( t\right) =\mathcal{S}\left( t\right) u_{0}$ is a strong
solution for $t\geq \rho ,$ for all $\rho >0.$
\end{corollary}

Next, since we already know that there exists an absorbing set $\mathcal{B}%
\subset L^{2}\left( \Omega \right) $ (cf. Ineq. (\ref{3.6})) for the
dynamical system $\left( \mathcal{S}\left( t\right) ,L^{2}\left( \Omega
\right) \right) $, it will be important to show that the semigroup $\mathcal{%
S}$ is also asymptotically smooth. This is essential in the attractor theory
and the related results we shall present in the next section. Clearly, we
want to make use of the fact that the weak solutions are sufficiently smooth
on the interval $[\rho ,\infty )$, for any $\rho >0$. It will suffice to
derive a uniform bound in $\mathcal{V}_{\Theta }^{1,2}(\Omega ,\partial
\Omega ,\mu ).$ Since the argument relies on the use of key test functions
(more precisely, we need to take $\xi =\partial _{t}u\left( t\right) $ into
the variational equation (\ref{de_form})), we will actually need to require
more regularity of the strong solution, i.e., $u\in W_{\text{loc}%
}^{1,s}((0,\infty );\mathcal{V}_{\Theta }^{1,2}(\Omega ,\partial \Omega ,\mu
)),$ $s>1$. However, lacking any further knowledge on these solutions (note
that $\Omega $ is an arbitrary bounded open set), we will work with
"truncated" solutions which can be obtained with the Galerkin approach and
has independent interest.

\begin{lemma}
\label{h1est}Let the assumptions of Proposition \ref{uniq} be satisfied.
Then, for $u_{0}\in L^{2}\left( \Omega \right) $ any orbit $u\left( t\right)
=\mathcal{S}\left( t\right) u_{0}$ of (\ref{1.1})-(\ref{1.3}) satisfies%
\begin{equation*}
u\in L^{\infty }([\rho ,\infty );\mathcal{V}_{\Theta }^{1,2}(\Omega
,\partial \Omega ,\mu ))\cap W^{1,2}\left( [\rho ,\infty );L^{2}\left(
\Omega \right) \right) ,
\end{equation*}%
for every $\rho >0$, and the following estimate holds:%
\begin{equation}
\sup_{t\geq \rho }\left( \left\Vert u\left( t\right) \right\Vert _{\mathcal{V%
}_{\Theta }^{1,2}(\Omega ,\partial \Omega ,\mu )}^{2}+\int_{0}^{t}\left\Vert
\partial _{t}u\left( s\right) \right\Vert _{2,\Omega }^{2}ds\right) \leq
C_{\delta },  \label{as_sm}
\end{equation}%
for some constant $C_{\delta }=C\left( \rho \right) >0,$ independent of $t$
and initial data.
\end{lemma}

\begin{proof}
We recall that by Theorem \ref{3prop}, $A_{\Theta ,\mu }$ is a positive and
self-adjoint operator in $L^{2}\left( \Omega \right) $. Then, we have a
complete system of eigenfunctions $\left\{ \xi _{i}\right\} _{i\in \mathbb{N}%
}$ for the operator $A_{\Theta ,\mu }$ in $L^{2}\left( \Omega \right) $ with
$\xi _{i}\in D\left( A_{\Theta ,\mu }\right) \subset \mathcal{V}_{\Theta
}^{1,2}(\Omega ,\partial \Omega ,\mu ),$ and $\xi _{i}\in L^{\infty }\left(
\Omega \right) .$ According to the general spectral theory, the eigenvalues $%
\lambda _{i}$ can be increasingly ordered and counted according to their
multiplicities in order to form a real divergent sequence. Moreover, the
respective eigenvectors $\xi _{i}$ turn out to form an orthogonal basis in $%
\mathcal{V}_{\Theta }^{1,2}(\Omega ,\partial \Omega ,\mu )$ and $L^{2}\left(
\Omega \right) ,$ respectively. The eigenvectors $\xi _{i}$ may be assumed
to be normalized in $L^{2}\left( \Omega \right) $. To this end, we can now
define the finite-dimensional spaces
\begin{equation*}
\mathcal{P}_{n}=span\left\{ \xi _{1},\xi _{2},...,\xi _{n}\right\} ,\quad
\mathcal{P}_{\infty }=\cup _{n=1}^{\infty }\mathcal{P}_{n}.
\end{equation*}%
Clearly, $\mathcal{P}_{\infty }$ is a dense subspace of $\mathcal{V}_{\Theta
}^{1,2}(\Omega ,\partial \Omega ,\mu )$. As usual, for any $n\in \mathbb{N}$%
, we look for functions of the form%
\begin{equation}
u_{n}(t)=\sum_{i=1}^{n}e_{i}\left( t\right) \xi _{i}  \label{Proj}
\end{equation}%
solving a suitable approximating problem. More precisely, for any given $%
n\geq 1$ we look for $C^{1}$-real functions $e_{i}\left( \cdot \right) $, $%
i=1,\dots ,n$, only depending on time, which solve the approximating problem
$\mathbf{P}\left( n\right) $ given by
\begin{equation}
\left( \partial _{t}u_{n},\Psi \right) _{2,\Omega }+\left( A_{\Theta ,\mu
}u_{n},\Psi \right) _{2,\Omega }+\left( f\left( u_{n}\right) ,\Psi \right)
_{2,\Omega }=0,  \label{Approx_p}
\end{equation}%
with initial condition%
\begin{equation}
\left\langle u_{n}\left( 0\right) ,\Psi \right\rangle _{2,\Omega
}=\left\langle u_{n0},\Psi \right\rangle _{2,\Omega },  \label{Approx_IC}
\end{equation}%
for all $\Psi \in \mathcal{P}_{n}.$ Here $u_{n0}$ is the orthogonal
projection of $u_{0}$ onto $\mathcal{P}_{n}.$ Observe that
\begin{equation}
\underset{n\rightarrow \infty }{\lim }u_{n0}=u_{0},\quad \text{ in }%
L^{2}\left( \Omega \right) \text{.}  \label{IC_limit}
\end{equation}

Notice that by the Cauchy-Lipschitz theorem, one can find a unique maximal
solution%
\begin{equation*}
u_{n}\in C^{1}([0,T_{n});D\left( A_{\Theta ,\mu }\right) \cap L^{\infty
}\left( \Omega \right) )
\end{equation*}%
to (\ref{Approx_p})-(\ref{Approx_IC}), for some $T_{n}\in (0,T)$. As in the
case when $\Omega $ is smooth (see the monographs \cite{BV, CV, R, T}), the
existence of a generalized solution, defined on the whole interval $[0,T],$
for every $n\in \mathbb{N}$, can be also obtained with this approach. In
particular, notice that (\ref{3.6}) is also satisfied by $u_{n}$. Here we
are interested to derive the bound (\ref{as_sm}). Notice that the key choice
of the test function $\xi =\partial _{t}u_{n}\in $ $C([0,T];\mathcal{V}%
_{\Theta }^{1,2}(\Omega ,\partial \Omega ,\mu )\cap L^{p}\left( \Omega
\right) )$ into the variational equation (\ref{de_form}) is now allowed for
these truncated solutions. We deduce%
\begin{equation*}
\frac{d}{dt}\left( \left\Vert u_{n}\left( t\right) \right\Vert _{\mathcal{V}%
_{\Theta }^{1,2}(\Omega ,\partial \Omega ,\mu )}^{2}+\left( F\left(
u_{n}\left( t\right) \right) ,1\right) _{2,\Omega }\right) +2\left\Vert
\partial _{t}u_{n}\left( t\right) \right\Vert _{2,\Omega }^{2}=0,
\end{equation*}%
for all $t\geq 0.$ Here and below, $F$ denotes the primitive of $f$, i.e., $%
F\left( s\right) =\int_{0}^{s}f\left( y\right) dy.$ Multiply this equation
by $t\geq \rho >0$ and integrate over $\left( 0,t\right) $ to get%
\begin{align*}
& t\left( \left\Vert u_{n}\left( t\right) \right\Vert _{\mathcal{V}_{\Theta
}^{1,2}(\Omega ,\partial \Omega ,\mu )}^{2}+\left( F\left( u_{n}\left(
t\right) \right) ,1\right) _{2,\Omega }\right) +2\int_{0}^{t}s\left\Vert
\partial _{t}u_{n}\left( s\right) \right\Vert _{2,\Omega }^{2}ds \\
& =\int_{0}^{t}\left( \left\Vert u_{n}\left( s\right) \right\Vert _{\mathcal{%
V}_{\Theta }^{1,2}(\Omega ,\partial \Omega ,\mu )}^{2}+\left( F\left(
u_{n}\left( s\right) \right) ,1\right) _{2,\Omega }\right) ds,
\end{align*}%
for all $t\geq \rho .$ Recalling that, due to (H2)-(H3), $F$ is bounded from
below, independently of $n$, and $\left\vert F\left( s\right) \right\vert
\leq C\left( 1+\left\vert s\right\vert ^{p}\right) $, we infer from (\ref%
{3.6}),%
\begin{equation}
\left\Vert u_{n}\left( t\right) \right\Vert _{\mathcal{V}_{\Theta
}^{1,2}(\Omega ,\partial \Omega ,\mu )}^{2}+2\int_{0}^{t}\left\Vert \partial
_{t}u_{n}\left( s\right) \right\Vert _{2,\Omega }^{2}ds\leq c\left( 1+\frac{1%
}{t}\right) ,  \label{as_sm2}
\end{equation}%
for some constant $c>0$ independent of $t,n$ and $\rho .$ On the basis of a
lower-semicontinuity argument, we easily obtain the desired estimate (\ref%
{as_sm}). The proof is finished.
\end{proof}

Notice that the uniqueness of the weak solutions (see Proposition \ref{uniq}%
) is important here. We will eventually need to use the fact that each such
solution $u$ belongs to $L^{\infty }([\rho ,\infty );L^{\infty }\left(
\Omega \right) \cap \mathcal{V}_{\Theta }^{1,2}(\Omega ,\partial \Omega ,\mu
)),$ for every $\rho >0,$ in order to apply a sequence of abstracts results
in the next section and obtain the desired "finite-dimensional" description
of the long-term dynamics of (\ref{1.1})-(\ref{1.3}) in terms of attractors
which possess good properties. Indeed, even though the weak solution of
Definition \ref{weak} can be also constructed with aid from the Galerkin
approach presented above, the application of this scheme seems quite
problematic for the proof of Proposition \ref{bounded} (see also the proof
of Theorem \ref{SG}) and it cannot be longer used in that context. In other
words, the procedure of approximating by the strong solutions $u_{\epsilon }$
gives a weak solution $\widetilde{u},$ while the usual limit procedure in
the Galerkin truncations $u_{n}$ gives another weak solution $\widehat{u}$.
The solution $\widehat{u}$ coincides with $\widetilde{u},$ if uniqueness is
known for (\ref{1.1})-(\ref{1.3}). Here is where the two theories seem to
depart if the latter property is not known.

\section{Finite dimensional attractors}

\label{gl}

The present section is focused on the long-term analysis of problem (\ref%
{1.1})-(\ref{1.3}). We proceed to investigate the asymptotic properties of (%
\ref{1.1})-(\ref{1.3}), using the notion of a global attractor. We begin
with the following.

\begin{definition}
\label{gl_notion}A set $\mathcal{G}_{\Theta ,\mu }\subset L^{2}\left( \Omega
\right) $ is a global attractor of the semigroup $\mathcal{S}\left( t\right)
$ on $L^2(\Omega)$ associated with (\ref{1.1})-(\ref{1.3}) if

\begin{itemize}
\item $\mathcal{G}_{\Theta ,\mu }$ is compact in $L^{2}\left( \Omega \right)
$;

\item $\mathcal{G}_{\Theta ,\mu }$ is strictly invariant, that is, $\mathcal{%
S}(t)\mathcal{G}_{\Theta ,\mu }=\mathcal{G}_{\Theta ,\mu }\;\forall\; t\ge 0$%
;

\item $\mathcal{G}_{\Theta ,\mu }$ attracts the images of all bounded
subsets of $L^{2}\left( \Omega \right) $, namely, for every bounded subset $%
B $ of $L^{2}\left( \Omega \right) $ and every neighborhood $\mathcal{O}$ of
$\mathcal{G}_{\Theta ,\mu }$ in the topology of $L^{2}\left( \Omega \right) $%
, there exists a constant $T=T(B;\mathcal{O})>0$ such that $\mathcal{S}%
(t)B\subset \mathcal{O}$, for every $t\geq T$.
\end{itemize}
\end{definition}

The next proposition gives the existence of such an attractor.

\begin{proposition}
\label{glatt}Let the assumptions of Proposition \ref{uniq} be satisfied. The
semigroup $\mathcal{S}\left( t\right)$ on $L^2(\Omega)$ associated with the
reaction-diffusion system (\ref{1.1})-(\ref{1.3}) possesses a global
attractor $\mathcal{G}_{\Theta ,\mu }$ in the sense of Definition \ref%
{gl_notion}. As usual, this attractor is generated by all complete bounded
trajectories of (\ref{1.1})-(\ref{1.3}), that is, $\mathcal{G}_{\Theta ,\mu
}=\mathcal{K}_{\mid t=0}$, where $\mathcal{K}$ is the set of all strong
solutions $u$ which are defined for all $t\in \mathbb{R}_+$ and bounded in
the $L^{\infty }\left( \Omega \right) \cap \mathcal{V}_{\Theta
}^{1,2}(\Omega ,\partial \Omega ,\mu )$-norm.
\end{proposition}

\begin{proof}
Due to the dissipative estimates (\ref{3.6}), (\ref{sup-b}) and (\ref{as_sm}%
), the ball%
\begin{equation*}
\mathcal{B}_{0}=\left\{ u\in \mathcal{X}:=L^{\infty }\left( \Omega \right)
\cap \mathcal{V}_{\Theta }^{1,2}(\Omega ,\partial \Omega ,\mu ):\left\Vert
u\right\Vert _{\mathcal{X}}\leq R\right\} ,
\end{equation*}%
for a sufficiently large radius $R>0,$ is an absorbing set for $\mathcal{S}%
\left( t\right) $ in $L^{2}\left( \Omega \right) $. Indeed, in light of (\ref%
{3.6}), it is not difficult to see that, for any bounded set $B\subset
L^{2}\left( \Omega \right) $, there exists a time $t_{\ast }=t_{\ast }(B)>0$
such that $\mathcal{S}\left( t\right) B\subset L^{2}\left( \Omega \right) $,
for all $t\geq t_{\ast }$. Next, we can choose $\tau ^{^{\prime }}=\tau
+2\mu $ with $\tau =t_{\ast }$ and $\mu =1$, so that the $L^{2}$-$L^{\infty
} $ smoothing property (\ref{sup-b}) together with the estimate (\ref{as_sm}%
) entails the desired assertion. Obviously, the ball $\mathcal{B}_{0}$ is
compact in the topology of $L^{2}\left( \Omega \right) $. Thus, $\mathcal{S}%
\left( t\right) $ possesses a compact absorbing set. On the other hand, due
to Proposition \ref{uniq} (see also (\ref{diffbis}) below), for every fixed $%
t\geq 0$, the map $\mathcal{S}\left( t\right) $ is continuous on $\mathcal{B}%
_{0}$ in the $L^{2}$-topology and, consequently, the existence of the global
attractor follows now from the classical attractor's existence theorem (see,
e.g. \cite{CV}).
\end{proof}

Let us now construct a Lyapunov functional for (\ref{1.1})-(\ref{1.3}).

\begin{lemma}
\label{L2}Make the assumptions of Corollary \ref{weak-strong} (or Theorem %
\ref{SG}). Then the functional%
\begin{equation*}
\mathcal{L}_{\Theta }\left( u(t)\right) :=\frac{1}{2}\left\Vert u\left(
t\right) \right\Vert _{\mathcal{V}_{\Theta }^{1,2}(\Omega ,\partial \Omega
,\mu )}^{2}+\left( F\left( u\left( t\right) \right) ,1\right) _{2,\Omega },
\end{equation*}%
has along the strong solutions of (\ref{1.1})-(\ref{1.3}), the derivative%
\begin{equation*}
\frac{d}{dt}\mathcal{L}_{\Theta }\left( u\left( t\right) \right)
=-\int_{\Omega }\left\vert \partial _{t}u\left( t\right) \right\vert ^{2}dx,%
\text{ }a.e.\text{ }t>0.
\end{equation*}%
In other words, the functional $\mathcal{L}_{\Theta }$ is decreasing, and
becomes stationary exactly on equilibria $u_{\ast }$, which are solutions of
the system:%
\begin{equation}
-\Delta u+f\left( u\right) =0\text{ in }\Omega ,\text{ }\partial _{\nu
}ud\sigma +(u+\Theta_\mu \left( u\right) )d\mu =0\text{ on }\partial \Omega .
\label{equi}
\end{equation}
\end{lemma}

\begin{proof}
The proof is a simple calculation which relies essentially on the fact that
strong solutions are smooth enough, see Definition \ref{strong} and Remark %
\ref{improved_reg}.
\end{proof}

Corollary \ref{weak-strong} together with Lemma \ref{L2} can now be used to
study the asymptotic behavior of the solutions of (\ref{1.1})-(\ref{1.3}) by
means of the LaSalle's invariance principle (see, e.g., \cite{R, T}). To
this end, to any (weak) trajectory of Eqns. (\ref{1.1})-(\ref{1.3}) we
associate the respective (positive) $\omega $-limit set $\Lambda
_{L^{2}}^{+} $:%
\begin{equation*}
\Lambda _{L^{2}}^{+}:=\left\{ y\in L^{2}\left( \Omega \right) :\exists
t_{n}\rightarrow \infty ,\text{ }y_{n}\in L^{2}\left( \Omega \right) \text{
such that }\mathcal{S}\left( t_{n}\right) y_{n}\rightarrow y\text{ in }L^{2}%
\text{-topology}\right\} .
\end{equation*}

The following lemma states some basic properties of the $\omega $-limit sets
(independent of any Lyapunov function). Its proof is immediate owing to the
continuity properties of the strong solutions (see Definition \ref{strong},
Corollary \ref{weak-strong} and Remark \ref{improved_reg}) and the
compactness of the embedding $\mathcal{V}_{\Theta }^{1,2}(\Omega ,\partial
\Omega ,\mu ) \hookrightarrow L^{2}\left( \Omega \right) .$

\begin{lemma}
\begin{enumerate}
\item[(i)] Any $\omega$-limit set $\Lambda _{L^{2}}^{+}$ is nonempty,
compact and connected.

\item[(ii)] The trajectory approaches its own limit set in the norm of $%
L^{2}\left( \Omega \right) $, i.e.,
\begin{equation*}
\lim_{t\rightarrow \infty }dist_{L^{2}\left( \Omega \right) }\left( \mathcal{%
S}\left( t\right) u_{0},\Lambda _{L^{2}}^{+}\right) =0.
\end{equation*}

\item[(iii)] Any $\omega $-limit set is invariant: new trajectories which
start at some point in $\Lambda _{L^{2}}^{+}$ remain in $\Lambda
_{L^{2}}^{+} $ for all times $t>0$.
\end{enumerate}
\end{lemma}

The dynamical system $\left( \mathcal{S}\left( t\right) ,L^{2}\left( \Omega
\right) \right) $ is a "gradient" system, namely, we have the following
(see, e.g., \cite[Theorem 10.13]{R}).

\begin{theorem}
\label{Lf}Let the assumptions of Proposition \ref{uniq} be satisfied. The
global attractor $\mathcal{G}_{\Theta ,\mu }$ coincides with the unstable
set of equilibria $E_{\ast },$ which consists of solutions of (\ref{equi}).
\end{theorem}

Thus, the asymptotic behavior of solutions of (\ref{1.1})-(\ref{1.3}) is
properly described by the existence of the global attractor $\mathcal{G}%
_{\Theta ,\mu }$. However, the global attractor does not provide an actual
control of the convergence rate of trajectories and might be unstable with
respect to perturbations. A more suitable object to have an effective
control on the longtime dynamics is the exponential attractor (see e.g.\cite%
{MZrev}). Contrary to the global one, the exponential attractor is not
unique (thus, in some sense, is an artificial object), and it is only
semi-invariant. However, it has the advantage of being stable with respect
to perturbations, and it provides an exponential convergence rate which can
be explicitly computed. Our construction of an exponential attractor is
based on the following abstract result \cite[Proposition 4.1]{EZ}.

\begin{proposition}
\label{abstract}Let $\mathcal{H}$,$\mathcal{V}$,$\mathcal{V}_{1}$ be Banach
spaces such that the embedding $\mathcal{V}_{1}\hookrightarrow \mathcal{V}$
is compact. Let $\mathbb{B}$ be a closed bounded subset of $\mathcal{H}$ and
let $\mathbb{S}:\mathbb{B}\rightarrow \mathbb{B}$ be a map. Assume also that
there exists a uniformly Lipschitz continuous map $\mathbb{T}:\mathbb{B}%
\rightarrow \mathcal{V}_{1}$, i.e.,%
\begin{equation}
\left\Vert \mathbb{T}b_{1}-\mathbb{T}b_{2}\right\Vert _{\mathcal{V}_{1}}\leq
L\left\Vert b_{1}-b_{2}\right\Vert _{\mathcal{H}},\quad \forall
b_{1},b_{2}\in \mathbb{B},  \label{gl1}
\end{equation}%
for some $L\geq 0$, such that%
\begin{equation}
\left\Vert \mathbb{S}b_{1}-\mathbb{S}b_{2}\right\Vert _{\mathcal{H}}\leq
\gamma \left\Vert b_{1}-b_{2}\right\Vert _{\mathcal{H}}+K\left\Vert \mathbb{T%
}b_{1}-\mathbb{T}b_{2}\right\Vert _{\mathcal{V}},\quad \forall
b_{1},b_{2}\in \mathbb{B},  \label{gl2}
\end{equation}%
for some constant $0\le \gamma <\frac{1}{2}$ and $K\geq 0$. Then, there
exists a (discrete) exponential attractor $\mathcal{M}_{d}\subset \mathbb{B}$
of the semigroup $\{\mathbb{S}(n):=\mathbb{S}^{n},n\in \mathbb{Z}_{+}\}$
with discrete time in the phase space $\mathcal{H}$, which satisfies the
following properties:

\begin{itemize}
\item semi-invariance: $\mathbb{S}\left( \mathcal{M}_{d}\right) \subset
\mathcal{M}_{d}$;

\item compactness: $\mathcal{M}_{d}$ is compact in $\mathcal{H}$;

\item exponential attraction: $dist_{\mathcal{H}}(\mathbb{S}^{n}\mathbb{B},%
\mathcal{M}_{d})\leq Ce^{-\alpha n},$ for all $n\in \mathbb{N}$ and for some
$\alpha >0$ and $C\geq 0$, where $dist_{\mathcal{H}}$ denotes the standard
Hausdorff semidistance between sets in $\mathcal{H}$;

\item finite-dimensionality: $\mathcal{M}_{d}$ has finite fractal dimension
in $\mathcal{H}$.
\end{itemize}
\end{proposition}

Moreover, the constants $C$ and $\alpha ,$ and the fractal dimension of $%
\mathcal{M}_{d}$ can be explicitly expressed in terms of $L$, $K$, $\gamma $%
, $\left\Vert \mathbb{B}\right\Vert _{\mathcal{H}}$ and Kolmogorov's $\kappa
$-entropy of the compact embedding $\mathcal{V}_{1}\hookrightarrow \mathcal{V%
},$ for some $\kappa =\kappa \left( L,K,\gamma \right) $. We recall that the
Kolmogorov $\kappa $-entropy of the compact embedding $\mathcal{V}%
_{1}\hookrightarrow\mathcal{V}$ is the logarithm of the minimum number of
balls of radius $\kappa $ in $\mathcal{V}$ necessary to cover the unit ball
of $\mathcal{V}_{1}$ (see, e.g., \cite{CV}).

We are now ready to state and prove

\begin{theorem}
\label{expo}Assume that the nonlinearity $f$\ obeys (H2)-(H3) and $\mu $
satisfies (H$_{\mu }$). The semigroup $\mathcal{S}\left( t\right) $ on $%
L^{2}(\Omega )$ associated with (\ref{1.1})-(\ref{1.3}) possesses an
exponential attractor $\mathcal{E}_{\Theta ,\mu }$ in the following sense:

\begin{itemize}
\item $\mathcal{E}_{\Theta ,\mu }$ is bounded in $L^{\infty }\left( \Omega
\right) \cap \mathcal{V}_{\Theta }^{1,2}(\Omega ,\partial \Omega ,\mu )$ and
compact in $L^{2}\left( \Omega \right) $;

\item $\mathcal{E}_{\Theta ,\mu }$ is semi-invariant: $\mathcal{S}\left(
t\right) \left( \mathcal{E}_{\Theta ,\mu }\right) \subset \mathcal{E}%
_{\Theta ,\mu }$, $t\geq 0$;

\item $\mathcal{E}_{\Theta ,\mu }$ attracts the images of bounded (in $%
L^{2}\left( \Omega \right) $) subsets exponentially in the metric of $%
L^{\infty }\left( \Omega \right) $, i.e. there exist $\beta >0$ and a
monotone function $Q$ such that, for every bounded set $B\subset L^{2}\left(
\Omega \right) $,%
\begin{equation*}
dist_{L^{\infty }\left( \Omega \right) }\left( \mathcal{S}\left( t\right) B,%
\mathcal{E}_{\Theta ,\mu }\right) \leq Q\left( \left\Vert B\right\Vert
_{L^{2}\left( \Omega \right) }\right) e^{-\beta t},\text{ for all }t\geq 0%
\text{;}
\end{equation*}

\item $\mathcal{E}_{\Theta ,\mu }$ has the finite fractal dimension in $%
L^{\infty }\left( \Omega \right) $.
\end{itemize}
\end{theorem}

Since an exponential attractor always contains the global one, the theorem
implies, in particular, the following result.

\begin{corollary}
\label{gl_fin}The fractal dimension of the global attractor $\mathcal{G}%
_{\Theta ,\mu }$ of Proposition \ref{glatt} is finite in $L^{\infty }\left(
\Omega \right) $. Moreover, we have%
\begin{equation*}
\lim_{t\rightarrow +\infty }dist_{L^{\infty }\left( \Omega \right) }\left(
\mathcal{S}\left( t\right) B,\mathcal{G}_{\Theta ,\mu }\right) =0,
\end{equation*}%
for every bounded set $B\subset L^{2}\left( \Omega \right) .$
\end{corollary}

We first recall that, due to the proof of Proposition \ref{glatt}, the
semigroup $\mathcal{S}\left( t\right) $ possesses an absorbing ball $%
\mathcal{B}_{0}$ in the phase space $\mathcal{X}$. Thus, it suffices to
construct the exponential attractor for the restriction of this semigroup on
$\mathcal{B}_{0}$ only. In order to apply Proposition \ref{abstract} to our
situation, we need to verify the proper estimate for the difference of
solutions, which is done in the following lemma.

\begin{lemma}
\label{L5}Let the assumptions of Theorem \ref{expo} hold, and let $u_{1}$
and $u_{2}$ be two weak solutions of (\ref{1.1})-(\ref{1.3}) such that $%
u_{i}\left( 0\right) \in \mathcal{B}_{0}.$ Then the following estimates are
valid:%
\begin{equation}
\left\Vert u_{1}\left( t\right) -u_{2}\left( t\right) \right\Vert _{2,\Omega
}^{2}\leq M\left\Vert u_{1}\left( 0\right) -u_{2}\left( 0\right) \right\Vert
_{2,\Omega }^{2}e^{-\omega t}+K\left\Vert u_{1}-u_{2}\right\Vert
_{L^{2}\left( \left[ 0,t\right] ;L^{2}\left( \Omega \right) \right) }^{2},
\label{diff1}
\end{equation}%
and
\begin{equation}
\left\Vert \partial _{t}u_{1}-\partial _{t}u_{2}\right\Vert _{L^{2}(\left[
0,t\right] ;(\mathcal{V}_{\Theta }^{1,2}(\Omega ,\partial \Omega ,\mu
))^{\ast })}^{2}+\int_{0}^{t}\left\Vert u_{1}\left( s\right) -u_{2}\left(
s\right) \right\Vert _{\mathcal{V}_{\Theta }^{1,2}(\Omega ,\partial \Omega
,\mu )}^{2}ds\leq Ce^{\nu t}\left\Vert u_{1}\left( 0\right) -u_{2}\left(
0\right) \right\Vert _{2,\Omega }^{2},  \label{diff2}
\end{equation}%
for some $\omega ,\nu >0,$ $M,K,C\geq 0,$ all independent of $t$ and $u_{i}.$
\end{lemma}

\begin{proof}
The injection $\mathcal{V}_{\Theta }^{1,2}(\Omega ,\partial \Omega ,\mu
)\hookrightarrow L^{2}\left( \Omega \right) $ is compact and continuous. The
application of Gronwall's inequality in (\ref{3.8}) entails the desired
estimate (\ref{diff1}). The second term on the left-hand side of (\ref{diff2}%
) can be easily controlled by integration over $\left( 0,t\right) $ in (\ref%
{3.8}). More precisely, we have%
\begin{equation}
\left\Vert u_{1}\left( t\right) -u_{2}\left( t\right) \right\Vert _{2,\Omega
}^{2}+\int_{0}^{t}\left\Vert u_{1}\left( s\right) -u_{2}\left( s\right)
\right\Vert _{\mathcal{V}_{\Theta }^{1,2}(\Omega ,\partial \Omega ,\mu
)}^{2}ds\leq \left\Vert u_{1}\left( 0\right) -u_{2}\left( 0\right)
\right\Vert _{2,\Omega }^{2}e^{\nu t}.  \label{diffbis}
\end{equation}%
Furthermore, in light of Corollary \ref{weak-strong} and estimates (\ref%
{sup-b}), (\ref{h1est}), recall that we have%
\begin{equation}
\sup_{t\geq 0}\left\Vert u_{i}\left( t\right) \right\Vert _{L^{\infty
}\left( \Omega \right) \cap \mathcal{V}_{\Theta }^{1,2}(\Omega ,\partial
\Omega ,\mu )}\leq C,  \label{4.5}
\end{equation}%
for some positive constant $C$ independent of $t$ and $u_{i}$. Thus, for any
test function $\xi $ $\in \mathcal{V}_{\Theta }^{1,2}(\Omega ,\partial
\Omega ,\mu )$, using the variational identity (\ref{de_form}) (which
actually holds pointwise for $t\geq 0$), we have for the function $%
u:=u_{1}-u_{2},$%
\begin{align*}
\left\langle \partial _{t}u\left( t\right) ,\xi \right\rangle & =-\mathcal{A}%
_{\Theta ,\mu }\left( u\left( t\right) ,\xi \right) -\left\langle f\left(
u_{1}\left( t\right) \right) -f\left( u_{2}\left( t\right) \right) ,\xi
\right\rangle \\
& \leq C\left\Vert u\left( t\right) \right\Vert _{\mathcal{V}_{\Theta
}^{1,2}(\Omega ,\partial \Omega ,\mu )}\left\Vert \xi \right\Vert _{\mathcal{%
V}_{\Theta }^{1,2}(\Omega ,\partial \Omega ,\mu )},
\end{align*}%
since $f\in C_{\text{loc}}^{1}\left( \mathbb{R}\right) $, owing to (\ref{4.5}%
). This estimate together with (\ref{diffbis}) gives the desired control on
the time derivative in (\ref{diff2}).
\end{proof}

The last ingredient we need is the uniform H\"{o}lder continuity of the time
map $t\mapsto \mathcal{S}\left( t\right) u_{0}$ in the $L^{\infty }$-norm,
namely,

\begin{lemma}
\label{time_cont}Let the assumptions of Theorem \ref{expo} be satisfied.
Consider $u\left( t\right) =\mathcal{S}\left( t\right) u_{0}$ with $u_{0}\in
\mathcal{B}_{0}$. Then, for every $t\geq 0$, the following estimate holds:%
\begin{equation}
\left\Vert u\left( t\right) -u\left( s\right) \right\Vert _{\infty ,\Omega
}\leq C\left\vert t-s\right\vert ^{\lambda },\text{ for all }t,s\geq 0,
\label{4.6}
\end{equation}%
where $\lambda <1$, $C>0$ are independent of $t,s$, $u$ and the initial data.
\end{lemma}

\begin{proof}
According to (\ref{4.5}), the following bound holds for $u$:%
\begin{equation*}
\sup_{t\geq 0}\left\Vert u\left( t\right) \right\Vert _{L^{\infty }\left(
\Omega \right) \cap \mathcal{V}_{\Theta }^{1,2}(\Omega ,\partial \Omega ,\mu
)}\leq C.
\end{equation*}%
Consequently, by comparison in (\ref{de_form}) and from (\ref{as_sm}), we
have that%
\begin{equation*}
\sup_{t\geq 0}\left( \left\Vert \partial _{t}u\left( t\right) \right\Vert _{(%
\mathcal{V}_{\Theta }^{1,2}(\Omega ,\partial \Omega ,\mu ))^{\ast
}}+\left\Vert \partial _{t}u\right\Vert _{L^{2}(\left[ 0,t\right]
;L^{2}\left( \Omega \right) )}\right) \leq C
\end{equation*}%
which entails the inequality%
\begin{equation}
\left\Vert u\left( t\right) -u\left( s\right) \right\Vert _{2,\Omega }\leq
C\left\vert t-s\right\vert ^{1/2}\text{, for all }t,s\geq 0.  \label{4.6bis}
\end{equation}%
Inequality (\ref{4.6}) is a consequence of (\ref{4.6bis}) and the $L^{2}$-$%
L^{\infty }$ smoothing property (\ref{sup-b}). This follows from the fact
that the nonlinear function $f$\ is continuously differentiable. Indeed, due
to the boundedness of $u\left( t\right) \in L^{\infty }\left( \Omega \right)
,$ a.e. $t\geq 0$\ and $u\left( t\right) \in V_{\Theta }^{1,2}(\Omega
,\partial \Omega ,\mu )$, $a.e.$ $t\geq \delta >0$ (cf., Lemma \ref{h1est}),
the nonlinearity $f$\ becomes subordinated to the linear part of the
equation (\ref{1.1}) (no matter how fast it grows). More precisely,
obtaining the $L^{2}$-$L^{\infty }$\ continuous dependance estimate for the
difference $u\left( t\right) -u\left( s\right) $\ of any two strong
solutions $u\left( t\right) ,u\left( s\right) ,$\ corresponding to the same
initial datum, is actually reduced to the same iteration procedure we used
in the proof of Theorem \ref{SG}. The proof is completed.
\end{proof}

We can now finish the proof of the main theorem of this section, using the
abstract scheme of Proposition \ref{abstract}.


\begin{proof}[Proof of Theorem \protect\ref{expo}]
First, we construct the exponential attractor $\mathcal{M}_{d}$ of the
discrete map $\mathcal{S}\left( T^{\ast }\right) $ on $\mathcal{B}_{0}$ (the
above constructed absorbing ball in $L^{\infty }\left( \Omega \right) \cap
\mathcal{V}_{\Theta }^{1,2}(\Omega ,\partial \Omega ,\mu )$), for a
sufficiently large $T^{\ast }$. Indeed, let $B_{1}=\left[ \cup _{t\geq
T^{\ast }}\mathcal{S}\left( t\right) \mathcal{B}_{0}\right] _{L^{2}}$, where
$\left[ \cdot \right] _{L^{2}}$ denotes the closure in the space $%
L^{2}\left( \Omega \right) $ and then set $\mathbb{B}:=\mathcal{S}\left(
1\right) B_{1}$. Thus, $\mathbb{B}$ is a semi-invariant compact (for the $%
L^{2}$-metric) subset of the phase space $L^{2}(\Omega )$ and $\mathcal{S}%
\left( T^{\ast }\right) :\mathbb{B}\rightarrow \mathbb{B}$, provided that $%
T^{\ast }$ is large enough. Then, we apply Proposition \ref{abstract} on the
set $\mathbb{B}$ with $\mathcal{H}=L^{2}\left( \Omega \right) $ and $\mathbb{%
S}=\mathcal{S}\left( T^{\ast }\right) ,$ with $T^{\ast }>0$ large enough so
that $Me^{-\omega T^{\ast }}=\gamma <\frac{1}{2}$ (see (\ref{diff1})).
Besides, letting%
\begin{align*}
\mathcal{V}_{1}& =L^{2}(\left[ 0,T^{\ast }\right] ;\mathcal{V}_{\Theta
}^{1,2}(\Omega ,\partial \Omega ,\mu ))\cap W^{1,2}(\left[ 0,T^{\ast }\right]
;(\mathcal{V}_{\Theta }^{1,2}(\Omega ,\partial \Omega ,\mu ))^{\ast }), \\
\mathcal{V}& =L^{2}(\left[ 0,T^{\ast }\right] ;L^{2}\left( \Omega \right) ),
\end{align*}%
we have that $\mathcal{V}_{1}\hookrightarrow \mathcal{V}$ is compact, with
reference to Theorem \ref{3prop}. Secondly, define $\mathbb{T}:\mathbb{B}%
\rightarrow \mathcal{V}_{1}$ to be the solving operator for (\ref{1.1})-(\ref%
{1.3}) on the time interval $\left[ 0,T^{\ast }\right] $ such that $\mathbb{T%
}u_{0}:=u\in \mathcal{V}_{1},$ with $u\left( 0\right) =u_{0}\in \mathbb{B}.$
Due to Lemma \ref{L5}, (\ref{diff2}), we have the global Lipschitz
continuity (\ref{gl1}) of $\mathbb{T}$ from $\mathbb{B}$ to $\mathcal{V}_{1}$%
, and (\ref{diff1}) gives us the basic estimate (\ref{gl2}) for the map $%
\mathbb{S}=\mathcal{S}\left( T^{\ast }\right) $. Therefore, the assumptions
of Proposition \ref{abstract} are verified and, consequently, the map $%
\mathbb{S}=\mathcal{S}\left( T^{\ast }\right) $ possesses an exponential
attractor $\mathcal{M}_{d}$ on $\mathbb{B}$. In order to construct the
exponential attractor $\mathcal{E}_{\Theta ,\mu }$ for the semigroup $%
\mathcal{S}(t)$ with continuous time, we note that, due to Proposition \ref%
{uniq}, this semigroup is Lipschitz continuous with respect to the initial
data in the topology of $L^{2}\left( \Omega \right) $, see also (\ref%
{diffbis}). Moreover, by Lemma \ref{time_cont} the map $\left(
t,u_{0}\right) \mapsto \mathcal{S}\left( t\right) u_{0}$ is also uniformly H%
\"{o}lder continuous on $\left[ 0,T^{\ast }\right] \times \mathbb{B}$, where
$\mathbb{B}$ is endowed with the metric topology of $L^{2}\left( \Omega
\right) $ (actually, even with respect of the metric topology of $L^{\infty
}\left( \Omega \right) $ due to the $L^{2}$-$L^{\infty }$ smoothing property
of $\mathcal{S}$ and the fact that $f\in C_{\text{loc}}^{1}(\mathbb{R})$).
Hence, the desired exponential attractor $\mathcal{E}_{\Theta ,\mu }$ for
the continuous semigroup $\mathcal{S}(t)$ can be obtained by the standard
formula%
\begin{equation}
\mathcal{E}_{\Theta ,\mu }=\bigcup_{t\in \left[ 0,T^{\ast }\right] }\mathcal{%
S}\left( t\right) \mathcal{M}_{d}.  \label{st}
\end{equation}%
Finally, the finite-dimensionality of $\mathcal{E}_{\Theta ,\mu }$ follows
from the finite dimensionality of $\mathcal{M}_{d}$ and the $L^{2}$-$%
(L^{\infty }\cap \mathcal{V}_{\Theta }^{1,2})$ smoothing property of the
semigroup $\mathcal{S}\left( t\right) $. The remaining properties of $%
\mathcal{E}_{\Theta ,\mu }$ are immediate. Theorem \ref{expo} is now proved.
\end{proof}

We conclude with the following result for the dynamical system $\left(
\mathcal{T}\left( t\right) ,L^{\infty }\left( \Omega \right) \right) ,$
associated with strong solutions of problem (\ref{1.1})-(\ref{1.3}) (see
Definition \ref{strong}).

\begin{theorem}
\label{expo2}Assume that the nonlinearity $f$\ obeys (H1) and $\mu $
satisfies (H$_{\mu }$). The semigroup $\mathcal{T}\left( t\right) $ on $%
L^{\infty }(\Omega )$ associated with strong solutions of (\ref{1.1})-(\ref%
{1.3}) possesses an exponential attractor $\mathcal{Y}_{\Theta ,\mu }$ in
the following sense:

\begin{itemize}
\item $\mathcal{Y}_{\Theta ,\mu }$ is bounded in $L^{\infty }\left( \Omega
\right) \cap D\left( A_{\Theta ,\mu }\right) $ and compact in $L^{2}\left(
\Omega \right) $;

\item $\mathcal{Y}_{\Theta ,\mu }$ is semi-invariant: $\mathcal{T}\left(
t\right) \left( \mathcal{Y}_{\Theta ,\mu }\right) \subset \mathcal{Y}%
_{\Theta ,\mu }$, $t\geq 0$;

\item $\mathcal{Y}_{\Theta ,\mu }$ attracts the images of bounded (in $%
L^{\infty }\left( \Omega \right) $) subsets exponentially in the metric of $%
L^{\infty }\left( \Omega \right) $, i.e. there exist $\beta >0$ and a
monotone function $Q$ such that, for every bounded set $B\subset L^{\infty
}\left( \Omega \right) $,%
\begin{equation*}
dist_{L^{\infty }\left( \Omega \right) }\left( \mathcal{T}\left( t\right) B,%
\mathcal{Y}_{\Theta ,\mu }\right) \leq Q\left( \left\Vert B\right\Vert
_{L^{\infty }\left( \Omega \right) }\right) e^{-\beta t},\text{ for all }%
t\geq 0\text{;}
\end{equation*}

\item $\mathcal{Y}_{\Theta ,\mu }$ has the finite fractal dimension in $%
L^{\infty }\left( \Omega \right) .$
\end{itemize}
\end{theorem}

\begin{proof}
It is also not difficult to complete the proof of the theorem, using the
abstract scheme of Proposition \ref{abstract}. Indeed, by estimate (\ref{3.3}%
), the semigroup $\mathcal{T}\left( t\right) $ admits a bounded absorbing
set in $L^{2}\left( \Omega \right) ,$ which combined with the estimate (\ref%
{sup-b}) of Proposition \ref{bounded}, gives a bounded absorbing set in $%
L^{\infty }\left( \Omega \right) $. Arguing now as in the proof of Lemma \ref%
{h1est}, by using Galerkin truncations of the strong solution, we can easily
deduce that the same ball $\mathcal{B}_{0}$ is an absorbing set for $%
\mathcal{T}\left( t\right) $ in $L^{\infty }\left( \Omega \right) .$ Hence,
the above scheme applies and we can once again obtain an exponential
attractor $\mathcal{Y}_{\Theta ,\mu }$ with the stated properties. This
completes the proof.
\end{proof}

\begin{remark}
\begin{enumerate}
\item \emph{All the results in Sections 3-4 also hold if the nonlocal
operator in (\ref{nonlocal-op}) is replaced by a more general one,%
\begin{equation*}
\langle \Theta _{\mu }(u),v\rangle :=\int \int_{\partial \Omega \times
\partial \Omega }K(x-y)(u(x)-u(y))(v(x)-v(y))d\mu _{x}d\mu _{y},
\end{equation*}%
with a nonnegative symmetric kernel $K\left( x\right) =k\left( \left\vert
x\right\vert \right) $ such that $K\in L^{1}\left( \Gamma _{\mu },d\mu
\right) $. This choice renders the nonlocal term dissipative everywhere in
the estimates, in particular, see the proof of Theorem \ref{SG}. Clearly,
the choice $k\left( r\right) =r^{-\left( N-1+2s\right) },$ $s\in \left(
0,1\right) $, gives (\ref{nonlocal-op}). We have restricted our attention to
this type of kernels only for the sake of convenience and simplicity of
presentation. }

\item \emph{All the results of this article also hold with no changes in all
the proofs if $\Omega \subset \mathbb{R}^{N}$ is an arbitrary open connected
set with finite Lebesgue measure. In the definition of all the spaces
involved, one has only to replace $W^{1,p}(\Omega )\cap C(\bar{\Omega})$
with $W^{1,p}(\Omega )\cap C_{c}(\bar{\Omega})$ where $C_{c}(\bar{\Omega})$
denotes the space of continuous functions on $\bar{\Omega}$ with compact
support. We have made the presentation with $\Omega $ bounded only for the
convenience of the reader.}
\end{enumerate}
\end{remark}

\section{Summary}

\label{sum}

In this article, we considered a general (scalar) reaction-diffusion
equation on $N$-dimensional bounded domains $\Omega $\ with non-smooth
boundary $\partial \Omega $, with $N\geq 2$. Our system captures most of the
specific and variants of the diffusion models considered in the Introduction
and analyzed in the literature. Our setting also captures a number of
additional models that have not been specifically identified or analyzed in
the literature. We give a unified analysis of the system using tools in
nonlinear potential analysis and Sobolev function theory on "rough" domains,
and then use them to obtain the sharpest results. In Section \ref{prelim},
we established our notations and gave some basic preliminary results for the
operators and spaces appearing in our framework. In Section \ref{wel}, we
built some well-posedness results for our nonlinear diffusion model, which
included existence results (Sections 3.1, 3.2), regularity results (Section
3.3), and uniqueness and stability results (Section 3.2). In Section \ref{gl}%
, we showed the existence of a finite-dimensional global attractor and gave
some further properties, then we also established the existence of an
exponential attractor. In addition to establishing a number of technical
results for our model in this general setting, the framework we developed
can recover most of the existing existence, regularity, uniqueness,
stability, attractor existence and dimension results for the well-known
reaction-diffusion equation on smooth domains.

The present unified analysis can be exploited to extend and establish
existence, regularity and existence of finite dimensional attractor results
for other important models based on (weakly)\ damped wave equations, systems
of reaction-diffusion equations for a vector $\overrightarrow{u}=\left(
u_{1},...,u_{k}\right) $ ($k\geq 2$), parabolic problems with degenerate
diffusion, and many others. For instance, our framework requires only minor
modifications to include reaction-diffusion systems for the vector-valued
function $\overrightarrow{u}$; the function spaces become product spaces,
and the principal dissipation and "smoothing" operators become block
operators on these product spaces, typically with block diagonal form. The
nonlinearities in these models can be treated in a similar way as in Section %
\ref{wel} (see \cite[Chapter II, Section 4]{CV}). Furthermore, we remark
that one can also easily allow for time-dependent external forces $h\left(
t\right) ,$ $h\in C_{b}^{1}\left( \mathbb{R};L^{2}\left( \Omega \right)
\right) ,$\ acting on the right-hand side of Eqn. (\ref{1.1}). Indeed, the
existence results still hold and one can generalize the notion of global
attractor and replace it by the notion of pullback attractor, for example.
One can still study the set of all complete bounded trajectories, that is,
trajectories which are bounded for all $t\in \mathbb{R}_+.$ All the results
that we have presented in this paper are still true in that case. We will
consider such questions in forthcoming contributions.

\section{Appendix}

To make the paper reasonably self-contained, we include supporting material
on regularity results for the operator $\left\{ e^{-tA_{\Theta ,\mu
}}\right\} _{t\geq 0}$ and some basic results from monotone operator theory
(see, e.g., \cite{Scho}). Let us recall the following result \cite[Corollary
5.3]{W3}.

\begin{theorem}
\label{SGbis} Let $\Omega \subset {\mathbb{R}}^{N}$ be an arbitrary open set
with finite measure and assume that $\mu $ satisfies (H$_{\mu }$). Then the
subgradient $-\partial \varphi _{\Theta ,\mu }(=-A_{\Theta ,\mu })$
generates a strongly continuous (linear) semigroup $\left\{ e^{-tA_{\Theta
,\mu }}\right\} _{t\geq 0}$ of contractions on $L^{2}(\Omega )$. In
particular, for every $u_{0}\in L^{2}(\Omega )$, the orbit $%
u(t)=e^{-tA_{\Theta ,\mu }}u_{0}$ is the unique strong solution of the first
order Cauchy problem%
\begin{equation}
\partial _{t}u\left( t,x\right) +\partial \varphi _{\Theta _{\mu }}\left(
u\left( t,x\right) \right) =0,\;\;\left( t,x\right) \in \mathbb{R}_{+}\times
\Omega ,\;\;u_{\mid t=0}=u_{0}\text{ in }\Omega ,  \label{C-pro}
\end{equation}%
such that%
\begin{equation*}
u\in C([0,\infty );L^{2}(\Omega ))\cap W_{\text{loc}}^{1,\infty }((0,\infty
);L^{2}(\Omega ))\;\mbox{ and }\;u(t,\cdot )\in D(A_{\Theta ,\mu })\text{,
a.e. on }\left( 0,\infty \right) .
\end{equation*}%
Moreover, the (linear) semigroup $\left\{ e^{-tA_{\Theta ,\mu }}\right\}
_{t\geq 0}$ is non-expansive on $L^{\infty }(\Omega )$ in the sense that%
\begin{equation}
\Vert e^{-tA_{\Theta ,\mu }}u_{0}\Vert _{\infty ,\Omega }\leq \left\Vert
u_{0}\right\Vert _{\infty ,\Omega },  \label{sup-nonexp}
\end{equation}%
for every $t\geq 0$ and $\,u_{0}\in L^{\infty }(\Omega ).$
\end{theorem}

We will now state some results for the non-homogeneous Cauchy problem
associated with (\ref{C-pro}). The first one is taken from \cite[Chapter IV,
Theorem 4.3]{Scho}.

\begin{theorem}
\label{m1}Let $\varphi :H\rightarrow (-\infty,+\infty] $ be a proper,
convex, and lower-semicontinuous functional on the Hilbert space $H$ and set
$A=\partial \varphi .$ Let $u$ be the generalized solution of%
\begin{equation}
\left\{
\begin{array}{ll}
u^{\prime}(t)+A\left( u\right) \ni g\left( t\right) , & t\in \left[ 0,T%
\right] , \\
u_{\mid t=0}=u_{0}\text{,} &
\end{array}%
\right.  \label{nonhom}
\end{equation}%
with $g\in L^{2}\left( 0,T;H\right) $ and $u_{0}\in \overline{D\left(
A\right) }.$ Then $\varphi \left( u\right) \in L^{1}\left( 0,T\right) ,$ $%
\sqrt{t}u^{\prime}(t)\in L^{2}\left( 0,T;H\right) $ and $u\left( t\right)
\in D\left( A\right) $ for a.e. $t\in \left[ 0,T\right] .$
\end{theorem}

The proof of Theorem \ref{SG} (Section 3) is based on the following
regularity result for the evolution problem (\ref{nonhom}) with a generic $m$%
-accretive operator $A$. Since we could not find a proof for it in the
literature, we choose to include here for the convenience of the reader. The
theorem is a more general version of \cite[Chapter IV, Proposition 3.2]{Scho}
with some modified assumptions.

\begin{theorem}
\label{ap_reg_thm}Let the assumptions of Theorem \ref{m1} be satisfied.
Assume that $A=\partial \varphi $ is \emph{strongly} accretive in $H$, that
is, $A-\omega I$ is accretive for some $\omega >0$ and, in addition,%
\begin{equation*}
g\in L^{\infty }\left( [\delta ,\infty );H\right) \cap W^{1,2}\left( [\delta
,\infty );H\right) ,
\end{equation*}%
for every $\delta >0$. Let $u$ be the generalized solution of (\ref{nonhom})
for $u_{0}\in \overline{D\left( A\right) }$. It follows that $u\left(
t\right) \in D\left( A\right) ,$ $t\geq \delta >0$, and the following
estimate holds for some $v\in D\left( A\right) $:%
\begin{equation}
\left\Vert A^{0}u\left( t\right) \right\Vert _{H}\leq C\left( \left\Vert
A^{0}v\right\Vert _{H}+\frac{1}{\delta }\left\Vert u_{0}-v\right\Vert
_{H}+\left\Vert g\right\Vert _{L^{\infty }\left( [\delta ,\infty );H\right)
}+\left\Vert g^{^{\prime }}\right\Vert _{L^{2}\left( [\delta ,\infty
);H\right) }\right) ,  \label{append1}
\end{equation}%
for each $u_{0}\in \overline{D\left( A\right) }$ and $t\geq \delta >0.$ The
constant $C>0$ is independent of $t,\delta ,u,v,u_{0}$ and $g.$ Here $A^{0}$
denotes the minimal section of $A$.
\end{theorem}

\begin{proof}
We proceed as follows. We first regularize the solution $u$ by a sequence of
approximate solutions $u_{\alpha }$ with $\alpha >0$, in which $u_{\alpha
}\in C^{1}\left( \left[ 0,T\right] ;H\right) $ such that $u_{\alpha }\left(
0\right) =u_{0}\in \overline{D\left( A\right) }$, and $u_{\alpha }$ solves
the following regularized problem%
\begin{equation}
u_{\alpha }^{\prime }\left( t\right) +A_{\alpha }\left( u_{\alpha }\left(
t\right) \right) =g_{\alpha }\left( t\right) ,\text{ }t\in \left[ 0,T\right]
.  \label{reg_app}
\end{equation}%
Here $A_{\alpha }$ corresponds to the Yosida approximation of $A$, and $%
g_{\alpha }\in C^{1}\left( \left[ 0,T\right] ;H\right) $ is a sequence of
approximate functions such that, as $\alpha \rightarrow 0^{+},$%
\begin{equation}
g_{\alpha }\rightarrow g\text{ in }L^{\infty }\left( 0,T;H\right) \cap
W^{1,2}\left( 0,T;H\right) .  \label{reg_app2}
\end{equation}%
Note that by \cite[Chapter IV, Proposition 1.8]{Scho} we have $A_{\alpha
}=\varphi _{\alpha }^{^{\prime }},$ $\alpha >0,$ i.e., $A_{\alpha }$ is
Lipschitz continuous. Hence by the standard Cauchy-Lipschitz theorem, there
exists at least one solution $u_{\alpha }\in C^{1}\left( \left[ 0,T\right]
;H\right) $ to problem (\ref{reg_app}).

\emph{Step 1.} If $h>0$, then $u_{\alpha }\left( t+h\right) $ is a solution
of (\ref{reg_app}) with $g_{\alpha }\left( t\right) $ replaced by $g_{\alpha
}\left( t+h\right) $, and the accretive estimate on $A_{\alpha }$ (i.e., $%
\left( A_{\alpha }w_{1}-A_{\alpha }w_{2},w_{1}-w_{2}\right) _{H}\geq \omega
\left\Vert w_{1}-w_{2}\right\Vert _{H}^{2},$ $\omega >0$) gives%
\begin{align*}
& \frac{1}{2}\frac{d}{dt}\left\Vert u_{\alpha }\left( t+h\right) -u_{\alpha
}\left( t\right) \right\Vert _{H}^{2}+\omega \left\Vert u_{\alpha }\left(
t+h\right) -u_{\alpha }\left( t\right) \right\Vert _{H}^{2} \\
& \leq \left\Vert g_{\alpha }\left( t+h\right) -g_{\alpha }\left( t\right)
\right\Vert _{H}\left\Vert u_{\alpha }\left( t+h\right) -u_{\alpha }\left(
t\right) \right\Vert _{H} \\
& \leq \frac{1}{2}\left( \omega \left\Vert u_{\alpha }\left( t+h\right)
-u_{\alpha }\left( t\right) \right\Vert _{H}^{2}+\frac{1}{\omega }\left\Vert
g_{\alpha }\left( t+h\right) -g_{\alpha }\left( t\right) \right\Vert
_{H}^{2}\right) .
\end{align*}%
Thus, for $0\leq s\leq t$ we get, on account of Gronwall's inequality (see,
e.g., \cite[Chapter IV, Lemma 4.1]{Scho}),%
\begin{equation*}
\left\Vert u_{\alpha }\left( t+h\right) -u_{\alpha }\left( t\right)
\right\Vert _{H}^{2}e^{\omega t}\leq \left\Vert u_{\alpha }\left( s+h\right)
-u_{\alpha }\left( s\right) \right\Vert _{H}^{2}e^{\omega s}+\omega
^{-1}\int_{s}^{t}e^{\omega \tau }\left\Vert g_{\alpha }\left( \tau +h\right)
-g_{\alpha }\left( \tau \right) \right\Vert _{H}^{2}d\tau .
\end{equation*}%
Dividing both sides of this inequality by $h^{2}$, then letting $%
h\rightarrow 0^{+}$, we obtain after standard transformations,%
\begin{align}
\left\Vert u_{\alpha }^{^{\prime }}\left( t\right) \right\Vert _{H}^{2}&
\leq \left\Vert u_{\alpha }^{^{\prime }}\left( s\right) \right\Vert
_{H}^{2}e^{-\omega \left( t-s\right) }+\omega ^{-1}\int_{s}^{t}e^{-\omega
\left( t-\tau \right) }\left\Vert g_{\alpha }^{^{\prime }}\left( \tau
\right) \right\Vert _{H}^{2}d\tau  \label{esta10} \\
& \leq \left\Vert u_{\alpha }^{^{\prime }}\left( s\right) \right\Vert
_{H}^{2}+\omega ^{-1}\int_{s}^{t}\left\Vert g_{\alpha }^{^{\prime }}\left(
\tau \right) \right\Vert _{H}^{2}d\tau ,  \notag
\end{align}%
for all $t\geq s\geq 0.$

\emph{Step 2}. Fix now $v\in H$ and $\alpha >0$, and define%
\begin{equation*}
\psi \left( w\right) =\varphi _{\alpha }\left( w\right) -\varphi_\alpha
\left( v\right) -\left( A_{\alpha }v,w-v\right) _{H},\text{ }w\in H.
\end{equation*}%
Note that $\psi ^{^{\prime }}\left( w\right) =\varphi _{\alpha }^{^{\prime
}}\left( w\right) -A_{\alpha }v$, $w\in H,$ $\min_{w\in H}\psi \left(
w\right) =\psi \left( v\right) =0$, and
\begin{equation*}
u_{\alpha }^{\prime}\left( t\right) +\psi ^{^{\prime }}\left( u_{\alpha
}\left( t\right) \right) =-A_{\alpha }v+g_{\alpha }\left( t\right) ,\text{ }%
t\in \left[ 0,T\right] .
\end{equation*}%
Since $\left( \psi ^{^{\prime }}\left( u_{\alpha }\right) ,v-u_{\alpha
}\right) _{H}\leq \psi \left( v\right) -\psi \left( u_{\alpha }\right)
=-\psi \left( u_{\alpha }\right) $, it follows that%
\begin{align*}
\psi \left( u_{\alpha }\left( t\right) \right) & \leq \left( u_{\alpha
}^{^{\prime }}\left( t\right) +A_{\alpha }v-g_{\alpha }\left( t\right)
,v-u_{\alpha }\left( t\right) \right) _{H} \\
& =-\frac{1}{2}\frac{d}{dt}\left\Vert v-u_{\alpha }\left( t\right)
\right\Vert _{H}^{2}+\left( A_{\alpha }v,v-u_{\alpha }\left( t\right)
\right) _{H} -\left( g_{\alpha }\left( t\right) ,v-u_{\alpha }\left(
t\right) \right) _{H}.
\end{align*}%
Hence, by integration over $\left( 0,T\right) $ we deduce%
\begin{align}
& \int_{0}^{T}\left( \psi \left( u_{\alpha }\left( t\right) \right) +\left(
A_{\alpha }v,u_{\alpha }\left( t\right) -v\right) _{H}+\left( g_{\alpha
}\left( t\right) ,v-u_{\alpha }\left( t\right) \right) _{H}\right) dt
\label{esta11} \\
& \leq \frac{1}{2}\left( \left\Vert v-u_{0}\right\Vert _{H}^{2}-\left\Vert
v-u_{\alpha }\left( T\right) \right\Vert _{H}^{2}\right) .  \notag
\end{align}%
This is the energy estimate. In order to estimate the derivative of $%
u_{\alpha }$, we take the scalar product in $H$ of equation (\ref{reg_app})
with $u_{\alpha }^{^{\prime }}$, multiply the resulting identity by $t\geq 0$%
, to deduce%
\begin{align}
& t\left\Vert u_{\alpha }^{^{\prime }}\left( t\right) \right\Vert _{H}^{2}+%
\frac{d}{dt}\left[ t\left( \psi \left( u_{\alpha }\left( t\right) \right)
+\left( A_{\alpha }v,u_{\alpha }\left( t\right) -v\right) _{H}+\left(
g_{\alpha }\left( t\right) ,v-u_{\alpha }\left( t\right) \right) _{H}\right) %
\right]  \label{esta12} \\
& =\psi \left( u_{\alpha }\left( t\right) \right) +\left( A_{\alpha
}v,u_{\alpha }\left( t\right) -v\right) _{H}+\left( g_{\alpha }\left(
t\right) ,v-u_{\alpha }\left( t\right) \right) _{H}.  \notag
\end{align}%
Integrating over $\left( 0,T\right) $ and combining with the energy estimate
(\ref{esta11}), on account of Young's inequality, we obtain%
\begin{equation}
\int_{0}^{T}t\left\Vert u_{\alpha }^{^{\prime }}\left( t\right) \right\Vert
_{H}^{2}dt\leq \frac{1}{2}\left( \left\Vert v-u_{0}\right\Vert
_{H}^{2}+T^{2}\left\Vert A_{\alpha }v\right\Vert _{H}^{2}+T^{2}\left\Vert
g_{\alpha }\left( T\right) \right\Vert _{H}^{2}\right) .  \label{esta13}
\end{equation}%
The $H$-norm of $u_{\alpha }^{^{\prime }}\left( t\right) $ satisfies (\ref%
{esta10}) with $t=T$ and $s=t$, so the left-hand side of (\ref{esta13})
dominates%
\begin{equation}
\frac{T^{2}}{2}\left\Vert u_{\alpha }^{^{\prime }}\left( T\right)
\right\Vert _{H}^{2}-\omega ^{-1}\int_{0}^{T}t\int_{t}^{T}\left\Vert
g_{\alpha }^{^{\prime }}\left( \tau \right) \right\Vert _{H}^{2}d\tau dt.
\label{esta14}
\end{equation}%
We obtain from (\ref{esta13}) and (\ref{esta14}), for each $0<\delta \leq
t\leq T$,%
\begin{equation*}
\frac{T^{2}}{2}\left\Vert u_{\alpha }^{^{\prime }}\left( T\right)
\right\Vert _{H}^{2}\leq \frac{1}{2}\left( \left\Vert v-u_{0}\right\Vert
_{H}^{2}+T^{2}\left\Vert A_{\alpha }v\right\Vert _{H}^{2}+T^{2}\left\Vert
g_{\alpha }\left( T\right) \right\Vert _{H}^{2}+\omega ^{-1}\frac{T^{2}}{2}%
\left\Vert g_{\alpha }^{^{\prime }}\right\Vert _{L^{2}\left( \delta
,T;H\right) }^{2}\right) ,
\end{equation*}%
which coupled together with the basic inequality $\left( a-b\right) ^{2}\geq
\frac{a^{2}}{2}-b^{2}$, and equation (\ref{reg_app}), yield%
\begin{equation}
\frac{T^{2}}{4}\left\Vert A_{\alpha }\left( u_{\alpha }\left( t\right)
\right) \right\Vert _{H}^{2}\leq \frac{1}{2}\left( \left\Vert
v-u_{0}\right\Vert _{H}^{2}+T^{2}\left\Vert A_{\alpha }v\right\Vert
_{H}^{2}+2T^{2}\left\Vert g_{\alpha }\left( T\right) \right\Vert
_{H}^{2}+\omega ^{-1}\frac{T^{2}}{2}\left\Vert g_{\alpha }^{^{\prime
}}\right\Vert _{L^{2}\left( \delta ,T;H\right) }^{2}\right) .  \label{esta15}
\end{equation}%
Finally, by virtue of (\ref{reg_app2}), we deduce the following bound for $%
A_{\alpha }\left( u_{\alpha }\right) $ in $L^{\infty }\left( \delta
,T;H\right) $:%
\begin{equation}
\left\Vert A_{\alpha }\left( u_{\alpha }\left( T\right) \right) \right\Vert
_{H}^{2}\leq \frac{2}{T^{2}}\left\Vert v-u_{0}\right\Vert
_{H}^{2}+2\left\Vert A^{0}v\right\Vert _{H}^{2}+4\left\Vert g\left( T\right)
\right\Vert _{H}^{2}+2\omega ^{-1}\left\Vert g^{^{\prime }}\right\Vert
_{L^{2}\left( \delta ,T;H\right) }^{2},  \label{esta16}
\end{equation}%
for every $\alpha >0$ and $T\geq \delta >0.$ Arguing now in a standard way
(see the proof of \cite[Chapter IV, Proposition 3.1]{Scho}), we can pass to
the limit as $\alpha \rightarrow 0$ in (\ref{esta16}), along a subsequence $%
\alpha _{n}\rightarrow 0$, $A_{\alpha _{n}}\left( u_{\alpha _{n}}\left(
T\right) \right) \rightarrow y\in A\left( u\left( T\right) \right) $, so
that (\ref{append1}) follows from (\ref{esta16}) by the weak
lower-semicontinuity of the norm. This completes the proof of the theorem.
\end{proof}

Finally, the following corollary follows as a consequence.

\begin{corollary}
\label{ap_reg_co}Let the assumptions of Theorem \ref{ap_reg_thm} be
satisfied. For each $u_{0}\in \overline{D\left( A\right) }$, there exists a
unique solution $u$ of (\ref{nonhom}) such that%
\begin{equation*}
u\in L^{\infty }\left( [\delta ,\infty );D\left( A\right) \right) \cap
W^{1,\infty }\left( [\delta ,\infty );H\right) ,
\end{equation*}%
for every $\delta >0.$
\end{corollary}


\begin{thebibliography}{99}
\bibitem{ADT} Y. Achdou, T. Deheuvels and N. Tchou. \emph{JLip versus
Sobolev spaces on a class of self-similar fractal foliages}.
\newblock{J.
Math. Pures Appl.} \textbf{(9) 97} (2012), 142--172.

\bibitem{AT} Y. Achdou, C. Sabot and N. Tchou. \emph{Diffusion and
propagation problems in some ramified domains with a fractal boundary}. %
\newblock Math. Model. Numer. Anal. \textbf{40} (2006), 623--652.

\bibitem{ATbis} Y. Achdou, C. Sabot and N. Tchou. \emph{A multiscale
numerical method for Poisson problems in some ramified domains with a
fractal boundary}. \newblock Multiscale Model. Simul. \textbf{5} (2006),
828--860.

\bibitem{ATtris} Y. Achdou and N. Tchou. \emph{Boundary value problems in
ramified domains with fractal boundaries}. \newblock Lecture Notes in
Comput. Sci. Eng., vol. \textbf{60} (2008), 419--426.

\bibitem{AT08} Y. Achdou and N. Tchou. \emph{Trace results on domains with
self-similar fractal boundaries}. \newblock J. Math. Pures Appl. \textbf{(9)
89} (2008), 596--623.

\bibitem{AT4} Y. Achdou and N. Tchou. \emph{Trace theorems for a class of
ramified domains with self-similar fractal boundaries}. \newblock SIAM J.
Math. Anal. \textbf{42} (2010), 1449--1482.

\bibitem{Ali} N.D. Alikakos. \emph{$L^{p}$-bounds of solutions to
reaction-diffusion equations}. \newblock Comm. Partial Differential
Equations \textbf{4} (1979), 827--868.

\bibitem{AW1} W.~Arendt and M.~Warma. \emph{The Laplacian with Robin
boundary conditions on arbitrary domains}. \newblock{Potential Analysis}
\textbf{19} (2003), 341--363.

\bibitem{AW2} W.~Arendt and M.~Warma. \emph{Dirichlet and {N}eumann boundary
conditions: {W}hat is in between?} \newblock{J. Evol. Equ.} \textbf{3}
(2003), 119--135.

\bibitem{BV} A.~V.~Babin and M.~I.~Vishik. \emph{Attractors of Evolutions
Equations}. \newblock North-Holland, Amsterdam, 1992.

\bibitem{B} M. Biegert. \emph{The relative capacity}. \newblock %
arXiv:0806.1417.

\bibitem{BW1} M.~Biegert and M.~Warma. \emph{The heat equation with
nonlinear generalized Robin boundary conditions}.
\newblock{J. Differential
Equation} \textbf{247} (2009), 1949--1979.

\bibitem{BW2} M. Biegert and M. Warma. \emph{Some quasi-linear elliptic
equations with inhomogeneous generalized Robin boundary conditions on
"bad\textquotedblright\ domains}. \newblock{Adv. Differential Equations}
\textbf{15} (2010), 893--924.

\bibitem{Ca} J. R. Cannon and G. H. Meyer. \emph{On a diffusion in a
fractured medium}. \newblock SIAM J. Appl. Math. \textbf{3} (1971), 434--448.

\bibitem{CK} Z-Q.~Chen and T.~Kumagai. \emph{Heat kernel estimates for
stable-like processes on $d$-sets}. \newblock{\em Stochastic Process. Appl}.
\textbf{108} (2003), 27--62.

\bibitem{CV} V.V. Chepyzhov and M.I. Vishik. \emph{Attractors for Equations
of Mathematical Physics}. \newblock Amer. Math. Soc., Providence, RI, 2002.

\bibitem{CD} J. W. Cholewa and T. Dlotko. \emph{Global Attractors in
Abstract Parabolic Problems}. \newblock Cambridge University Press, 2000.

\bibitem{D1} \newblock D.~Daners. \emph{Robin boundary value problems on
arbitrary domains}. \newblock {Trans. Amer. Math. Soc.} \textbf{352} (2000),
4207--4236.

\bibitem{D2} D. Daners. \emph{A priori estimates for solutions to elliptic
equations on non-smooth domains}. \newblock Proc. Roy. Soc. Edinburgh Sect.
A \textbf{132} (2002), 793--813.

\bibitem{DD} D.~Daners and P.~Dr\'{a}bek. \emph{A priori estimates for a
class of quasi-linear elliptic equations}. \newblock{Trans. Amer. Math. Soc.}
\textbf{361} (2009), 6475--6500.

\bibitem{DGN} D.~Danielli, N.~Garofalo and D-H.~Nhieu. \emph{Non-doubling
Ahlfors Measures, Perimeter Measures, and the Characterization of the Trace
Spaces of Sobolev Functions in Carnot-Carath\'{e}odory Spaces}.
\newblock{Mem.
Amer. Math. Soc}. \textbf{182} (2006).

\bibitem{DL} R. Dautray and J.L. Lions. \emph{Mathematical Analysis and
Numerical Methods for Science and Technology. vol. 2}. \newblock %
Springer-Verlag, Berlin, 1988.

\bibitem{Dav} E.~B.~Davies. \emph{Heat Kernel and Spectral Theory}. %
\newblock{Cambridge University Press, Cambridge,} 1989.


\bibitem{EZ} M. Efendiev and S. Zelik. \emph{Finite-dimensional attractors
and exponential attractors for degenerate doubly nonlinear equations}. %
\newblock Math. Methods Appl. Sci. \textbf{32} (2009), 1638--1668.

\bibitem{EG} L.~C.~Evans and R.~F.~Gariepy.
\newblock{\em Measure Theory and
Fine Properties of Functions}. \newblock{CRC Press, Boca Raton, FL}, 1992.

\bibitem{Fal} K. Falconer. \emph{Fractal Geometry}. \newblock Mathematical
Foundations and Applications. Second edition. John Wiley \& Sons, Inc., 2003.

\bibitem{Fa} K. Falconer and J. Hu. \emph{Nonlinear diffusion equations on
unbounded fractal domains}. \newblock J. Math. Anal. Appl. \textbf{256}
(2001), 606--624.

\bibitem{FT1} M.~Fukushima and M.~Tomisaki.
\newblock{Reflecting diffusions on Lipschitz domains with cusps: analytic construction and Skorohod representation.
Potential theory and degenerate partial differential operators (Parma)}. %
\newblock{Potential Anal}. \textbf{4} (1995), 377--408.

\bibitem{FT2} M.~Fukushima and M.~Tomisaki.
\newblock{Construction and
decomposition of reflecting diffusions on Lipschitz domains with H\"older
cusps}. \newblock Probab. Theory Related Fields \textbf{106} (1996),
521--557.

\bibitem{FW} Feng-yu Wang. \emph{Functional inequalities, semigroup
properties and spectrum estimates}. Infin. Dimens. Anal. Quantum. Probab.
Relat. Top. \textbf{3} (2000), 263-295.

\bibitem{G} C.G. Gal. \emph{On a class of degenerate parabolic equations
with dynamic boundary conditions.} \newblock{ J. Differential Equations}
\textbf{253} (2012), 126--166.

\bibitem{GalNS} C.G.~Gal. \emph{Sharp estimates for the global attractor of
scalar reaction-diffusion equations with a Wentzell boundary condition}. %
\newblock J. Nonlinear Science \textbf{22} (2012), 85--106.

\bibitem{GM} F.~Gesztesy and M.~Mitrea. \emph{Nonlocal Robin Laplacians and
some remarks on a paper by Filonov on eigenvalue inequalities}.
\newblock{J.
Differential Equations} \textbf{247} (2009), 2871--2896.

\bibitem{Kos} P.~Haj\l asz, P.~Koskela and H.~Tuominen. \emph{Sobolev
embeddings, extensions and measure density condition}.
\newblock{J. Funct.
Anal.} \textbf{254} (2008), 1217--1234.

\bibitem{Hu} J. Hu. \emph{Nonlinear diffusion equations on bounded fractal
domains}. \newblock Z. Anal. Anwendungen \textbf{20} (2001), 331--345.

\bibitem{HP} P. H. Hung and E. Sanchez-Palencia. \emph{Phenom\`enes de
transmission \'a travers des couches minces de conductivit\'e \'elev\'ee}. %
\newblock J. Math. Anal. Appl. \textbf{47} (1974), 284--309.

\bibitem{H81} J.E. Hutchinson. \emph{Fractals and self-similarity}.
\newblock
Indiana Univ. Math. J. \textbf{30} (1981), 713--747.

\bibitem{JW} A.~Jonsson and H.~Wallin.
\newblock{\em Function Spaces on
Subsets of $\RR^n$}.
\newblock{Math. Rep. Vol. 2 Part I, Academic
Publishers, Harwood}, 1984.

\bibitem{La1} M. R. Lancia. \emph{A transmission problem with a fractal
interface}. \newblock Z. Anal. Anwendungen \textbf{21} (2002), 113--133.

\bibitem{La2} M. R. Lancia. \emph{Second order transmission problems across
a fractal surface}. \newblock Rend. Accad. Naz. Sci. XL Mem. Mat. Appl.
\textbf{5} (2003), 191--213.

\bibitem{La3} M. R. Lancia and P. Vernole. \emph{Semilinear evolution
transmission problems across fractal layers}. \newblock Nonlinear Anal.
\textbf{75} (2012), 4222--4240.

\bibitem{La4} M. R. Lancia and P. Vernole. \emph{Irregular heat flow problems%
}. \newblock SIAM J. Math. Anal. \textbf{42} (2010), 1539--1567.

\bibitem{MZ} J.~Mal\'{y} and W.~P.~Ziemer. \emph{Fine Regularity of
Solutions of Elliptic Partial Differential Equations}.
\newblock{Amer. Math.
Soc., Mathematical Surveys and Monographs, Vol. 51, 1997}.

\bibitem{Ma} B.B. Mandelbrodt. \emph{The Fractal Geometry of Nature}. %
\newblock Freeman \&Co, 1982.

\bibitem{MF} B.B. Mandelbrot and M. Frame. \emph{The canopy and shortest
path in a self-contacting fractal tree}. \newblock Math. Intelligencer
\textbf{21} (1999), 18--27.

\bibitem{MS1} B. Mauroy, M. Filoche, J.S. Andrade and B. Sapoval. \emph{%
Interplay between flow distribution and geometry in an airway tree}. %
\newblock Phys. rev. Lett. \textbf{90} (2003), 1--4.

\bibitem{MS2} B. Mauroy, M. Filoche, E.R. Weibel and B. Sapoval. \emph{The
optimal bronchial tree is dangerous}. \newblock Nature \textbf{90}, 2004.

\bibitem{Maz85} V.G.~Maz'ya. \emph{Sobolev Spaces}. %
\newblock{Springer-Verlag, Berlin}, 1985.

\bibitem{MP} V.G.~Maz'ya and S.V.~Poborchi. \emph{Differentiable Functions
on Bad Domains}. \newblock{World Scientific Publishing}, 1997.

\bibitem{MZrev} A. Miranville and S. Zelik. \emph{Attractors for dissipative
partial differential equations in bounded and unbounded domains. Handbook of
differential equations: evolutionary equations}. \newblock Vol. IV,
103--200, Handb. Differ. Equ., Elsevier/North-Holland, Amsterdam, 2008.

\bibitem{R} J.C. Robinson. \emph{Infinite-Dimensional Dynamical Systems. An
Introduction to Dissipative Parabolic PDEs and the Theory of Global
Attractors}. \newblock Cambridge Texts in Applied Mathematics. Cambridge
University Press, Cambridge, 2001.

\bibitem{Scho} R.~E.~Showalter. \emph{Monotone Operators in Banach Space and
Nonlinear Partial Differential Equations}.
\newblock{Amer. Math. Soc.,
Providence, RI}, 1997.

\bibitem{T} R.~Temam. \emph{Infinite-Dimensional Dynamical Systems in
Mechanics and Physics}. \newblock Springer-Verlag, New York, 1997.

\bibitem{VW} A. Velez-Santiago and M. Warma. \emph{A class of quasi-linear
parabolic and elliptic equations with nonlocal Robin boundary conditions}. %
\newblock{J. Math. Anal. Appl}. \textbf{372} (2010), 120--139.

\bibitem{Wal} H.~Wallin. \emph{The trace to the boundary of Sobolev spaces
on a snowflake}. \newblock{\em Manuscripta Math}. \textbf{73} (1991),
117--125.

\bibitem{War} M. Warma. \emph{The Laplacian with general Robin boundary
conditions}. \newblock{Ph.D Dissertation, University of Ulm}, 2002.

\bibitem{W} M. Warma. \emph{The Robin and Wentzell-Robin Laplacians on
Lipschitz domains}. \newblock{Semigroup Forum} {\bf 73} (2006), 10--30.

\bibitem{W2} M. Warma. \emph{Regularity and well-posedness of some
quasi-linear elliptic and parabolic problems with nonlinear general Wentzell
boundary conditions on nonsmooth domains}. \newblock Nonlinear Anal. \textbf{%
75} (2012), 5561--5588.

\bibitem{W3} M. Warma. \emph{The p-Laplace operator with the nonlocal Robin
boundary conditions on arbitrary open sets}. \newblock Annali di Matematica
Pura ed Applicata (2012), 1--33. DOI: 10.1007/s10231-012-0273-y.
\end{thebibliography}
\end{document}